\documentclass[11pt,reqno]{amsart}
\usepackage[english]{babel}
\usepackage[a4paper,twoside]{geometry}
\usepackage{microtype}
\usepackage{enumerate}
\usepackage{csquotes}
\usepackage{longtable}
\usepackage{bbm}

\usepackage{xspace}
\patchcmd{\subsection}{-.5em}{.5em}{}{}
\usepackage[hyperfootnotes=false]{hyperref} 
\usepackage{color}
\hypersetup{
	colorlinks = true,
	linkcolor = {blue},
	urlcolor = {red},
	citecolor = {blue}
}

\makeatletter
\renewcommand{\tocsection}[3]{
  \indentlabel{\@ifnotempty{#2}{\ignorespaces#1 #2\quad}}\bfseries#3}
\renewcommand{\tocsubsection}[3]{
  \indentlabel{\@ifnotempty{#2}{\ignorespaces#1 #2\quad}}#3}

\newcommand\@dotsep{4.5}
\def\@tocline#1#2#3#4#5#6#7{\relax
  \ifnum #1>\c@tocdepth
  \else
    \par \addpenalty\@secpenalty\addvspace{#2}
    \begingroup \hyphenpenalty\@M
    \@ifempty{#4}{
      \@tempdima\csname r@tocindent\number#1\endcsname\relax
    }{
      \@tempdima#4\relax
    }
    \parindent\z@ \leftskip#3\relax \advance\leftskip\@tempdima\relax
    \rightskip\@pnumwidth plus1em \parfillskip-\@pnumwidth
    #5\leavevmode\hskip-\@tempdima{#6}\nobreak
    \leaders\hbox{$\m@th\mkern \@dotsep mu\hbox{.}\mkern \@dotsep mu$}\hfill
    \nobreak
    \hbox to\@pnumwidth{\@tocpagenum{\ifnum#1=1\bfseries\fi#7}}\par
    \nobreak
    \endgroup
  \fi}
\AtBeginDocument{
\expandafter\renewcommand\csname r@tocindent0\endcsname{0pt}
}
\def\l@subsection{\@tocline{2}{0pt}{2.5pc}{5pc}{}}
\makeatother

\usepackage{amsmath}
\usepackage{amsthm}
\usepackage{amssymb}
\usepackage{mathrsfs} 
\usepackage{eucal}
\usepackage{mathtools}

\usepackage{graphicx}
\usepackage{tikz}
\usetikzlibrary{cd}

\newcounter{results}[section]

\theoremstyle{plain}
\newtheorem{theorem}[results]{Theorem}
\newtheorem{lemma}[results]{Lemma}
\newtheorem{proposition}[results]{Proposition}
\newtheorem{corollary}[results]{Corollary}

\theoremstyle{remark}
\newtheorem{remark}[results]{Remark}

\theoremstyle{definition}
\newtheorem{definition}[results]{Definition}

\numberwithin{equation}{section}

\renewcommand{\H}{\mathbb{H}}
\newcommand{\M}{\mathbb{M}}
\newcommand{\N}{\mathbb{N}}
\newcommand{\Q}{\mathbb{Q}}
\newcommand{\R}{\mathbb{R}}

\renewcommand{\AA}{\mathscr{A}}

\newcommand{\FF}{\mathscr{F}}

\newcommand{\LL}{\mathscr{L}}

\newcommand{\cL}{{\ensuremath{\mathcal L}}}

\newcommand{\cN}{{\ensuremath{\mathcal N}}}

\newcommand{\ff}{{\boldsymbol f}}

\newcommand{\ii}{{\mbox{\boldmath$i$}}}

\newcommand{\mm}{{\mbox{\boldmath$m$}}}

\newcommand{\gG}{{\mbox{\boldmath$G$}}}

\newcommand{\vV}{{\mbox{\boldmath$V$}}}
\newcommand{\wW}{{\mbox{\boldmath$W$}}}

\newcommand{\ggamma}{{\mbox{\boldmath$\gamma$}}}

\newcommand{\eeta}{{\mbox{\boldmath$\eta$}}}
\newcommand{\iiota}{{\boldsymbol \iota}}
\newcommand{\jjmath}{{\boldsymbol \jmath}}

\newcommand{\mmu}{{\mbox{\boldmath$\mu$}}}

\newcommand{\pphi}{{\boldsymbol \phi}}

\newcommand{\sfd}{{\sf d}}
\newcommand{\sfe}{{\sf e}}

\newcommand{\sfm}{{\sf m}}

\newcommand{\sfx}{{\sf x}}

\newcommand{\sfD}{{\sf D}}

\newcommand{\sfQ}{{\mathsf Q}}

\newcommand{\rmC}{{\mathrm C}}

\newcommand{\rmB}{{\mathrm B}}
\newcommand{\rmD}{{\mathrm D}}

\newcommand{\rmG}{{\mathrm G}}
\newcommand{\brmG}{{\boldsymbol\rmG}}

\newcommand{\rmT}{{\mathrm T}}

\newcommand{\Kliminf}{K\kern-3pt-\kern-2pt\mathop{\rm lim\,inf}\limits}  
\newcommand{\supp}{\mathop{\rm supp}\nolimits}   
\newcommand{\Lip}{\mathop{\rm Lip}\nolimits}          
\newcommand{\Lipb}{\mathop{\rm Lip}_b\nolimits}          

\newcommand{\lip}{\mathop{\rm lip}\nolimits}          

\renewcommand{\d}{{\mathrm d}}

\newcommand{\restr}[1]{\lower3pt\hbox{$|_{#1}$}}

\newcommand{\Leb}[1]{{\mathscr L}^{#1}}      
\newcommand{\la}{{\langle}}                  
\newcommand{\ra}{{\rangle}}
\newcommand{\down}{\downarrow}              

\newcommand{\eps}{\varepsilon}  
\newcommand{\nchi}{{\raise.3ex\hbox{$\chi$}}}
\newcommand{\weakto}{\rightharpoonup}
\newcommand{\prob}{\mathcal P}

\renewcommand{\mm}{\mathfrak m}
\newcommand{\bmm}{\boldsymbol{\mathfrak m}}

\newcommand{\res}{\mathop{\hbox{\vrule height 7pt width .5pt depth 0pt
\vrule height .5pt width 6pt depth 0pt}}\nolimits}

\newcommand{\J}I
\newcommand{\CE}{\mathsf{C\kern-1pt E}}
\newcommand{\NE}{\mathsf{N\kern-2.5pt E}}
\newcommand{\wCE}{\mathsf{wC\kern-1pt E}}

\newcommand{\pCE}{\mathsf{pC\kern-1pt E}}
\newcommand{\Mod}{\operatorname{Mod}}

\newcommand{\kkappa}{{\boldsymbol \kappa}}

\newcommand{\relgrad}[1]{|\rmD #1|_\star}

\newcommand{\relgradA}[1]{|\rmD #1|_{\star,{\scriptscriptstyle \AA}}}

\newcommand{\cyl}[3]{\mathfrak C^{#1}_{#2}\big ( #3 \big
  )}
\newcommand{\ccyl}[3]{\operatorname{FC}^{#1}_{#2}\!\!\left ( #3 \right )}
\newcommand{\cnorm}[2]{ \left \| \rmD #1 \left [ #2 \right ] \right \|_{#2}}
\newcommand{\uphi}{\pphi}
\newcommand{\interval}{\mathcal{J}}
\DeclareMathOperator{\Tan}{Tan}

\newcommand{\sqm}[1]{\mathsf m_2^2(#1)} 
\newcommand{\rsqm}[1]{\mathsf m_2(#1)} 
\newcommand{\mres}{\mathbin{\vrule height 1.6ex depth 0pt width
0.13ex\vrule height 0.13ex depth 0pt width 1.3ex}}

\newcommand{\dY}{{\delta}}
\newcommand{\W}{\mathbb W}
\newcommand{\Wmms}[2]{\W_2(#1,#2)}

\newcommand{\domG}{\mathcal D}
\newcommand{\Residual}{\mathrm G_0}
\newcommand{\Tangent}{\mathrm T}

\newcommand{\ocR}{\mathrm O_0}
\newcommand{\Span}{\mathrm{span}}
\newcommand{\lin}[1]{\mathsf L_{#1}}
\newcommand{\prbt}{{\prob_2(\R^d)}}
\newcommand{\mgrad}{\boldsymbol {\mathrm D}_\mm}
\newcommand{\diff}{\mathsf{diff}}

\title[Wasserstein Sobolev spaces]{Density of subalgebras of Lipschitz functions\\
  in metric Sobolev spaces and
  \\
  applications to Wasserstein Sobolev spaces}

\author{Massimo Fornasier}
\address{Massimo Fornasier: TUM Fakult\"at f\"ur Mathematik, Boltzmannstrasse 3, 85748 Garching bei M\"unchen (Germany)}
\email{massimo.fornasier@ma.tum.de}

\author{Giuseppe Savar\'e}
\address{Giuseppe Savar\'e: Bocconi University,
  Department of Decision Sciences and BIDSA, Via Roentgen 1, 20136 Milano (Italy)}
\email{giuseppe.savare@unibocconi.it}

\author{Giacomo Enrico Sodini}
\address{Giacomo Enrico Sodini: TUM Fakult\"at f\"ur Mathematik, Boltzmannstrasse 3, 85748 Garching bei M\"unchen (Germany)}
\email{sodini@ma.tum.de}


\begin{document}

\begin{abstract}
  We prove a general criterion for the density in energy of suitable
  subalgebras of Lipschitz functions in the  metric Sobolev   space
  $H^{1,p}(X,\sfd,\mm)$ associated with a  positive and finite Borel measure $\mm$ in a separable and complete metric space $(X,\sfd)$.

  We then provide a relevant application   to the case of the
  algebra of cylinder functions 
  in the Wasserstein Sobolev space
  $H^{1,2}(\prob_2(\M),W_{2},\mm)$
  arising from a positive and finite Borel  measure $\mm$ on the
  Kantorovich-Rubinstein-Wasserstein space $(\prob_2(\M),W_{2}
  )$ of probability measures 
in a finite dimensional Euclidean space,
  a complete Riemannian manifold,
  or a
  separable Hilbert space $\M$.
  We will show that such a Sobolev space is always Hilbertian,
  independently of the choice of the reference measure $\mm$ so that
  the resulting Cheeger energy is a Dirichlet form.

  We will eventually provide an explicit characterization for the
  corresponding notion of $\mm$-Wasserstein gradient, showing useful
  calculus rules and its consistency with the tangent bundle
  and the $\Gamma$-calculus inherited from the Dirichlet form.
\end{abstract}

\maketitle
\tableofcontents
\thispagestyle{empty}

\section{Introduction}
The theory of Sobolev spaces associated to a metric measure space
$(X,\sfd,\mm)$ has been much developed in recent years. One of the most important
approaches (we refer to
the monographs \cite{Bjorn-Bjorn11,HKST15} and to the lecture notes \cite{GP20,Savare22})
is based on the notion of upper gradient
\cite{Heinonen-Koskela98,Koskela-MacManus98} of a map $f:X\to \R$: 
it is a Borel map $g:X\to[0,+\infty]$ satisfying
\begin{equation}
  \label{eq:133}
  |f(\gamma(b))-f(\gamma(a))|\le \int_\gamma g
\end{equation}
along every $\sfd$-Lipschitz (or even rectifiable) curve
$\gamma:[a,b]\to X$.
The Dirichlet space $D^{1,p}(X,\sfd,\mm)$, $p\in (1,+\infty)$  can then be defined as the
class of measurable functions $f:X\to \R$ that possess a $p$-integrable upper
gradient, thus resulting in a finite 
Newtonian energy
\begin{equation}
  \label{eq:134}
  \NE_p(f):=\inf\Big\{\int_Xg^p\,\d\mm:\text{$g$ is an upper gradient
    of $f$}\Big\}.
\end{equation}
A crucial and nontrivial fact is that $\NE_p$ admits a local
representation
\begin{equation}
  \label{eq:136}
  \NE_p(f)=\int_X |\rmD f|_N^p\,\d\mm
\end{equation}
in terms of the \emph{minimal $p$-weak upper gradient} $|\rmD
f|_N^p$ of $f$, which can be characterized in terms of the upper gradient
property \eqref{eq:133} along $\mathrm{\Mod}_p$-a.e.~curve and enjoys
many nice metric-nonsmooth calculus rules.
When $f$ is Lipschitz, then the pointwise and asymptotic Lipschitz constants
\begin{equation}
  \label{eq:1intro}
  |\rmD f|(x):=\limsup_{y\to x}\frac{|f(y)-f(x)|}{\sfd(x,y)},
  \quad
  \lip f(x):=
   \limsup_{y,z\to x,\ y\neq z}\frac{|f(y)-f(z)|}{\sfd(y,z)}
\end{equation}
are upper gradients, so that
\begin{equation}
  \label{eq:36}
  \text{if $f\in \Lip_b(X)$ then}\quad
  |\rmD f|_N\le |\rmD f|\le \lip f\quad\text{$\mm$-a.e.~in $X$}.
\end{equation}
Functions in $L^p(X,\mm)$ which admit a good representative
(in the usual Lebesgue class defined up to $\mm$-negligible sets) in
$D^{1,p}(X,\sfd,\mm)$ give raise to the Newtonian spaces $\hat
N^{1,p}(X,\sfd,\mm)$ \cite{Shanmugalingam00}
\cite[Def.~1.19]{Bjorn-Bjorn11}, which 
is a Banach space with the norm $\|f\|_{\hat
  N^{1,p}}:=\big(\|f\|_{L^p}^p+\NE_p(f)\big)^{1/p}$. $\hat N^{1,p}(X,\sfd,\mm)$ can also be
identified with the domain of the $L^p$-relaxation of $\NE_p$
\cite{Cheeger99}. 

\medskip
\paragraph{\em\bfseries Density of Lipschitz
  functions: the case of doubling spaces supporting a Poincar\'e inequality}
It is a natural question if
$\NE_p$ can be recovered starting from the distinguished class of
upper gradients given by the pointwise or asymptotic
Lipschitz constants \eqref{eq:1intro}
of Lipschitz functions. 
When $(X,\sfd,\mm)$ satisfies a doubling condition and supports a $p$-Poincar\'e inequality,
then Lipschitz functions are dense in $\hat N^{1,p}(X,  \sfd,\mm)$ 
\cite[Theorem 4.1]{Shanmugalingam00}
and for Lipschitz functions the minimal $p$-weak upper gradient $|\rmD f|_N$
coincides with the pointwise Lipschitz constant $|\rmD f|$
\cite[Theorem 6.1]{Cheeger99}. In particular 
for every $f\in \hat N^{1,p}(X,\sfd,\mm)$ there exists a sequence
$f_n\in \Lip_b(X)$ such that 
\begin{equation}
  \label{eq:142}
  f_n\to f,\quad 
  |\rmD f_n|\to |\rmD f|_N\quad\text{strongly in }L^p(X,\mm).
\end{equation}
It is worth noticing that in this case $\hat N^{1,p}(X,\sfd,\mm)$ is a
reflexive space \cite[Theorem 4.48]{Cheeger99}.

\medskip
\paragraph{\em\bfseries Density in energy of subalgebras of Lipschitz
  functions}
The strong approximation property \eqref{eq:142}
holds in fact for arbitrary complete and separable
metric spaces, a result obtained in \cite{AGS14I} (when $p=2$)
and \cite{AGS13}, where also the approximation by
the asymptotic Lipschitz constant is considered.

The first aim of the present paper is to discuss the extension of this
result when the sequence $f_n$ in \eqref{eq:142} is chosen in a 
suitable unital subalgebra $\AA\subset \Lip_b(X)$
separating the points of $X$, i.e.
\begin{equation}
  \label{eq:116-intro}
  1\in \AA,\quad
  \text{for every $x_0,x_1\in X$ there exists $f\in \AA$:\quad $f(x_0)\neq f(x_1)$}.
\end{equation}
The use of a subalgebra is  not a formal exercise with no
implications. In fact, in relevant examples such as
Wasserstein Sobolev spaces, which we shall introduce and discuss below, and the related subalgebra of cylinder functions, it is possible to recover the Cheeger energy as suitable relaxation of an {\it explicitly computable Dirichlet form}. Namely, from the so-called pre-Cheeger energy
\begin{equation}\label{eq:prec-intro}
  \pCE_{p}(f) := \int_X (\lip f)^p \,\d \mm, \quad f \in \Lip_b(X),
\end{equation}
we recover the Cheeger energy as its relaxation starting from $\AA$: 
\begin{equation}\label{eq:relpre-intro}
  \CE_{p,\AA}(f) = \inf \left \{ \liminf_{n \to + \infty}
    \pCE_{p}(f_n) : f_n \in \AA, \, f_n \to f \text{ in } L^0(X,\mm)
  \right \}.
\end{equation}
Thanks to the algebraic properties of $\AA$ and
\eqref{eq:116-intro}
it is possible to prove \cite[Sec.~3]{Savare22} that $\CE_{p,\AA}$ admits a local
representation
of the form 
\begin{equation}
  \label{eq:143}
  \CE_{p,\AA}(f)=\int_X |\rmD f|_{\star,\AA}^p(x)\,\d\mm(x)\quad
  \text{whenever }\CE_{p,\AA}(f)<+\infty,
\end{equation}
in terms of a minimal $(p,\AA)$-relaxed gradient
$|\rmD f|_{\star,\AA}$, enjoying the same calculus rules as $|\rmD
f|_N$, see Theorem \ref{thm:omnibus} below.
We denote by $H^{1,p}(X,\sfd,\mm;\AA)$ the class of functions in $L^p(X,\mm)$ with finite $(p,\AA)$-Cheeger energy.

It is easy to check that $H^{1,p}(X,\sfd,\mm;\AA)\subset \hat N^{1,p}(X,\sfd,\mm)$ with
$|\rmD f|_N\le |\rmD f|_{\star,\AA}$ $\mm$-a.e.
It turns out that the strong approximation property
\begin{equation}
  \label{eq:144}
  \text{for every }f\in \hat N^{1,p}(X,\sfd,\mm)\ 
  \text{there exist }f_n\in\AA:\ \ 
  f_n\to f,\ \lip f_n\to |\rmD f|_N\text{ in }L^p(X,\mm)
\end{equation}
is equivalent to the identification
\begin{equation}
  \label{eq:145}
  H^{1,p}(X,\sfd,\mm;\AA)=\hat N^{1,p}(X,\sfd,\mm), 
  \quad |\rmD f|_N=|\rmD f|_{\star,\AA}\quad\text{for every }f\in \hat N^{1,p}(X,\sfd,\mm).
\end{equation}
The density results of \cite{AGS14I,AGS13}
show that \eqref{eq:145} always hold if $\AA=\Lip_b(X)$.
When $\AA$ is a proper subalgebra of $\Lip_b(X)$,
a first sufficient condition for the validity of \eqref{eq:145}, in the more general framework of
extended topological metric measure spaces, 
is provided by the
compatibility condition between $\sfd$ and $\AA$
\cite[Theorems 3.2.7, 5.3.1]{Savare22}
\begin{equation}
  \label{eq:11p}
  \sfd(x,y)=\sup\Big\{f(x)-f(y):
  f\in \AA,\ \Lip(f, X)\le 1\Big\}.
\end{equation}
We are able to improve \eqref{eq:11p} and to show (Theorem
\ref{theo:startingpoint}) that a necessary and
sufficient condition for \eqref{eq:145} is that
for every $y\in X$ (or in a dense subset of $X$) the distance function
$\sfd_y:x\mapsto \sfd(x,y)$ 
satisfies 
\begin{equation}
  \label{eq:146-intro}
  |\rmD \sfd_y|_{\star,\AA}(x)\le 1\quad
  \text{for $\mm$-a.e.~$x\in X$}.
\end{equation}
As mentioned above,  the density of distinguished subalgebras of
Lipschitz functions can provide valuable information on the structure
of the  metric Sobolev   space $\hat N^{1,p}(X,\sfd,\mm)$, in particular
when the asymptotic Lipschitz constant $\lip f$ exhibits a more
regular behaviour when restricted to $\AA$.
A relevant example arises when an algebra $\AA$ exists for which 
the pre-Cheeger energy \eqref{eq:prec-intro} is induced by a bilinear
form and one wants to study its closure
as it is typical in the theory of Dirichlet forms. In this case our
result shows that this construction is intrinsically linked to the
metric structure, so that it is independent of the particular choice
of  the algebra $\AA$ satisfying \eqref{eq:146-intro}  and it is invariant with respect to measure-preserving
isometries.
As a byproduct, we will recover in a simple way previous Hilbertianity
results of
\cite{DLP20,LP20,Savare22}.

\medskip
\paragraph{\em\bfseries
  The Wasserstein Sobolev space}
An important application, which has been one of the inspiring motivations of
our investigation, 
concerns functional analysis over spaces of probability measures. In fact, smooth functions do appear recently as solutions of new types of partial differential equations over spaces of probability measures defined by diverse forms of differentiation, namely  nonlinear transport equations \cite{AFMS21,AG08} for describing population evolutionary games, Kolmogorov equations \cite{M22-1,M22-2} in nonlinear filtering, and Hamilton-Jacobi-Bellman equations \cite{Bayraktar2016RandomizedDP,GANGBO2019119,PW18,CCP20} as appearing, e.g., in the theory of mean-field games and mean-field optimal control. 
In some of these instances, the solutions are considered in classical sense, because of the lack of  weak formulations and variational descriptions. Moreover, while the expression of these equations is in most cases of foundational interest, their relevance in terms of providing insights about solutions and their explicit computation remained so far rather unclear.
Hence, a proper definition of function spaces of regular functions, rules of calculus, and density properties are fundamental for developing a more systematic framework for the analysis of such novel forms of infinite dimensional PDEs and explaining their practical use and impact.

Moreover, thanks to significant advances in computational optimal transport that made its numerical realization feasible also for problems of relatively high dimension, in the past decade there has been an increasing and more accepted adoption of probability measures to model data points in image and shape processing and other machine learning applications. While the first applications were about discriminating data encoded as distributions available in the form of bags-of-features or descriptors, more recent developments explored geometric interpolation of data provided by optimal transport, for instance in the form of Wasserstein barycenters. In the meanwhile the literature on the subject has grown significantly to be really able to offer a complete account and we may more simply refer to the recent survey \cite{PC19} for insights and references. 

Building upon these advances, approximating or interpolating efficiently functions over data points modeled as (probability) measures can also provide a novel framework for machine learning tasks, such as classification and regression. Also for such developments a proper foundation of functional analysis is necessary.

These are relevant motivations for us to focus on the study of  Sobolev spaces generated by a finite measure $\mm$ on the
Wasserstein space $\prob_2(\M)$ of Borel probability measures in
a complete Riemannian manifold $(\M,\sfd_\M)$ with finite quadratic 
moment 
\begin{equation}
  \label{eq:147}
  \int_{\M} \sfd^2_{ \M}(x,x_o)\,\d\mu(x)<+\infty\quad
  \text{for some, and thus any, $x_o\in \M$},
\end{equation}
endowed with the
$L^2$-Kantorovich-Rubinstein-Wasserstein distance $W_2$
\begin{equation}
  \label{eq:148}
  W_{2,\sfd_\M}^2(\mu, \nu) := \min
  \left \{ \int_{\M\times \M} \sfd^2_\M(x,y)\,\d\mmu(x,y)\mid \mmu \in \Gamma(\mu, \nu) \right \};
\end{equation}
here $\Gamma(\mu, \nu)$ is the set of couplings between $\mu$ and
$\nu$, i.e.~probability measures $\mmu$ in $\M\times \M$ whose
marginals are $\mu$ and $\nu$.

The space of probability measures $(\prob_2(\M),W_{2,\sfd_\M})$  
may be considered a model class for the above mentioned applications and  it is an example of complete and separable metric
space, which exhibits a non-smooth, infinite dimensional
pseudo-Riemannian character
\cite{Otto01,AGS08,Villani09}. In particular, it is not isometric to a
finite dimensional Riemannian manifold or a $\mathrm{Cat}(\kappa)$
space \cite{DGPS21}; when $\M$ has nonnegative
sectional curvature as in the case of the Euclidean space
$\R^d$, then  $(\prob_2(\M),W_{2,\sfd_\M})$ has nonnegative curvature
in the sense of Aleksandrov \cite{Villani09}; for a general Riemannian
manifold $\M$,
$(\prob_2(\M),W_{2,\sfd_\M})$ is not an Aleksandrov space and lacks of
any lower or upper curvature bound.

When $\M$ is compact,
Sobolev spaces on $(\prob_2(\M),W_2)$ have been constructed in  \cite{DelloSchiavo20} starting
from measures $\mm$ which have full support and satisfy an integration-by-parts formula
(see Section \ref{subsec:examples} below)
on the unital algebra of cylinder maps $\ccyl\infty c{\prob_2(\M)}$, generated by linear
functionals of the form
\begin{equation}
  \label{eq:149}
  \lin\phi:\mu\mapsto \int_\M \phi\,\d\mu, \quad \phi\in\rmC^\infty_c(\M).
\end{equation}
It turns out that the restriction of the pre-Cheeger energy $\pCE_2$ \eqref{eq:prec-intro}
to $\ccyl\infty c{\prob_2(\M)}$ is induced by a
bilinear form
so that one can study the Dirichlet form arising by its closure.

Our main result is that for every separable and complete Riemannian manifold $\M$ and for every 

positive and finite Borel  measure $\mm$ on
$\prob_2(\M)$ (so, full support and integration-by-parts properties are not required)
the algebra $\ccyl\infty c{\prob_2(\M)}$ satisfies
property \eqref{eq:146-intro} and therefore it is dense in the metric
Sobolev space $H^{1,2}(\prob_2(\M),W_{2,\sfd_\M},\mm)$, which is
therefore a Hilbert space.
Due to the non-smooth character of $W_2$ such a Hilbertianity
property was far from obvious even in the flat case $\M=\R^d$.
We will also show that this result holds when $\M$ is an
infinite-dimensional, separable, Hilbert space.

Our metric analysis is also supplemented with a detailed discussion of
the structure of the Cheeger energy and of the minimal relaxed gradient,
in the case when $\M$ is the Euclidean space $\R^d$.
Introducing the measure $\displaystyle\bmm=\int\big(\delta_\mu\otimes
\mu\big)\,\d\mm(\mu)$ in $\prob\big(\prob_2(\R^d)\times \R^d)$, we
will show that there is a linear continuous Wasserstein-gradient
operator $\rmD_\mm:H^{1,2}(\prob_2(\R^d),W_2,\mm)\to
L^2(\prob_2(\R^d)\times \R^d,\bmm ; \R^d )$ representing the bilinear form
associated to the Cheeger energy as
\begin{equation}
  \label{eq:150}
  \CE_2(F,G)=\int \rmD_\mm F(\mu,x)\cdot \rmD_\mm G(\mu,x)\,\d\bmm(\mu,x),
\end{equation}
and satisfying useful calculus rules which are typical of
$\Gamma$-calculus for Dirichlet form. $\rmD_\mm$ also allows for an explicit
characterization of the
tangent bundle $L^2\big(\rmT \prob_2(\R^d)\big)$ in the sense of Gigli
\cite{Gigli18,GP20}.

We are also able to study the relaxation effect occurring in the
construction of the Cheeger energy starting from
\eqref{eq:prec-intro}.
We claim that our results
are sufficiently strong and provide useful tools to pave the way for
further studies on the structure and the promising applications of
Wasserstein Sobolev spaces.
In particular, the techniques developed in the present paper can also be
applied
to study the general class of Wasserstein Sobolev spaces
$H^{1,q}( \prob_p(\M), 
W_p, \mm)$, with $p,q \in (1,+\infty)$,
a topic that  has been addressed in \cite{S22}. 

Moreover,  as a direct consequence of our results, the recovery of the Cheeger energy in terms of relaxation of the explicitly computable
 pre-Cheeger energy $\pCE_2$ \eqref{eq:prec-intro}
on $\ccyl\infty c{\prob_2(\M)}$ does allow the equally explicit formulation of Euler-Lagrange equations of properly formulated variational problems defined on $H^{1,2}(\prob_2(\M),W_{2,\sfd_\M},\mm)$, which can be solved numerically over  finite dimensional suitably graduated approximations of $\ccyl\infty c{\prob_2(\M)}$, as a sort of (nonlinear) Galerkin approximation. Hence, as a concluding remark, perhaps surprisingly, the use of the subalgebra of cylindric functions $\ccyl\infty c{\prob_2(\M)}$ instead of $\Lip_b\big(\prob_2(\M)\big)$ as a fundamental nucleus to define Wasserstein Sobolev spaces allows to bring the theory from its foundational level to rather concrete applicability. In particular, we have in mind the above mentioned applications to the solutions of PDEs over $\prob_2(\M)$ and machine learning.

\medskip
\paragraph{\em\bfseries Plan of the paper}
After a quick review of the construction of the Cheeger energy
starting from a subalgebra $\AA$, \textbf{Section \ref{sec:main1}}
is devoted to prove our main density result under condition
\eqref{eq:146-intro}
(Section \ref{subsec:density}). The last part \ref{subsec:intrinsic1}
extends the applicability of the results to a larger class of distances:
one of its quite useful applications will concern the
extension of the results for the Wasserstein Sobolev
spaces modeled on $\prob_2(\R^d)$ to the general case of
$\prob_2(\M)$ for a complete Riemannian manifold $\M$, which will be
carried out in Sections \ref{subsec:intrinsic} and \ref{subsec:WSR}.

We will recap a few properties of the Wasserstein distance in \textbf{Section
\ref{sec:Wasserstein}}.
\textbf{Section \ref{sec:main2}} contains
a collection of some properties of cylinder functions, of their
asymptotic Lipschitz constants (Section \ref{subsec:cylindrical}), and
our main density and
Hilbertianity result for the Wasserstein Sobolev space
$H^{1,2}(\prob_2(\R^d),W_2,\mm)$ (Theorem \ref{thm:main}).

Calculus rules for the $\mm$-differential are presented in
\textbf{Section \ref{sec:calculus}}; the structure of the tangent
bundle, the properties of the residual differentials, and the study of the relaxation
effect are discussed in Section \ref{subsec:resdiff}, together with a
few examples in Section \ref{subsec:examples}.

The last \textbf{Section \ref{sec:extensions}} shows how to
extend the result of Section \ref{sec:main2} from $\R^d$ to an
arbitrary complete Riemannian manifold $\M$ and to a separable Hilbert
space $\H$.

\medskip
\paragraph{\em\bfseries Acknowledgments}
The authors warmly thank N.~Gigli, E.~Pasqualetto,
G.~Peyr\'e for their comments
and L.~Dello Schiavo for the careful reading and his valuable remarks
on a first draft of the present paper.

The authors
  gratefully acknowledge the support of the Institute for Advanced
  Study of the Technical University of Munich, funded by the German Excellence Initiative.\\
  G.S.~has also been supported by 
  IMATI-CNR, Pavia and by the MIUR-PRIN 2017
  project \emph{Gradient flows, Optimal Transport and Metric Measure
    Structures.}  The authors are grateful to the anonymous reviewers for their valuable comments.

\section{Metric Sobolev spaces and density of unital algebras}
\label{sec:main1}
In this section we will briefly recap the construction of metric
Sobolev spaces adapting the relaxation viewpoint of the Cheeger
energy to the presence of a distinguished algebra of Lipschitz
functions \cite{AGS14I,AGS13,Savare22}.

\subsection{Sobolev functions and minimal relaxed gradients}
Let $(X,\mathsf d)$ be a complete and separable metric space.
We will denote by $\Lip_b(X,\sfd)$
the space of bounded and Lipschitz real
functions $f:X\to \mathbb R$. The asymptotic Lipschitz constant of
$f\in \Lip_b(X, \sfd)$ is defined as
\begin{equation}
  \label{eq:1}
  \lip_\sfd f(x):=\lim_{r\down0}\Lip(f,\rmB(x,r),\sfd)=
  \limsup_{y,z\to x,\ y\neq z}\frac{|f(y)-f(z)|}{\sfd(y,z)},
\end{equation}
where $\rmB(x,r)$ denotes the open ball centered at $x$ with
radius $r$ and,
for $A \subset X$, the quantity $\Lip(f,A,\sfd)$ is defined as
\[ \Lip(f,A,\sfd):= \sup_{x,y\in A, \, x \neq
    y}\frac{|f(x)-f(y)|}{\sfd(x,y)}. \]
We will simply write $\Lipb(X), \lip f, \Lip (f,A)$, omitting to
explicitly mention $\sfd$, when the
choice of the metric $\sfd$ is clear from the context. 

We will also deal with
a unital algebra
$\AA\subset \Lip_b(X)$
separating the points of $X$, i.e.
\begin{equation}
  \label{eq:116}
  1\in \AA,\quad
  \text{for every $x_0,x_1\in X$ there exists $f\in \AA$:\quad $f(x_0)\neq f(x_1)$}.
\end{equation}
The initial Hausdorff topology $\tau_\AA$ induced on $X$ by $\AA$ is clearly
coarser than the metric topology of $X$.

Let $\mm$ be a  finite and positive  Borel measure on $X$
(being $X$ a Polish space, $\mm$ is also a Radon measure).
We will denote by $\cL^0(X,\mm)$
the set of $\mm$-measurable real
functions defined in $X$; $L^0(X,\mm)$ is the usual quotient of
$\cL^0(X,\mm)$
obtained by identifying two functions which coincide $\mm$-a.e.~in
$X$. In a similar way, 
$\cL^p(X,\mm)$ and $L^p(X,\mm)$ are the usual Lebesgue spaces of
$p$-summable $\mm$-measurable (equivalence classes of) real functions,
$p\in [1,+\infty]$. 
 It is worth noticing that by \cite[Lemma 2.1.27]{Savare22} we have that 
\begin{equation}\label{eq:isdense}
\begin{split}
    &\quad \quad \quad \text{ for every $p\in [1,\infty)$ and every $f \in \mathcal{L}^p(X, \mm)$ taking values in an interval $I\subset \R$} \\
    &\text{there exists a sequence $(f_n)_n \subset \AA$ with values in $I$ converging to $f$ in $L^p(X,\mm)$}.
\end{split}
\end{equation}
We will endow $L^0(X,\mm)$ with the topology of the convergence
in measure, which is induced by the metric
\begin{equation}
  \label{eq:84}
  \mathrm d_{L^0}(f_1,f_2):=\int_X \vartheta(|f_1-f_2|)\,\d\mm
  , \quad f_1, f_2 \in L^0(X,\mm), 
\end{equation}
where $\vartheta:[0,+\infty)\to [0,+\infty)$ is any increasing,
concave, bounded function with
$\vartheta(0)=\lim_{r\down0}\vartheta(r)=0$.
 In the following we fix an exponent $p\in
(1,+\infty)$.

\begin{definition}[$(p,\AA)$-relaxed gradient]
  \label{def:relgrad}
  We say that $G\in L^p(X,\mm)$ is a $(p,\AA)$-relaxed gradient of
  a $\mm$-measurable function $f\in L^0(X,\mm)$ if
  there exists a sequence $(f_n)_{n\in \N}\in \AA$ such that:
  \begin{enumerate}
  \item $f_n\to f$ in $\mm$-measure and
    $\lip f_n\to \tilde G$ weakly in $L^p(X,\mm)$;
  \item $\tilde G\le G$ $\mm$-a.e.~in $X$.  
  \end{enumerate}
  The minimal $(p,\AA)$-relaxed gradient of $f$ (denoted by
  $|\rmD f|_{\star,\AA}$) is the element of minimal $L^p$-norm among all
  the $(p,\AA)$-relaxed gradient of $f$.
  We will just write $|\rmD f|_\star$ if $\AA=\Lip_b(X)$.
\end{definition}
\begin{remark}
Notice that the minimal relaxed gradient $|\rmD f|_{\star,\AA}$ depends also on $p \in [1,+\infty)$, see e.g.~\cite{AGS11c, Bjorn-Bjorn11, HKST15}. Since it will be always clear from the context which value of $p$ we are considering (a general one or, in the second part of the paper, $p=2$), we omit to write explicitly this dependence.
\end{remark}

We collect in the following Theorem
the main properties of $|\rmD f|_{\star,\AA}$ we will extensively use.
\begin{theorem}
  \label{thm:omnibus}\ 
  \begin{enumerate}[\rm (1)]
  \item The set
    \begin{displaymath}
      S:=\Big\{(f,G)\in L^0(X,\mm)\times L^p(X,\mm):
      \text{$G$ is a $(p,\AA)$-relaxed gradient of $f$}\Big\}
    \end{displaymath}
    is convex and it is closed with respect to
    to the product topology of the convergence in $\mm$-measure and
    the weak convergence in $L^p(X,\mm)$.
    In particular, the restriction $S_q:=S\cap L^q(X,\mm)\times
    L^p(X,\mm)$ is weakly closed in
    $L^q(X,\mm)\times
    L^p(X,\mm)$ for every $q\in (1,+\infty)$.
  \item (Strong approximation) If
    $f \in L^0(X,\mm)$  has a $(p,\AA)$-relaxed gradient then
    $|\rmD
    f|_{\star,\AA}$ is well defined. If $f$ takes values in a closed
    (possibly unbounded) 
    interval $I\subset \R$ then
    there exists a sequence $f_n\in \AA$ with values in $I$
    such that
    \begin{equation}
      \label{eq:4}
      f_n \to f\text{ $\mm$-a.e.~in $X$},\quad
      \lip f_n\to |\rmD f|_{\star,\AA}\text{ strongly in }L^p(X,\mm).
    \end{equation}
    If moreover $f\in L^q(X,\mm)$ for some $q\in [1,+\infty)$ then 
    we can also find a sequence as in \eqref{eq:4} converging strongly
    to $f$ in $L^q(X,\mm)$.
  \item (Pointwise minimality)
    If $G$ is a $(p,\AA)$-relaxed gradient of
    $f \in L^0(X,\mm)$  then $|\rmD
    f|_{\star,\AA}\le G$ $\mm$-a.e.~in $X$.
  \item (Leibniz rule)
    If $f,g\in L^\infty(X,\mm)$ have $(p,\AA)$-relaxed gradient, then
    $h:=fg$ has $(p,\AA)$-relaxed gradient and
    \begin{equation}
      \label{eq:5}
      |\rmD (fg)|_{\star,\AA}\le |f|\,|\rmD g|_{\star,\AA}+
      |g|\,|\rmD f|_{\star,\AA} \quad\text{$\mm$-a.e.~in }X.
    \end{equation}
  \item (Sub-linearity)
    If $f,g \in L^0(X,\mm)$  have $(p,\AA)$-relaxed gradient  and $\alpha, \beta \in \R$,  then
    \begin{equation}
      \label{eq:6}
      |\rmD(\alpha f+\beta g)|_{\star,\AA}\le |\alpha|\,
      |\rmD f|_{\star,\AA}+
      |\beta|\,|\rmD g|_{\star,\AA} \quad\text{$\mm$-a.e.~in }X.
    \end{equation}
  \item (Locality)
    If $f \in L^0(X,\mm)$  has a $(p,\AA)$-relaxed gradient, then
    for any $\LL^1$-negligible Borel subset $N\subset \R$ we have
    \begin{equation}
      \label{eq:7}
      |\rmD f|_{\star,\AA}=0\quad\text{$\mm$-a.e.~on $f^{-1}(N)$}.
    \end{equation}
  \item (Chain rule)
    If $f \in L^0(X,\mm)$  has a $(p,\AA)$-relaxed gradient and $\phi\in \Lip(\R)$
    then $\phi\circ f$
    has $(p,\AA)$-relaxed gradient and
    \begin{equation}
      \label{eq:8}
      |\rmD (\phi\circ f)|_{\star,\AA}\le |\phi'(f)|\,|\rmD f|_{\star,\AA}  \quad\text{$\mm$-a.e.~in }X,
    \end{equation}
    and equality holds in \eqref{eq:8} if $\phi$ is monotone or
    $\rmC^1$.
    \item (Truncations) If $f_j \in L^0(X,\mm)$  has $(p,\AA)$-relaxed gradient, $j=1,\cdots,J$,
      then
      also the functions $f_+:=\max(f_1,\cdots,f_J)$ and
      $f_-:=\min(f_1,\cdots,f_J)$ have $(p,\AA)$-relaxed gradient
      and
      \begin{align}
        \label{eq:9}
        |\rmD f_+|_{\star,\AA}=|\rmD f_j|_{\star,\AA}&\quad\text{$\mm$-a.e.~on }
                                                       \{x\in
                                                       X:f_+=f_j\},\\
        |\rmD f_-|_{\star,\AA}=|\rmD f_j|_{\star,\AA}&\quad\text{$\mm$-a.e.~on }
                                                       \{x\in X:f_-=f_j\}.
      \end{align}
  \end{enumerate}
\end{theorem}
\begin{remark} Notice that the product in \eqref{eq:8} is well defined since there
  exists a $\Leb 1$-negligible Borel set $N\subset \R$ such that
  $\phi$ is differentiable in $\R\setminus N$  and
  $|\rmD f|_{\star, \AA}$ vanishes $\mm$-a.e.~in $f^{-1}(N)$ thanks to
  the locality property \eqref{eq:7}.
\end{remark}
\begin{proof}
  We give a few references for the proofs.
  The case when $p=2$, $\AA=\Lip_b(X)$ and
  the local slope of $f$ is used to define relaxed gradients
  have been considered in \cite[Sec.~4]{AGS14I}, whose proof
  generalizes easily to the case $p\in (1,\infty)$ and 
  the asymptotic Lipschitz constant \eqref{eq:1}, see also \cite{AGS13}.

  The definition and the properties involving a general unital
  subalgebra $\AA$ have been discussed in \cite[Sec.~3]{Savare22}:
  points (1,2) correspond to Lemma 3.1.6 and Corollary 3.1.9,
  (3) has been stated in Lemma 3.1.11, (4) refers to Corollary 3.1.10,
  (5,6,7,8) are proved in Theorem 3.1.12 and its Corollary 3.1.13.

  Let us make three further technical comments:
  \begin{itemize}
  \item both \cite{AGS14I,Savare22} involve an auxiliary topology
    $\tau$:
    in the present case, being $X$ complete and separable
    and $\sfd$ a canonical metric (thus $\sfd$ only take finite
    values),
    we can select $\tau$ as the (Polish) topology induced by $\sfd$.
  \item In order to deal with \emph{extended} distances, in
    \cite{Savare22} has also been assumed that the
     unital algebra $\AA$ 
    satisfies the \emph{stronger} compatibility condition
    \begin{equation}
      \label{eq:11}
      \sfd(x,y)=\sup\Big\{f(x)-f(y):
      f\in \AA,\ \Lip(f, X)\le 1\Big\},
    \end{equation}
    which clearly implies that $\AA$ separates the points of $X$ as in \eqref{eq:116}.
    However, such a property is not needed in the construction and the
    proofs of
    Section 3.1.1 of \cite{Savare22}.
     The only point where \eqref{eq:11} explicitly occurs is in
    the proof of Locality \cite[Lemma 3.1.11]{Savare22}, to ensure
    that the restriction of $\AA$ to each compact set $K\subset X$ is
    uniformly dense in $\rmC(K)$, a property which is guaranteed 
    in the present setting  by
    \eqref{eq:116} thanks to Stone-Weierstrass Theorem.
  \item The standard approach of \cite{AGS14I,Savare22}
    considers first functions $f$ belonging to $L^p(X,\mm)$ instead of
    general $\mm$-measurable functions.
    However, the compatibility with truncations showing that
    for every $k>0$
    \begin{equation}
      \label{eq:10}
      |\rmD \,T_k(f)|_{\star,\AA}(x)=
      \begin{cases}
        |\rmD f|_{\star,\AA}(x)&\text{if }|f(x)|< k,\\
        0&\text{if }|f(x)|\ge k,
      \end{cases}
      \qquad
      T_k(f):=-k\lor f\land k,
    \end{equation}
    and the possibility to find strong approximations of $T_k(f)  \in L^p(X, \mm)$  (recall that $\mm$ is finite)  satisfying
    \eqref{eq:4} and taking values in $[-k,k]$
    (see  \cite[Cor 2.1.24, Cor.~3.1.9]{Savare22} where an approximation argument involving odd polynomials is implemented)
     allow for a standard extension of the theory from
     $L^p(X,\mm)$ to
    $L^0(X,\mm)$, see also the discussion related to (4.16) of
    \cite{AGS14I}.  Notice also that, from a metric point of view, there is no reason to couple the integrability of a function $f$ and the one of its minimal relaxed gradient $|\rmD f|_{\star, \AA}$. Also the choice of working in $L^0(X, \mm)$ gives more flexibility and in particular allows to treat the distance function $\sfd_y$ from one point $y \in X$ without imposing any integrability condition. This will be crucial in the rest of this paper (see e.g.~Theorem \ref{theo:startingpoint}).
    \qedhere
  \end{itemize}
\end{proof}
 Starting from Definition \ref{def:relgrad} and using the
properties of Theorem \ref{thm:omnibus} it is natural to introduce the
following notions. 
\begin{definition}[Cheeger energy and Sobolev space]
   We call $D^{1,p}(X,\sfd,\mm;\AA)$ the set of functions in  $L^0(X,\mm)$ with
  a $(p,\AA)$-relaxed gradient and we set
  \begin{equation}
    \label{eq:3}
    \CE_{p,\AA}(f):=\int_X |\rmD f|_{\star,\AA}^p(x)\,\d\mm(x)\quad
    \text{for every $f\in
  D^{1,p}(X,\sfd,\mm;\AA)$},
  \end{equation}
  with $\CE_{p,\AA}(f):=+\infty$ if $f\not\in D^{1,p}(X,\sfd,\mm;\AA).$
  The Sobolev space $H^{1,p}(X,\sfd,\mm;\AA)$ is defined as
  $L^p(X,\mm)\cap D^{1,p}(X,\sfd,\mm;\AA)$
  and 
  it is a Banach space with the norm
  $\|f\|_{H^{1,p}(X,d,\mm;\AA)}^p:=\|f\|_{L^p}^p+\CE_{p, \AA}(f)$.
  As usual, we will write $D^{1,p}(X,\sfd,\mm), \CE_p(f)$, $H^{1,p}(X,\sfd,\mm)$ and $\|f\|_{H^{1,p}}$ when $\AA=\Lip_b(X)$. 
\end{definition}
\begin{remark}[Cheeger energy as relaxation of the pre-Cheeger
  energy] \label{rem:relprec}
  We can equivalently define the Cheeger energy $\CE_{p,\AA}$ as  a sort of 
   $L^0$-lower semicontinuous relaxation of the restriction to
  $\AA$ of the \emph{pre-Cheeger energy} $\pCE_p$,
   the latter being defined as
\begin{equation}\label{eq:prec} \pCE_{p}(f) := \int_X (\lip f)^p \,\d \mm, \quad f \in \Lip_b(X).
\end{equation}
 In other words, for every $f \in L^0(X,\mm)$ it holds (\cite[Corollary 3.1.7]{Savare22})
\begin{equation}\label{eq:relpre}
  \CE_{p,\AA}(f) = \inf \left \{ \liminf_{n \to + \infty}
    \pCE_{p}(f_n) : f_n \in \AA, \, f_n \to f \text{ in } L^0(X,\mm)
  \right \}.
\end{equation}
In particular the functional $\CE_{p,\AA}$ is lower semicontinuous in
$L^0(X,\mm)$.
Here the choice of the $L^0$-topology
does not play a crucial role, since, by Theorem \ref{thm:omnibus}(2), the restriction of $\CE_{p,\AA}$ to $L^q(X,\mm)$, $q\in [1,\infty)$,
can be equivalently obtained as $L^q$-relaxation:
\begin{equation}\label{eq:relpreq}
  \CE_{p,\AA}(f) = \inf \left \{ \liminf_{n \to + \infty}
    \pCE_{p}(f_n) : f_n \in \AA, \, f_n \to f \text{ in } L^q(X,\mm)
  \right \},
    \quad f\in L^q(X,\mm).
\end{equation}
 Notice also that, when $\mm$ has not full support, two different elements $f_1 , f_2 \in \AA$ may give rise to the same equivalence class in $L^0(X,\mm)$. In this case, $\CE_{p,\AA}$ can be equivalently defined as the $L^0$-lower semicontinuous relaxation of the functional 
\[ \widetilde{\pCE}_p(f):= \inf \left \{ \pCE_p(g) : g \in \AA, \, g = f \text{ $\mm$-a.e.} \right \}, \quad f \in \AA_\mm,\]
where $\AA_\mm$ is the quotient of $\AA$ with respect to equality $\mm$-a.e.. 
\end{remark}
It is clear that we have the obvious implication for $f \in L^0(X, \mm)$:
\begin{equation}
  \label{eq:13}
  \text{$f$ has a $(p,\AA)$-relaxed gradient}\quad
  \Rightarrow
  \quad
  \left\{
  \begin{aligned}
    &\text{$f$ has a $(p,\Lip_b(X))$-relaxed gradient and}\\
    &|\rmD
    f|_\star\le |\rmD f|_{\star,\AA} \text{ $\mm$-a.e.~in $X$}.
  \end{aligned}
  \right.
\end{equation}
The converse implication
together with the identity $|\rmD
f|_\star= |\rmD f|_{\star,\AA} $
is an important density property for an algebra $\AA$:
by Theorem \ref{thm:omnibus}(2), it is equivalent to the following
property.
\begin{definition}[Density in energy of a subalgebra of Lipschitz functions]
  \label{def:density}
  We say that a subalgebra $\AA\subset \Lip_b(X)$ is \emph{dense in $p$-energy}
  if for every $f\in L^0(X,\mm)$ with a $p$-relaxed gradient 
  there exists a sequence $(f_n)_{n\in \N}$ satisfying
  \begin{equation}
    \label{eq:4bis}
    f_n\in \AA,\quad
    f_n \to f\text{ $\mm$-a.e.~in $X$},\quad
    \lip f_n\to |\rmD f|_{\star}\text{ strongly in }L^p(X,\mm).
  \end{equation}
  When $\AA$ is unital and separating, this is equivalent to the fact
  that $f$
  has a $(p,\AA)$-relaxed gradient and
  \begin{equation}
    \label{eq:15}
    |\rmD f|_{\star,\AA}=|\rmD f|_{\star}\quad\text{$\mm$-a.e.~in $X$}.
  \end{equation}
  In particular $D^{1,p}(X,\sfd,\mm;\AA)=D^{1,p}(X,\sfd,\mm)$.
\end{definition}
\begin{remark}[Comparison with the Newtonian approach]
  \label{rem:N}
  By the identification \eqref{eq:145} when $\AA=\Lip_b(X)$ we always
  have
  \begin{equation}
    \label{eq:72}
    H^{1,p}(X,\sfd,\mm)=\hat N^{1,p}(X,\sfd,\mm),\quad
    |\rmD f|_N=|\rmD f|_{ \star  
}\quad\text{for every }f\in \hat N^{1,p}(X,\sfd,\mm).
  \end{equation}
  If $\AA$ is dense in $p$-energy and $f\in \hat N^{1,p}(X,\sfd,\mm)$
  we thus obtain
  \begin{equation}
    \label{eq:135}
    |\rmD f|_{\star,\AA}=|\rmD f|_{\star}=|\rmD f|_N\quad\text{$\mm$-a.e.~in $X$}.
  \end{equation}
  Notice that \eqref{eq:135} and \eqref{eq:36} immediately yield the
  uniform upper bound in terms of the pointwise Lipschitz constant
  \begin{equation}
    \label{eq:43}
    \text{if $f$ is Lipschitz then}\quad
    |\rmD f|_{\star,\AA}\le |\rmD f|\quad \text{$\mm$-a.e.~in $X$}.
  \end{equation}
\end{remark}
\begin{remark}
  \label{rem:dense2}
  As we already mentioned in Remark \ref{rem:relprec}, the choice of
  arbitrary measurable maps $f\in L^0(X,\mm)$
  in Definition \ref{def:density}  and of the pointwise
  $\mm$-a.e.~convergence in \eqref{eq:4bis} is not restrictive:
  a simple truncation argument (which can be implemented by using odd
  polynomials, see \cite[Corollary 2.1.24]{Savare22}) shows that 
  $\AA$ is dense in $p$-energy if and only if
  for every $f\in L^p(X,\mm)$ with a $p$-relaxed gradient 
  there exists a sequence $(f_n)_{n\in \N}$ satisfying
  \begin{equation}
    \label{eq:4tris}
    f_n\in \AA,\quad
    f_n \to f\text{ in $L^p(X,\mm)$},\quad
    \lip f_n\to |\rmD f|_{\star}\text{ strongly in }L^p(X,\mm).
  \end{equation}
  If $\AA$ is unital and separating this is equivalent to
  $H^{1,p}(X,\sfd,\mm;\AA)=H^{1,p}(X,\sfd,\mm)$ with equal norms.  
\end{remark}
A first sufficient condition to obtain the density in energy of a subalgebra $\AA$,  in the more general framework of
extended topological metric measure spaces, 
is provided by the
compatibility condition \eqref{eq:11}
\cite[Theorems 3.2.7, 5.3.1]{Savare22}
(see also \cite{Ambrosio-Stra-Trevisan17}
for the algebra generated by truncated distance functions).

 In the present Polish setting,
 we notice that
 \eqref{eq:4bis} (and, a fortiori, \eqref{eq:11})
 implies the weaker condition
\begin{equation}
  \label{eq:14}
  \text{for every $y\in X$ the function
    $\sfd_y:x\mapsto \sfd(x,y)$
    has $(p,\AA)$-relaxed gradient $1$,}
\end{equation}
which is equivalent, thanks to Theorem \ref{thm:omnibus}( 3), to 
\begin{equation}
  \label{eq:2}
  |\rmD \sfd_y|_{\star,\AA}\le 1\quad\text{$\mm$-a.e.~in $X$}.
\end{equation}
In fact,
using the truncations \eqref{eq:10}, each function $\sfd_y$ can be approximated by the increasing sequence
$f_k:=T_k \sfd_y$ of  bounded  $1$-Lipschitz maps, so that
\begin{equation}
  \label{eq:90}
  |\rmD \sfd_y|_\star\le 1\quad\text{$\mm$-a.e.~in $X$ for every $y\in X$},
\end{equation}
and therefore \eqref{eq:4bis} yields \eqref{eq:2}.
  \begin{remark}[The effect of truncations]
    \label{rem:trunc}
    The $(p,\AA)$-relaxed gradient is not affected by truncations of
    the distance functions, in particular it is not restrictive
      to assume $\sfd$ bounded above by a constant, e.g.~$1$\emph{}.
      In fact, if we introduce  a parameter $a>0$ and the truncated distance
    \begin{equation}
      \label{eq:41}
      \sfd_{ a}(x_1,x_2):=\sfd(x_1,x_2)\land a\quad\text{for every
      }x_1,x_2\in X,
    \end{equation}
    $(X,\sfd_{ a} )$ is still a complete and
    separable metric space, the sets $\Lip_b(X,\sfd)$ and
    $\Lip_b(X,\sfd_{ a} )$ coincide, and it is easy to check that 
    \begin{equation}
      \lip_{\sfd} f=\lip_{\sfd_{ a}}f\quad
      \text{for every bounded and Lipschitz function $f$}. 
  \label{eq:42}  
  \end{equation}
  We deduce that $\sfd$ and $\sfd_{ a} $ induce the same $(p,\AA)$-relaxed
  gradient.
  Notice moreover that using  \eqref{eq:41} we can also easily cover the case of
  extended distances (i.e.~possibly assuming the value $+\infty$),
  \emph{provided $(X,\sfd_a)$ is a separable metric space}.
  The case when $(X,\sfd_a)$ is not separable requires a more refined
  setting involving an auxiliary topology $\tau$ \cite{Savare22}.
\end{remark}
It is possible to express \eqref{eq:2} in a more flexible way,
by using suitable nonlinear functions of $\sfd_y$.
We state a general result.
\begin{lemma}
  \label{le:truncations}
  Let $I=(a,b)$ be an interval (possibly unbounded) of $\R$ and
  let $\zeta:\R\to \R$ be a Lipschitz and nondecreasing map satisfying
\begin{equation}
  \label{eq:153}
    \text{the restriction of $\zeta$ to $I$ is
      of class $\rmC^1$ with $\zeta'(s)>0$ if $s\in I$.}
\end{equation}
If $f:X\to \overline I$ is a Borel function, then the condition
\begin{equation}
  f\in D^{1, p  }( X,\sfd,\mm;\AA),\quad
  |\rmD f|_{\star,\AA}\le 1
  \label{eq:188}
\end{equation}
is equivalent
to 
  \begin{equation}
      \label{eq:2zeta}
      \zeta\circ f\in D^{1, p }( X,\sfd,\mm;\AA),\quad
      \big|\rmD (\zeta\circ f)\big|_{\star,\AA}(x)\le \zeta'(f(x))\quad\text{for $\mm$-a.e.~$x\in X$}.
    \end{equation}
\end{lemma}
\begin{proof}
  It is clear that if $|\rmD f|_{\star,\AA}\le 1$ then
  \eqref{eq:2zeta} holds, thanks to \eqref{eq:8}.
  In order to prove the converse implication, we consider a strictly
  decreasing sequence $a_n\downarrow a$, a strictly increasing
  sequence $b_n\uparrow b$ 
  and nondecreasing and bounded Lipschitz functions $\psi_n:\R\to
  \R$ 
  such that
  \begin{displaymath}
    \psi_n(z)=a_n\ \text{if }z<\zeta(a_n),\quad
    \psi_n(\zeta(s))=s\ \text{for every }s\in [a_n,b_n],\quad
    \psi_n(z)=b_n\ \text{if }z>\zeta(b_n).
  \end{displaymath}
  The
  restriction of $\psi_n$ to the interval $[\zeta(a_n),\zeta(b_n)]$ is of
  class $\rmC^1$.

  Setting $
   h(x):=\zeta(f(x))$,
   the Chain rule \eqref{eq:8} yields
  \begin{displaymath}
    |\rmD (\psi_n\circ h)|_{\star,\AA}(x)\le (\psi_n'\circ h)\, |\rmD
    h|_{\star,\AA}(x)\le
    (\psi_n'\circ \zeta(f(x))) \zeta'(f(x)).
  \end{displaymath}
  Since $\psi_n(h(x))= a_n\lor f(x)\land b_n$, 
  the locality property \eqref{eq:7}, the truncation property
  \ref{thm:omnibus}(8), and the fact that
  $\psi_{ n  }'(\zeta(s))\zeta'(s)=1$ if $s\in
  [a_n,b_n]$ yield
    \begin{equation}
      \label{eq:17}
      |\rmD (\psi_n\circ h)|_{\star,\AA}\le1 \quad\text{$\mm$-a.e.}
    \end{equation}
    Since $\psi_n\circ h\to f$ pointwise in $X$  as $n\to\infty$, passing to the limit
    in
    \eqref{eq:17} we get $|\rmD f|_{\star,\AA}\le 1$.
  \end{proof}
  \begin{remark}
      \label{rem:powers}
      Thanks to Lemma \ref{le:truncations},
      if $\sfd$ is a bounded metric and $q>1$,
  \eqref{eq:2} is equivalent to
  \begin{equation}
      \label{eq:2q}
      |\rmD \sfd^q_y|_{\star,\AA}(x)\le q\, \sfd_y^{q-1}(x)\quad\text{for $\mm$-a.e.~$x\in X$}.
    \end{equation}
    In particular, if \eqref{eq:2q} holds  for  some $q \ge 1$, it
    holds for any $q \ge 1$.
  \end{remark}
\subsection{A density result}
\label{subsec:density}
We have seen that in the present setting of Polish spaces,
condition \eqref{eq:2}
 (or, equivalently, \eqref{eq:2zeta} for some admissible truncation
 satisfying \eqref{eq:153})
 is a necessary condition for the validity of
 the approximation property
\eqref{eq:4bis} and of the identification $|\rmD f|_\star=|\rmD f|_{\star,\AA}$.
We want to show that \eqref{eq:2} or \eqref{eq:2zeta} are also 
\emph{sufficient} conditions. 

\begin{theorem}\label{theo:startingpoint}
  Let $(X,\sfd,\mm)$ be a Polish metric measure space,
  let $Y\subset X$ be a dense subset,
  and let $\AA$ be a unital separating subalgebra of $\Lipb(X)$
  as in \eqref{eq:116}.
  If 
\begin{equation}\label{eq:214-15}
  \text{for every } y \in { Y}
  \text{ it holds}\quad
  \sfd_y\in D^{1,p  }( X,\sfd,\mm;\AA),\quad
  \big|\rmD \sfd_y\big|_{\star,\AA}\le 1
\end{equation}
then
$\AA$ is dense in $p$-energy according to Definition
\ref{def:density}.
  \end{theorem}
  \begin{proof}
    We split the proof in various steps.
    Notice that by \eqref{eq:13}
    it is sufficient to prove that
      \begin{equation}
        \label{eq:15bis}
        |\rmD f|_{\star,\AA}
        \le |\rmD f|_{\star}\quad\text{$\mm$-a.e.~in $X$}.
      \end{equation}
      \smallskip \noindent
      (1) \emph{It is not restrictive
        to assume $\sfd$ bounded above by $1$ }: see Remark \ref{rem:trunc}.

      By Lemma \ref{le:truncations} and Remark \ref{rem:powers} we know that \eqref{eq:2q} holds
      for every $y\in Y$ and every $q \ge 1$ .

    \smallskip \noindent
    (2) \emph{It is sufficient to prove that}
      \begin{equation}
        \label{eq:92}
         \CE_{p,\AA}(f)\le \int_X (\lip f)^p\,\d\mm=\pCE_p(f)
        \quad \text{\em for every $f\in \Lip_b(X)$}.
      \end{equation}
    In fact, if $f$ has ($p,\Lip_b(X)$)-relaxed gradient, by
    \eqref{eq:4} we can find a sequence $f_n\in \Lip_b(X)$ such that
    $f_n\to f$ $\mm$-a.e.~and $\lip f_n\to |\rmD f|_\star$ strongly in
    $L^p(X,\mm)$
    as $n\to\infty.$
     By the $L^0$-lower semicontinuity of the
    $\CE_{p,\AA}$-energy, passing to the limit in \eqref{eq:92}
    written for $f_n$ we get
    \begin{displaymath}
      \CE_{p,\AA}(f)=\int_X |\rmD f|_{\star,\AA}^p\,\d\mm\le \int_X |\rmD f|_\star^p\,\d\mm=\CE_p(f)<\infty.
    \end{displaymath}
    We deduce that $f$ has a $(p,\AA)$-relaxed gradient and
    that \eqref{eq:15} holds, 
    since $|\rmD f|_{\star}\le |\rmD f|_{\star,\AA}$ $\mm$-a.e.

    \smallskip \noindent
    (3)
    For every $f\in \Lip_b(X)$ and $t>0$ we introduce the Hopf-Lax
    regularization $\sfQ_t f:X\to \R$ defined by 
    \begin{align}
        \label{eq:16}
      \sfQ_t f(x):={}&\inf_{y\in X}\frac{1}{ qt^{q-1}}\sfd^{ q }(x,y)+f(y),
                       \quad x\in X,
    \end{align}
     where $q \in(1,+\infty)$ is the conjugate exponent of $p$ i.e.~$1/q+1/p=1$.  
    It is clear that $\sfQ_t f$ is bounded (it takes values in the
    interval
    $[\inf_X f,\sup_X f]$) and Lipschitz, being the infimum of a family
    of  uniformly  Lipschitz functions.
    We consider the upper semicontinuous function
    \cite[(3.4) and Prop.~3.2]{AGS14I}
      \begin{align}
       \label{eq:18}
        \sfD_t^+ f(x):=&\sup_{(y_n)}\limsup_{n\to\infty} \sfd(x,y_n),
      \end{align}
      where the $(y_n)_n$'s vary among all the minimizing sequences 
      of \eqref{eq:16}.
      $\sfD_t^+ f$ is also uniformly bounded and satisfies
      (see e.g.~\cite[Lemma 3.2.1]{Savare22})
      \begin{equation}
        \label{eq:20}
         \left (  \frac{\sfD_t^+ f(x)}t  \right )^{q}  \le  \left ( q\Lip (f,X)  \right )^p.
      \end{equation}
      In fact, if $y_n$ is a minimizing sequence of \eqref{eq:16}, for
      every $\eps>0$ we eventually  have
      $$\frac{1}{ qt^{q-1}}\sfd^{ q }(x,y_n)+f(y_n)\le
      \sfQ_tf(x)+\eps\le f(x)+\eps$$
      i.e., setting $L:=\Lip(f,X)$,
      \begin{displaymath}
         \frac{1}{t^{ q}}\sfd^{ q }(x,y_n)\le \frac{\eps q}{t} + \frac{q}{t}(f(x)-f(y_n)) \le \frac{\eps q}{t} + qL \frac{\sfd(x,y_n)}{t} \le \frac{\eps q}{t} + (qL)^p + \frac{\sfd^{q}(x,y_n)}{qt^{q}p^{1/(p-1)}}.
      \end{displaymath}
      We thus get
      \begin{displaymath}
        \limsup_{n\to\infty}\frac{1}{ t^{q}}\sfd^{ q}(x,y_n)\le  \frac{\eps q}{t} + (qL)^p 
      \end{displaymath}
      which yields \eqref{eq:20} since $\eps>0$ is arbitrary.

      (4) \smallskip\noindent
      \emph{For every $f\in \Lip_b(X)$ and for every $t>0$
        \begin{equation}
          \label{eq:19}
          |\rmD \sfQ_t f|_{\star,\AA}(x)\le  \left (  t^{-1} \sfD_t^+ f(x)  \right )^{q-1}  \quad
          \text{for $\mm$-a.e.~$x\in X$}.
        \end{equation}
      }
       Let $Y'=\{y_n\}_{n\in \N}$ be a countable set dense in $Y$; 
      since $f\in \Lip_b(X)$ it is easy to check that
      \begin{equation}
        \label{eq:21}
        \sfQ_t f(x)=\inf_{y\in Y}\frac{1}{ q t^{ q-1}}\sfd^{ q}(x,y)+f(y)=
        \lim_{n\to\infty}\sfQ^n_tf(x),\quad
        \sfQ^n_t f(x):=\min_{1\le k\le n}\frac{1}{ q t^{ q-1}}\sfd^{ q}(x,y_k)+f(y_k).
      \end{equation}
       We consider now the upper semicontinuous   function
      \begin{equation}
        \label{eq:23}
        \sfD^n_t(x):= \max\Big\{\sfd(x,y_k):1\le k\le n,\ 
        \sfQ_t^n(x)=\frac{1}{ q t^{ q-1}}\sfd^{ q}(x,y_k)+f(y_k)\Big\}.
      \end{equation}
       By \eqref{eq:2q} and Theorem \ref{thm:omnibus}(8), we have that $ (  t^{-1}\sfD^n_t  )^{q-1} $ is  a $(p,\AA)$-relaxed gradient of $\sfQ^n_t f$.
      It is then clear that for every $x$ there exists a sequence
      $n\mapsto  y'(n,x)  $ with $ y'(n,x) \in \{y_1,\cdots,y_n\}$ such that 
      $\sfD^n_t(x)= \sfd(x, y'(n,x))$  and
      $\sfQ^n_tf(x)=\frac{1}{ q t^{ q-1}}\sfd^{ q}(x, y'(n,x) )+f( y'(n,x) )\to \sfQ_t f(x)$
      as $n\to\infty$, i.e.~$ y'(n,x) $ is a minimizing sequence of \eqref{eq:16}.
      We deduce that
      \begin{equation}
        \label{eq:24}
        \limsup_{n\to\infty}\sfD^n_t(x)=
         \limsup_{n\to\infty}\sfd(x, y'(n,x) )\le \sfD_t^+f(x)  \text{ for every } x \in X.
      \end{equation}
      Since $\sfD^n_t f $ are uniformly bounded, up to extracting a
      suitable subsequence we can suppose that
      $  (  t^{-1}  \sfD^n_t )^{ q-1} \weakto^* G$ weakly* $L^\infty(X,\mm)$  so that, by Theorem \ref{thm:omnibus}(1), $G$
      is a $(p,\AA)$-relaxed gradient of $\sfQ_t f$, hence $|\rmD \sfQ_t f|_{\star,\AA}\le G$ $\mm$-a.e.~by Theorem \ref{thm:omnibus}(3). Also notice that by Fatou's lemma and weak* $L^\infty(X,\mm)$ convergence, we have
      \[ \int_B G \, \d \mm = \lim_{n \to + \infty} \int_B (t^{-1}\sfD^n_t)^{q-1} \, \d \mm \le \int_B \limsup_{n \to + \infty} (t^{-1}\sfD^n_t(x))^{q-1} \, \d \mm(x) \le \int_B (t^{-1}\rmD_t^+ f(x))^{q-1} \, \d \mm (x),\]
      for every Borel set $B \subset X$. We conclude that $|\rmD \sfQ_t f|_{\star,\AA}\le (t^{-1}\rmD_t^+ f(x))^{q-1}$ for $\mm$-a.e.~$x \in X$.

      (5) \smallskip\noindent \emph{For every $x\in X,\ t>0,$ and $f\in
        \Lip_b(X)$
        we have
        \begin{align}
          \label{eq:25}
          \frac{f(x)-\sfQ_tf(x)}t&=\frac 1{ p}\int_0^1
                                   \Big(\frac{\rmD_{rt}^+f(x)}{rt}\Big)^{ q}
                                   \,\d r,\\
          \label{eq:27}
          \limsup_{t\downarrow0}\frac{f(x)-\sfQ_tf(x)}t&\le \frac1{ p}\big(\lip f(x)\big)^{ p}.
        \end{align}
      }
      This follows by \cite[Thm. 3.2.4]{Savare22}
      (see also \cite[Thm.~3.1.4, Lemma 3.1.5]{AGS08}).

      ( 6) \smallskip\noindent \emph{Conclusion.}
      We argue as  in  \cite[Theorem 3.2.7]{Savare22}:
      \eqref{eq:25} and \eqref{eq:20} yield the uniform bound
      \begin{equation}
        \label{eq:26}
        \frac{f(x)-\sfQ_tf(x)}t\le  \frac{1}{p}  \big(  q  \Lip(f,X)\big)^{ p}\quad
        \text{for every }x\in X,\ t>0.
      \end{equation}
      Integrating \eqref{eq:27} in $X$ 
      and applying Fatou's Lemma we get
      \begin{equation}
        \label{eq:28}
        \limsup_{t\downarrow0}
        \int_{X} \frac{f(x)-\sfQ_tf(x)}{t}\,\d\mm(x)\le
        \frac1{ p}\int_{ X} \big(\lip f(x)\big)^{ p} \,\d\mm(x).
      \end{equation}
      On the other hand, \eqref{eq:25} and Fubini's Theorem yield
      \begin{equation}
        \label{eq:29}
        \int_{ X}  \frac{f(x)-\sfQ_tf(x)}{t}\,\d\mm(x)
        =\frac 1{ p} \int_0^1 \int_{ X}  \Big(\frac{\sfD_{rt}^+f(x)}{rt}\Big)^{ q}
        \,\d \mm(x)\,\d r.
      \end{equation}
      A further application of Fatou's Lemma yields
      \begin{equation}
        \label{eq:29bis}
        \liminf_{t\downarrow0}\int_{ X}  \frac{f(x)-\sfQ_tf(x)}{t}\,\d\mm(x)
        \ge  \frac 1{ p}\liminf_{t\downarrow0} \int_{ X}  \Big(\frac{\sfD_{t}^+f(x)}{t}\Big)^{ q}
        \,\d \mm(x).
      \end{equation}
      Using the fact that
      $t^{-1}\sfD_t^+f$ is uniformly bounded by \eqref{eq:20}, we can
      find a decreasing and vanishing sequence
      $ n\mapsto t(n)$ and a limit function $G\in
      L^\infty(X,\mm)$
      such that
      \begin{gather}
          \left (  t(n)^{-1}\sfD_{t(n)}^+f  \right )^{q-1}  \weakto^* G\quad \text{weakly$^*$ in
           $L^\infty(X,\mm)$}\quad \text{as }n\to\infty,\notag\\
         \label{eq:117}
        \lim_{n\to\infty}
        \int_{ X}  \Big(\frac{\sfD_{t(n)}^+f(x)}{t(n)}\Big)^{ q}
        \,\d \mm(x)=\liminf_{t\downarrow0} \int_{ X}  \Big(\frac{\sfD_{t}^+f(x)}{t}\Big)^{ q}
        \,\d \mm(x).
      \end{gather}
      Since $ \left (  t^{-1}\sfD_t^+f  \right )^{q-1} $  is  a $(p,\AA)$-relaxed gradient of
      $\sfQ_t f$ by claim (4) and $\sfQ_t f\to f$ pointwise
      everywhere, using Theorem \ref{thm:omnibus}(1) we get that $G$ is
      a $(p,\AA)$-relaxed gradient of $f$. \\
      Using the lower semicontinuity of the
       $L^{ p}$-norm w.r.t.~the weak$^*$ $L^\infty(X,\mm)$
      convergence, we get that
      \begin{align}
        \label{eq:30}
        \lim_{n\to\infty} \int_{ X} \Bigg(\frac{\sfD_{t(n)}^+f(x)}{t(n)}\Bigg)^{ q}
        \,\d \mm(x) &=\lim_{n\to\infty} \int_{ X} \Bigg(\frac{\sfD_{t(n)}^+f(x)}{t(n)}\Bigg)^{ p(q-1)}
        \,\d \mm(x)\\
        &\ge \int_{ X} G^{ p} \,\d\mm(x) \\
        &\ge 
        \int_{ X} |\rmD f|_{\star,\AA}^{ p}(x)\,\d\mm(x),
      \end{align}
      where we also used the pointwise minimality of  $|\rmD f|_{\star,\AA}$  given by Theorem \ref{thm:omnibus}(3). 
      Combining \eqref{eq:30}, \eqref{eq:117}, \eqref{eq:29bis} and
      \eqref{eq:28} 
      we deduce that
      \begin{displaymath}
        \int_{ X}  |\rmD f|_{\star,\AA}^{ p}(x)\,\d\mm(x)\le \int_{ X}  \big(\lip
        f(x)\big)^{ p}\,\d\mm(x) 
      \end{displaymath}
      so that \eqref{eq:92} holds. 
  \end{proof}

\begin{corollary}[Density in energy of $\AA$ in $H^{1,p}$]
  \label{cor:density}
  If $\AA$ is a separating unital subalgebra of $\Lip_b(X)$ satisfying
   \eqref{eq:214-15} 
  then 
  \begin{equation}
    \label{eq:22}
    \CE_{p,\AA}(f)= \CE_{p}(f)=\NE_p(f)
    \quad
    \text{for every $\mm$-measurable function $f:X\to \R$.}
  \end{equation}
  In particular, $H^{1,p}(X,\sfd,\mm)=H^{1,p}(X,\sfd,\mm;\AA)$.
\end{corollary}

As we have already said, \eqref{eq:22} can be interpreted as a density
result in $H^{1,p}(X,\sfd,\mm)$: for every $f\in H^{1,p}(X,\sfd,\mm)$
there exists a sequence $f_n\in \AA$, $n\in \N$, such that
\begin{equation}
  \label{eq:33}
  f_n\to f,\
  \lip f_n\to |\rmD f|_*\quad \text{strongly in }L^p(X,\mm),\quad
  \int_X |\lip f_n|^p\,\d\mm\to \CE_p(f)\quad\text{as }n\to\infty.
\end{equation}

\subsection{Applications}
\label{subsec:app1}
We first recall a useful result showing that it is possible to remove the assumption that
$\AA$ is unital, if $\AA$ satisfies a suitable tightness condition.
We will denote by $\mathbbm 1$ the unit constant function.
\begin{proposition}
  \label{prop:unital-not-necessary}
  Let $\AA\subset \Lip_b(X)$ be a separating subalgebra of Lipschitz
  functions
  and let 
  \begin{equation}
    \label{eq:155}
    \AA_1:=\AA\oplus\{c\mathbbm 1\}=\Big\{f+c\mathbbm 1:f\in \AA,\, c\in \R\Big\}
  \end{equation}
  be the minimal unital subalgebra containing $\AA$.
  If $\AA_1$ is dense in $p$-energy and
  there exist a sequences of compact sets $K_n\subset X$ and functions
  $f_n\in \AA$ such that
  \begin{equation}
    \label{eq:146}
    f_n(x)\ge 1\ \text{for every $x\in K_n$},\quad
    \lim_{n\to\infty}\int_{X\setminus K_n}\Big(1+|\lip f_n(x)|^p\Big)\,\d\mm(x)=0,
  \end{equation}
  then $\AA$ is dense in $p$-energy as well.
\end{proposition}
The \emph{proof} is a simple adaptation of \cite[Proposition
5.3.2]{Savare22}.
The next result shows that the algebra
generated
by (suitable compositions/truncations of) distance functions
is always sufficient to generate the Sobolev space $H^{1,2}(X,\sfd,\mm)$.
\begin{theorem}
  \label{cor:Gigli-inspired}
  Let $Y$ be a dense subset of $X$ 
  and
  let $\zeta:[0,+\infty)\to [0,+\infty)$ be a Lipschitz nondecreasing
  function
  such that $\zeta'>0$ in an interval $I=(0,r)\subset (0,+\infty)$  and $\zeta \in \rmC^1(I)$. 
  Then the unital algebra $\AA$ generated by the functions $
  x\mapsto
  \zeta(\sfd(x,y))$ 
  is dense in $p$-energy.
\end{theorem}
\begin{proof} 
Thanks to Remark \ref{rem:trunc}, we can assume that $\sfd$ is bounded above by $r$.  
  It is not difficult to check that $\AA$ separates the points of
  $X$, so that 
  in order to apply Theorem \ref{theo:startingpoint}, it is enough to check
  that \eqref{eq:2zeta} with $f:= \sfd_y$ (recall the notation \eqref{eq:14}) holds. 
  
  Such a property follows immediately from the corresponding estimate
  on the asymptotic Lipschitz constant: for every $y\in Y$ and 
  $g(x):=\zeta(\sfd_y(x))$,
  a simple direct computation shows that
  \begin{displaymath}
    \lip g(x)\le
    \zeta'(\sfd_y(x))
     \quad\text{for every }x\in \rmB(y,r).
  \end{displaymath}
  Since $g\in \AA$ we have $|\rmD g|_{\star,\AA}\le \lip g
  \le \zeta'(f)$, so that \eqref{eq:2zeta} holds and we conclude by applying Lemma \ref{le:truncations}. 
\end{proof}
We now consider a simple application
of Theorem \ref{theo:startingpoint} to the
case when $p=2$ and $\lip f$ has good properties for functions of
$\AA$.
\begin{theorem}[A Hilbertianity condition]\label{thm:hilb}
  Let $p=2$ and let $\AA$ be a separating unital subalgebra of
  $\Lip_b(X)$ satisfying \eqref{eq:214-15}. 
  If for every $f,g\in \AA$
  \begin{equation}
    \label{eq:31}
    \int_X \Big(|\lip (f+g)|^2+|\lip (f-g)|^2\Big)\,\d\mm=2
    \int_X \Big(|\lip f|^2+|\lip g|^2\Big)\,\d\mm,
  \end{equation}
  then $H^{1,2}(X,\sfd,\mm)$ is a Hilbert space, $\CE_2$ is a Dirichlet (thus quadratic)
  form, and
  $\AA$ is strongly dense.
\end{theorem}
\begin{proof}
  It is sufficient to prove that
  the Cheeger energy is a quadratic form in its domain.
  Thanks to \cite[Prop.~11.9]{DalMaso93} and the $2$-homogeneity of
  $\CE_2$,
  this property is equivalent
  to 
  \begin{equation}
    \label{eq:32}
    \CE_2(f+g)+\CE_2(f-g)\le 2\CE_2(f)+2\CE_2(g) \quad
    \text{for every } f,g \in H^{1, 2 }(X,\sfd,\mm). 
  \end{equation}
  We can find two sequences $f_n,g_n\in \AA$ such that
  $f_n\to f,\ g_n\to g$ in $\mm$-measure as $n\to \infty$ and
  $\lip f_n\to |\rmD f|_\star$, $\lip g_n\to |\rmD g|_\star$ in
  $L^2(X,\mm)$.
  Clearly we have $f_n+g_n\to f+g$, $f_n-g_n\to f-g$ in $\mm$-measure
  and \eqref{eq:31} shows that $\lip(f_n+g_n)$ and $\lip(f_n-g_n)$ are
  uniformly bounded in $L^2(X,\mm)$.
  Up to extracting a suitable sequence, it is not restrictive to
  assume that $\lip(f_n+g_n)\weakto G_+\ge |\rmD (f+g)|_\star$  and
  $\lip(f_n-g_n)\weakto G_-\ge |\rmD (f-g)|_\star$ $\mm$-a.e.~in $X$. 
  \eqref{eq:31} then yields
  \begin{align*}
    \CE_2(f+g)+\CE_2(f-g)&=\int_X |\rmD (f+g)|_\star^2\,\d\mm
    +\int_X |\rmD (f-g)|_\star^2\,\d\mm
                           \\&\le 
      \liminf_{n\to\infty}
      \int_X |\lip (f_n+g_n)|^2\,\d\mm
      +\int_X |\lip (f_n-g_n)|^2\,\d\mm
    \\&=
    \liminf_{n\to\infty}
    2\int_X |\lip f_n|^2\,\d\mm
    +2\int_X |\lip g_n|^2\,\d\mm=
    2\CE_2(f)+\CE_2(g).
  \end{align*}
  Since $H^{1,2}(X,\sfd,\mm)$ is Banach space,
  we deduce that $H^{1,2}(X,\sfd,\mm)$ is a Hilbert space, so it is
  reflexive.
  This also shows that $\AA$ is strongly dense.
\end{proof}
\begin{remark} \label{rem:prec} In the framework of Theorem \ref{thm:hilb}, there exists a scalar product $\la \cdot, \cdot \ra_{H^{1,2}}$ on $H^{1,2}(X,\sfd,\mm)$ inducing the norm $\|\cdot\|_{H^{1,2}}$ and satisfying 
  \begin{equation}\label{eq:ce2sp}
    \la f, g \ra_{H^{1,2}} = \int_X fg \, \d \mm + \CE_2(f,g) \text{ for every } f,g \in H^{1, 2  }(X,\sfd,\mm),
  \end{equation}
  where $\CE_2(\cdot, \cdot)$ denotes the bilinear form associated to
  $\CE_2(\cdot)$.
\end{remark}
\begin{remark}
  \label{rem:closable}
  If \eqref{eq:31} holds then the restriction $(\pCE_2,\AA)$ of $\pCE_2$ to $\AA$
  is
  a quadratic form which is induced by a corresponding bilinear form
  $\pCE_2(\cdot,\cdot)$ defined by the parallelogram rule.
  We recall that such a form is \emph{closable} (see e.g.~\cite[\S
  1.3]{Bouleau-Hirsch91}, \cite[Chapter I, \S.3]{Ma-Rockner92}) if for any sequence $(f_n)_{n\in \N}$ in $\AA$
  \begin{equation}
    \label{eq:35}
    f_n\to 0\quad\text{in }L^2(X,\mm),\quad
    \limsup_{m,n\to\infty}\pCE_2(f_n-f_m)=0\quad
    \Rightarrow\quad
    \lim_{n\to\infty}\pCE_2(f_n)=0.
  \end{equation}
  Theorem \ref{thm:hilb} shows in particular that if
  $(\pCE_2,\AA)$ is quadratic and closable,
  then the Cheeger energy $(\CE_2,H^{1,2}(X,\sfd,\mm))$ coincides with the smallest closed
  extension of $(\pCE_2,\AA)$. In this case, trivially, the restriction of
  $\CE_2$ to $\AA$ coincides with $\pCE_2$.
  Since the Cheeger energy $\CE_2$ is quasi-regular (see 
  \cite[Lemma 6.7]{AGS11b}, \cite[Thm.~4.1]{Savare14},
  \cite[Prop.~3.21]{DelloSchiavo-Suzuki21},
  \cite{Savare23}), as a by-product we obtain the quasi-regularity of
  the closure of $(\pCE_2,\AA)$.
\end{remark}
An immediate consequence is the Hilbertianity of
$H^{1,2}(\H,\sfd_\H,\mm)$ in the case when $(\H,\sfd_\H)$ is a
separable Hilbert space (in particular $\R^d$) endowed with the
distance induced by its Hilbertian norm  \cite{DGPS21, DLP20,Savare22}.
\begin{corollary}
  \label{cor:Hilbert}
  Let $(\H,\sfd_\H)$ be a separable Hilbert space and let $\mm$ be a
  finite and positive Borel measure on $\H$.
  Then $H^{1,2}(\H,\sfd_\H,\mm)$ is a Hilbert space.
\end{corollary}
\begin{proof}
  Let $\AA$ be the algebra $\rmC^1_b(\H)$ of bounded $\rmC^1$
  functions with bounded continuous gradient. It is
  immediate to check that for every $\phi\in \rmC^1_b(\H)$ we have 
  $\lip \phi(x)=\|\nabla \phi(x)\|_\H$ so that $\pCE_2$ is a quadratic
  form on $\AA$, thus satisfying \eqref{eq:31}.

  On the other hand $\AA$ contains the functions
  $x\mapsto \tanh(\sfd^2(x,y))$, $y\in \H$, so that we can apply
  Theorem \ref{cor:Gigli-inspired}.
\end{proof}
\begin{remark}[Density of $\rmC^\infty_c(\R^d)$ in $H^{1,2}(\R^d,\sfd,\mm)$]
  \label{rem:serve-tutto}
  When $\H=\R^d$ is finite dimensional,
  we can also prove that the algebra $\AA=\rmC^\infty_c(\R^d)$
  is strongly dense in $H^{1,2}(\R^d,\sfd,\mm)$.
  In fact, if $\zeta$ is the restriction to $[0,\infty)$ of a smooth
  nondecreasing
  transition function $\tilde \zeta\in \rmC^\infty(\R)$ satisfying $\tilde \zeta(s)=0$ if $s\le 0$,
  $\tilde \zeta(s)=1$ is $s>1$ and $\tilde \zeta'(s)>0$ if $s\in
  (0,1)$,
  it is immediate to check that for every $y\in \R^d$ the functions
  $\tilde \zeta(\sfd_y)$ belong to $\AA_1$, so that $\AA_1$ is dense in
  $2$-energy by Theorem \ref{cor:Gigli-inspired}.
  
  On the other hand, being $\mm$ tight, it is easy to check that $\AA$
  satisfies \eqref{eq:146}, so that we can apply Proposition \ref{prop:unital-not-necessary}.
\end{remark}
\subsection{Intrinsic distances}
\label{subsec:intrinsic1}
By using the general properties of metric Sobolev spaces
 and the equivalence with the Newtonian viewpoint based on the
notion of upper gradient \cite{Bjorn-Bjorn11,HKST15} it is
possible to improve considerably the density result of Corollary
\ref{cor:density}.
Let us first recall the notion of \emph{metric velocity}
\begin{equation}
  \label{eq:118}
        |\dot\gamma|_\sfd(t):=\limsup_{h\to0}\frac{\sfd(\gamma(t+h),\gamma(t))}{|h|}
\end{equation}
and \emph{length} 
  \begin{equation}
    \label{eq:34}
    \begin{aligned}
      \ell_{\sfd}(\gamma,[\alpha,\beta]):={}&
      \sup\Big\{\sum_{n=1}^N\sfd(\gamma(t_{n-1}),\gamma(t_n)):
      t_0=\alpha<t_1<\cdots<t_{N-1}<t_N=\beta\Big\}
      \\&=\int_\alpha^\beta|\dot\gamma|_\sfd(t)\,\d t
    \end{aligned}
  \end{equation}
  of a
  $\sfd$-Lipschitz curve $\gamma:[a,b]\to X$; here $[\alpha,\beta]\subset
  [a,b]$ and we just write $\ell_\sfd(\gamma)$ for
  $\ell_\sfd(\gamma,[a,b])$.
 
   If $Y\subset X$ is a given set, we can introduce the \emph{length (or intrinsic) extended
    distance} $\sfd_{Y,\ell}$ induced by $\sfd$ on $Y$, as the infimum of the
  length of $Y$-valued Lipschitz curves connecting two given points
  $y_0,y_1\in Y$:
    \begin{align}
    \label{eq:40}
    \sfd_{Y,\ell}(y_0,y_1):={}&
    \inf\Big\{\ell_\sfd(\gamma):\gamma\in \Lip([0,1];(Y,\sfd)),\
                                \gamma(0)=y_0,\ \gamma(1)=y_1\Big\}\\
      \label{eq:40bis}={}&
       \inf\Big\{\ell>0:\gamma\in \Lip([0,\ell];(Y,\sfd)),\
                                \gamma(0)=y_0,\ \gamma(\ell)=y_1,\
      |\dot\gamma|_\sfd\le 1\text{ a.e.}\Big\}.
  \end{align}
  Clearly we have
  \begin{equation}
    \label{eq:138}
    \sfd(y_0,y_1)\le \sfd_{X,\ell}(y_0,y_1)\le
    \sfd_{Y,\ell}(y_0,y_1)\quad
    \text{for every }y_0,y_1\in Y.
  \end{equation}
If $g:X\to [0,+\infty]$
  is a Borel function, the integral of $g$ along $\gamma$ is defined by
  \begin{equation}
    \label{eq:119}
    \int_\gamma g:=\int_a^b g(\gamma(t))|\dot \gamma|_\sfd(t)\,\d t.
  \end{equation}
  It is well known that length and integral are invariant with respect
  to arc-length reparametrization of $\gamma$ and it is always possible to find
  a $1$-Lipschitz curve $R_\gamma:[0,\ell_\sfd(\gamma)]\to X$ such
  that
  \begin{equation}
    \label{eq:121}
    R_\gamma(\ell_\sfd(\gamma,[a,t]) )=\gamma(t) \text{ for every }t\in [a,b],\quad
    |\dot R_\gamma|(s)=1\text{ a.e.~in $[0,\ell_\sfd(\gamma)]$},\quad
    \int_{R_\gamma}g=\int_\gamma g
  \end{equation}
  for every nonnegative Borel function $g$ (see e.g.~\cite[Section 3.3]{Savare22}).
  A Borel function $g:X\to [0,+\infty]$ is an upper gradient of 
  $f:X\to \R$ if 
  \begin{equation}
    \label{eq:120}
    |f(\gamma(b))-f(\gamma(a))|\le \int_\gamma g
    \quad \text{for every $\gamma\in
  \Lip([a,b];(X,\sfd))$.}
\end{equation}
Functions in $\mathcal L^p(X,\mm)$ which  admit  an upper gradient in $\mathcal L^p(X,\mm)$
characterize the Newtonian Sobolev space $N^{1,p}(X,\sfd,\mm)$
\cite{Bjorn-Bjorn11,HKST15}.
We state here a useful consequence of the main equivalence results
\cite[Theorem 6.2]{AGS14I} \cite[Theorem 7.4]{AGS13}.
\begin{theorem}
  \label{thm:usefulN}
  Let $Y$ be a
   Borel subset of $X$ of full $\mm$-measure (i.e.~$\mm(X\setminus
  Y)=0$) satisfying
  \begin{equation}
    \label{eq:139}
    \gamma\in \Lip([a,b];(X,\sfd)),\quad
    R_\gamma(s)\in Y\ \text{for $\Leb 1$-a.e.~$s\in [0,\ell_\sfd(\gamma)]$}\quad
    \Rightarrow
    \quad
    \gamma([a,b])\subset Y,
  \end{equation}
  let $f:X\to \R$ be a $\mm$-measurable function and let
  $g:Y\to[0,+\infty]$ be a Borel function satisfying
    \begin{equation}
    \label{eq:120bis}
    |f(\gamma(b))-f(\gamma(a))|\le \int_\gamma g
    \quad \text{for every $\gamma\in
      \Lip([a,b];(Y,\sfd))$}.
\end{equation}
If $\displaystyle\int_Y |g|^p\,\d\mm<\infty$ then $f$ has a $p$-relaxed gradient and
  \begin{equation}
    \label{eq:122}
    |\rmD f|_\star\le g\quad\text{$\mm$-a.e.~in }Y.
  \end{equation}
\end{theorem}
Notice that condition \eqref{eq:120bis} is weaker than \eqref{eq:120},
since the upper gradient condition is imposed only along curves taking
values in $Y$; however, starting from any
function $g\in \mathcal L^p(Y,\mm)$ satisfying \eqref{eq:120bis} we
can define a new Borel function $\tilde g:X\to[0,+\infty]$ whose
restriction to $Y$ coincides with $g$ such that $\tilde
g\restr{X\setminus Y}\equiv +\infty$. Clearly
$$\int_X \tilde
g^p\,\d\mm=
\int_Y g^p\,\d\mm<+\infty\quad\text{since }\mm(X\setminus Y)=0.$$
Moreover $\tilde g$ is an upper gradient for $f$
according to \eqref{eq:120}: in fact it is sufficient to check
\eqref{eq:120} for those curves $\gamma$ with $\gamma=R_\gamma$ and $\displaystyle
\int_\gamma \tilde g<+\infty$; since $\tilde g(\gamma(s))=+\infty$
if $\gamma(s)\not\in Y$, we deduce that $\gamma(s)\in Y$ for
$\Leb{1}$-a.e.~$s\in [0,\ell_\sfd(\gamma)]$ so that $\gamma\in
\Lip([0,\ell_\sfd(\gamma)];(Y,\sfd))$ by \eqref{eq:139},
and \eqref{eq:120} then follows by \eqref{eq:120bis}.

It is also immediate to check that \eqref{eq:139} holds if $Y$ is closed.

We consider the situation where
\begin{enumerate}[(A)]
\item  $Y\subset X$ is a Borel set with full $\mm$-measure
  satisfying \eqref{eq:139};
\item
  a metric 
  $\dY:Y\times Y\to [0,+\infty)$
  is given on $Y$ such that $(Y,\dY)$ is complete and separable
  and 
  (recall Remark \ref{rem:trunc})
    \begin{equation}
    \label{eq:39}
    \sfd_1(y_1,y_2)\le \dY(y_1,y_2)\le
    \sfd_{Y,\ell}(y_1,y_2)\quad\text{for every }y_1,y_2\in Y.
  \end{equation}
\end{enumerate}

\begin{remark}[$Y$-intrinsic distance]\label{rem:equivcond}
$\delta$ is intrinsically equivalent to $\sfd$ on $Y$, i.e.~every
  $\sfd$-Lipschitz curve $\gamma:[0,1]\to Y$ is also
  $\dY$-Lipschitz, its $\dY$-length
  coincides with the corresponding $\sfd$-length, and integration
  along $\gamma$ does not depend on the choice of the distance.  In particular condition \eqref{eq:120bis} can be equivalently stated
  in terms of $\dY$.\\
   To see that these conditions are implied by \eqref{eq:39}, let us fix a $\sfd$-Lipschitz curve $\gamma:[0,1] \to Y$ with Lipschitz constant bounded by $L\ge0$; then
  \[ \sfd_{Y, \ell}(\gamma(s), \gamma(t)) \le \ell_\sfd \left (\gamma|_{[s,t]} \right ) = \int_s^t |\dot{\gamma}|_\sfd(r) \, \d r \le L|t-s| \quad 0 \le s \le t \le 1,\]
  so that $\gamma$ is $\sfd_{Y, \ell}$-Lipschitz continuous and thus, by \eqref{eq:39}, also $\dY$-Lipschitz continuous. To see that the $\dY$ and the $\sfd$-lengths of $\gamma$ coincide, it is enough to show that $\ell_\dY(\gamma) \le \ell_\sfd(\gamma)$, since \eqref{eq:39} and the trivial equality $\ell_{\sfd_1}(\gamma)= \ell_{\sfd}(\gamma)$ already give the other inequality; by $\eqref{eq:39}$ we immediately have $\ell_{\dY}(\gamma) \le \ell_{\sfd_{Y,\ell}}(\gamma)$ and by the very definition of $\sfd_{Y, \ell}$ we see that $\ell_{\sfd_{Y,\ell}}(\gamma) \le \ell_{\sfd}(\gamma)$. Finally, to see that the integral along $\gamma$ does not depend on the choice of the distance, it is enough to see that $|\dot{\gamma}|_\sfd = |\dot{\gamma}|_\delta$ a.e.~in $[0,1]$. The $\le$ inequality is an immediate consequence of \eqref{eq:39} and \eqref{eq:118}, while the $\ge$ follows by
  \[ \frac{\delta(\gamma(s), \gamma(t))}{t-s} \le \frac{\ell_\delta(\gamma|_{[s,t]})}{t-s} = \frac{\ell_\sfd(\gamma|_{[s,t]})}{t-s} = \frac{1}{t-s}\int_s^{t} |\dot{\gamma}|_\sfd(r) \, \d r \quad 0 \le s < t \le 1, \]
  and passing to the limit as $s \to t$ for every Lebesgue point $t$ of $|\dot{\gamma}|_\sfd$.
\end{remark}
 Since $\mm(X\setminus Y)=0$
we can identify $L^p(Y,\mm)$ with $L^p(X,\mm)$.
In general, the topology induced by $\dY$ is finer than the $\sfd$
topology on $Y$, and they coincide if $\dY$ is continuous
w.r.t.~$\sfd$.
It is also clear from property  (B) that the restriction to $Y$ of every
bounded $\sfd$-Lipschitz function
$f:X\to \R$ is also $\dY$-Lipschitz.
Thanks to \eqref{eq:39}
 (which in particular implies that
$\delta$-balls of radius $r<1$  centered at some point $y \in Y$  are included in $\sfd$-balls of the
same radius and with the same center)
it is also clear that
\begin{equation}
  \label{eq:44}
   \lip_{\dY}f(y)\le \lip_\sfd f(y)\quad\text{for every }y\in Y,\ f\in \Lip_b(X,\sfd).
\end{equation}
Since $\lip_{\dY}f$ is bounded and $\delta$-u.s.c.~in $Y$,
it is $\mm$-measurable and
we can define the $\dY$ pre-Cheeger energy
\begin{equation}
  \label{eq:37}
  \pCE_{p,\dY}(f):=\int_Y |\lip_{\dY}f(y)|^p\,\d\mm(y)
\end{equation}
and we can still consider its l.s.c.~envelope in $L^{ 0}(Y,\mm)$
\begin{equation}
  \label{eq:45}
  \CE_{p,\dY,\AA}(f):=\inf\Big\{\liminf_{n\to\infty}\pCE_{p,\dY}(f_n):
  f_n\in \AA,\ f_n\to f\ \text{in }L^{ 0}(X,\mm)\Big\}.
\end{equation}
\begin{theorem}
  \label{thm:identification2}
   Let $\AA(X, \sfd):=\Lip_b(X, \sfd)$, let $\AA$ be  a separating unital subalgebra of
  $\Lip_b(X,\sfd)$ satisfying \eqref{eq:214-15}
  and  assume that  $(Y,\delta)$ satisfies the conditions {\upshape (A), (B)} above.  Then 
   we have
  \begin{equation}
    \label{eq:46}
    \CE_{p,\dY, \AA(X, \sfd)}(f)=\CE_{p,\dY,\AA}(f)=\CE_{p,\AA}(f)=\CE_p(f)  \quad \text{ for every } f \in L^0(X,\mm).
  \end{equation}
  In particular, the minimal $p$-relaxed gradients of $f \in
  L^0(X, \mm)$ computed w.r.t.~$(\delta, \AA)$,
  $(\delta, \Lipb(Y))$,
  $(\sfd, \AA)$ or $(\sfd, \Lipb(X))$ coincide and we have
  $  D^{1,p}(Y,\dY,\mm) =D^{1,p}(Y,\dY,\mm;\AA)=
  D^{1,p}(X,\sfd,\mm)=D^{1,p}(X,\sfd,\mm;\AA).$ 
\end{theorem}
\begin{proof}
    Since $\pCE_{p,\dY}(f)\le \int_{X}\big(\lip_\sfd
    f(x)\big)^p\,\d\mm$ for every $f\in \Lip_b(X,\sfd)$, we clearly have
  \begin{displaymath}
    \CE_{p,\dY, \AA(X, \sfd)}(f)\le\CE_{p,\dY,\AA}(f)\le\CE_{p,\AA}(f)=\CE_p(f)\quad\text{for
      every $f\in L^{ 0}(X,\mm)$},
  \end{displaymath}
  where the last equality follows from Corollary \ref{cor:density}. It is then sufficient to prove that  $\CE_{p,\dY, \AA(X, \sfd)}(f)\ge \CE_p(f)$ in
  order to get \eqref{eq:46}. Using \eqref{eq:45} and the $L^{ 0}(X,\mm)$-lower
  semicontinuity of $\CE_p$
  (see Remark \ref{rem:relprec}), the latter inequality will be  a consequence of
  \begin{equation}
    \label{eq:47}
    \int_Y |\lip_{\dY}f(y)|^p\,\d\mm(y)\ge \CE_p(f)\quad\text{for
      every }f\in \Lip_b(X,\sfd).
  \end{equation}
   In order to prove \eqref{eq:47} it is sufficient to apply
  Theorem \ref{thm:usefulN} and prove that the Borel function
  $g:=\lip_{\dY} f$
  satisfies \eqref{eq:120bis}.
  Now we use the fact that the restriction to $Y$ of a function $f\in\Lip_b(X, \sfd)  $ belongs to $\Lip_b(Y,\dY)$
  and every $\sfd$-Lipschitz curve $\gamma$ with values in $Y$ is also
  $\dY$-Lipschitz, the respective  lengths  coincide and therefore
  also the arc-length reparametrizations are the same.
  Since $\lip_\delta$ is an upper gradient we thus obtain

  \begin{displaymath}
    |f(\gamma(b))-f(\gamma(a))|\le \int_\gamma  \lip_\dY f \quad  \text{ for every } \gamma \in \Lip([a,b]; (Y, \delta)).
    \qedhere
  \end{displaymath}
  By Theorem \ref{thm:usefulN} and also using that $\mm(X \setminus Y)=0$, we conclude.  
\end{proof}
Combining Theorem \ref{thm:identification2} with Corollary
\ref{cor:Hilbert} we recover the following result of \cite{LP20}.
\begin{corollary}
  \label{cor:Riemann}
  Let $(\M,\sfd_\M)$ be a complete Riemannian manifold endowed with
  the canonical Riemannian distance and let $\mm$ be a
  finite and positive Borel measure on $\M$.
  Then $H^{1,2}(\M,\sfd_\M,\mm)$ is a Hilbert space and
  $\rmC^\infty_c(\M)$ is dense in $H^{1,2}(\M,\sfd_\M,\mm)$.
\end{corollary}
\begin{proof}
  By Nash isometric embedding Theorem \cite{Nash54}
  we can find a dimension $d$,
  and an isometric embedding
  $\jmath:\M\to
  \jmath(\M)\subset \R^d$.
  
   Since $\M$ is complete and $\jmath$ is an
  imbedding, $M:=\jmath(\M)$ is a closed subset of $\R^d$ and the (Riemannian) metric $\sfd_M$ inherited by $\sfd_\M$ given by $\sfd_M(\jmath(x),\jmath(y)):=\sfd_\M(x,y)$ is an isometry. In particular $\sfd_M$ induces on $M$ the relative topology of $\R^d$ and $(M, \sfd_M)$ is a complete and separable metric space. 
  Setting $\tilde \mm:=\jmath_\sharp \mm$, 
  it is clear that the map $\jmath^*:f\to f\circ \jmath$
  is a linear isometric isomorphism between
  $H^{1,2}(M,\sfd_M,\tilde\mm)$
  and $H^{1,2}(\M,\sfd_\M,\mm)$.
  It is then sufficient to prove the statement for $H^{1,2}(M,\sfd_M,\tilde\mm)$.

  We can now apply Theorem \ref{thm:identification2} with
  the choices $(Y,\dY):=(M,\sfd_M)$ and $X=\R^d$ endowed with the
  Euclidean distance $\sfd$.
  Condition (A) clearly holds since $M$ is closed in
  $\R^d$  and $\tilde{\mm}$ is supported on $M$.  Similarly, also (B) holds since $\jmath$ is an isometric
  immersion.

  Remark \ref{rem:serve-tutto} shows that $\rmC^\infty_c(\R^d)$ is
  dense in $H^{1,2}(\R^d,\sfd_{\R^d},\tilde\mm)$   so that
  $\jmath^*\big(\rmC^\infty_c(\R^d)\big)\subset \rmC_c^\infty(\M)$
  is dense in $H^{1,2}(\M,\sfd_\M,\mm).$  
\end{proof}

\newcommand{\pd}{{2}}
\section{Wasserstein spaces}
\label{sec:Wasserstein}
In this section we list some properties of Wasserstein spaces we will use in the sequel. A complete account of this matter can be found e.g.~in \cite{Villani09,AGS08}.\\
If $(X, \sfd)$ is a complete and separable metric space, we denote by
$\prob(X)$ the space of Borel probability measures on $X$ and by
$\prob_\pd(X)$,
the set
\[ \prob_\pd(X):= \left \{ \mu \in \prob(X) \mid \int_{X} \sfd^\pd(x,x_0) \d \mu(x) < + \infty \text{ for some } x_0 \in X \right \}.\]
Given $\mu, \nu \in \prob(X)$ the set of transport plans between $\mu$ and $\nu$ is denoted by $\Gamma(\mu, \nu)$ and defined as
\[ \Gamma(\mu, \nu): = \left \{ \mmu \in \prob(X \times X) \mid \pi^1_\sharp \mmu=\mu, \, \pi^2_\sharp \mmu=\nu \right \},\]
where $\pi^i(x_1,x_2)=x_i$ for every $(x_1,x_2) \in X \times X$ and $\sharp$ denotes the push forward operator.
The $L^\pd$-Wasserstein distance $W_\pd$ between $\mu, \nu \in \prob_\pd(X)$ is defined as
\[ W_\pd^\pd(\mu, \nu) := \inf \left \{ \int_{X \times X} \sfd^\pd \,\d \mmu \mid \mmu \in \Gamma(\mu, \nu) \right \}.\]
It is well known that the infimum above is attained in a non-empty and
convex set $\Gamma_o(\mu, \nu) \subset \Gamma(\mu, \nu)$; elements of
$\Gamma_o(\mu, \nu)$ are called optimal transport plans. 

The space $(\prob_\pd(X), W_\pd)$ is complete and separable and its topology is stronger than the narrow topology, the latter being defined as the coarsest topology on $\prob(X)$ making the maps
\[ \mu \mapsto \int_{X} \varphi \, \d \mu\]
continuous for every $\varphi \in \rmC_b(X)$, the space of continuous and bounded functions on $X$. In particular, for a sequence $(\mu_n)_n \subset \prob_\pd(X)$ and a point $\mu \in \prob_\pd(X)$, we have
\begin{equation}
    W_\pd(\mu_n, \mu) \to 0 \Leftrightarrow \begin{cases}
      &\displaystyle\int_{X} \sfd^\pd(x,x_0) \d \mu_n(x) \to
      \int_{\R^d} \sfd^\pd(x,x_0) \d \mu(x) \quad \text{ for some }
      x_0\in X,\\ &\mu_n \to \mu\quad\text{narrowly in }\prob(X).
    \end{cases}
\end{equation}
 Moreover,  the Wasserstein distance is narrowly lower semicontinuous,
meaning that, if $(\mu_n)_n$ and $(\mu'_n)_n$ are two sequences in $\prob_\pd(X)$, $\mu,
,\mu' \in \prob_\pd(X)$ and $\mu_n \to \mu$, $\mu'_n \to \mu'$
narrowly in $\prob(X)$, then we have
\[ \liminf_{n\to\infty} W_\pd(\mu_n, \mu'_n) \ge W_\pd(\mu, \mu').\]
The following Theorem is \cite[Theorem 8.3.1, Proposition 8.4.5 and
Proposition 8.4.6]{AGS08} in case $X=\R^d$.
 Recall that for every $\mu\in \prbt$
\begin{equation}
  \label{eq:129}
  \Tan_{\mu}\prob_2(\R^d) :={} \overline{\{ \nabla \varphi \mid
    \varphi \in \rmC^\infty_c(\R^d) \}}^{L^2(\R^d, \mu; \R^d)}.
\end{equation}
\begin{theorem}[Wasserstein velocity field]
  \label{thm:tangentv}
Let  $(\mu_t)_{t \in \interval} \subset \prob_2(\R^d)$  be a locally absolutely
continuous curve defined in an open interval $\interval\subset \R$.
There exists a Borel vector field
$v:\interval\times \R^d\to \R^d$
and
a set $A( (\mu_t)_{t \in \interval}) \subset \interval$ with
$ \Leb{1}(\interval \setminus A( (\mu_t)_{t \in \interval}))=0$ such that
for every $t\in A( (\mu_t)_{t \in \interval})$ 
\begin{displaymath}
  \begin{gathered}
    v_t\in\Tan_{\mu_t}\prob_2(\R^d) ,\quad
  \int_{\R^d} |v_t|^2\,\d\mu_t=|\dot \mu_t|^2=\lim_{h\to
    0}\frac{W_2^2(\mu_{t+h},\mu_t)}{h^2},
\end{gathered}
\end{displaymath}
and the continuity equation
\[\partial_t\mu_t+\nabla\cdot(v_t\mu_t)=0\]
holds in the sense of distributions in $\interval\times \R^d$.
Moreover, $v_t$ is uniquely determined
in $L^2(\R^d, \mu_t; \R^d)$ for $t\in A( (\mu_t)_{t \in \interval})$ and 
\begin{equation}
 \lim_{h \to 0} \frac{W_2((\ii_{\R^d}+h v_t)_{\sharp}\mu_t,
   \mu_{t+h})}{|h|} =0 \quad \text{for every }t \in A( (\mu_t)_{t \in \interval}),\label{eq:74}
\end{equation}
 where $\ii_{\R^d}$ is the identity map on $\R^d$.
\end{theorem}

\subsection{Kantorovich duality and estimates for Kantorovich
  potentials}
\label{subsec:Kuseful}
The Kantorovich duality for the Wasserstein distance states that
\begin{equation}\label{eq:duality}
    W_\pd^\pd(\mu, \nu) = \sup \left \{ \int_{X}u\, \d \mu + \int_{X} v\, \d \nu \mid  (u, v)  \in \text{Adm}_\pd(X) \right \} \quad \text{ for every } \mu, \nu \in \prob_\pd(X),
  \end{equation}
where $\text{Adm}_\pd(X)$ is the set of pairs $(u,v) \in \rmC_b(X) \times \rmC_b(X)$ such that 
\[ u(x) + v(y) \le \sfd^\pd(x,y) \quad \text{ for every } x,y \in X.\]
 It is easy to check that for every $f\in \Lip(X,\sfd)$
\begin{equation}
  \label{eq:123}
  \int_X f\,\d(\mu-\nu)\le \Lip(f,X)W_2(\mu,\nu),
\end{equation}
since choosing $\mmu\in \Gamma_o(\mu,\nu)$ and setting $L:=\Lip(f,X)$,
\begin{align*}
  \int_X f\,\d(\mu-\nu)=
  \int (f(x)-f(y))\,\d\mmu(x,y)\le
  L\int \sfd\,\d\mmu\le
  L\Big(\int \sfd^2\,\d\mmu\Big)^{1/2}=LW_2(\mu,\nu).
\end{align*}

When $X=\R^d$, we denote by $\prob_2^r(\R^d)$ the subset of
$\prob_2(\R^d)$ of probability measures that are absolutely continuous
w.r.t.~the $d$-dimensional Lebesgue measure.
We also set
\begin{equation}
  \label{eq:183}
  \sqm\mu:=\int_{\R^d}|x|^2\,\d\mu(x)=W_2^2(\mu,\delta_0).
\end{equation}

The next result uses the celebrated Brenier-Knott-Smith Theorem
\cite[Section 3]{Villani03} to collect various useful properties of the optimal
potentials realizing the supremum in  \eqref{eq:duality} in a particular geometric
situation. We will use the elementary property
that
\begin{equation}
  \label{eq:diseq}
  \text{if $u:\overline{B(0,R)}\to [-\infty,+\infty)$ is concave with $u(0)>-\infty$
  then}\quad
  \sup_{ \overline{\rmB(0,R)}} u  = \sup_{\rmB(0,R)} u,
  \end{equation}
  which follows by the fact that for every $y_0 \in \partial
  \rmB(0,R)$
  the concavity of $t\mapsto u(ty_0)$ in $[0,1]$ 
  yields $u(y_0)\le \sup_{0\le t<1} u(ty_0)$.
\begin{theorem} \label{thm:ot} Let $\mu, \nu \in \prob_2^r(\R^d)$
  with $\supp{\nu} =  \overline{\rmB(0,R)}$ for some $R>0$. Then there
  exists a unique pair of
  continuous
  and convex functions 
  \begin{equation}
 \varphi=\Phi(\nu,\mu) : \rmB(0,R) \to \R,
    \quad \varphi^*=\Phi^*(\nu,\mu): \R^d \to \R
    \label{eq:157}
  \end{equation}
such that 
\begin{enumerate}
\item[(i)]
  $\varphi^*$ is $R$-Lipschitz and
  \begin{align}
    \label{eq:160}
    \displaystyle \varphi^*(y) &= \sup_{x \in \rmB(0,R)}  \la x
  , y \ra - \varphi(x) && \text{for every $y \in \R^d$},\\
  \varphi(x)&=\sup_{y\in \R^d} \la
  y,x\ra-\varphi^*(y)&& \text{for every }x\in \rmB(0,R),
  \label{eq:159}
\end{align}
\item[(ii)] $\displaystyle \varphi^*(0)=\inf_{\rmB(0,R)} \varphi=0$,
\item[(iii)] $\displaystyle \int_{\rmB(0,R)} \varphi\, \d \nu + \int_{\R^d} \varphi^*\, \d \mu = \frac{1}{2} \sqm{\nu} + \frac{1}{2} \sqm{\mu} - \frac{1}{2}W_2^2(\nu, \mu)$.
\end{enumerate}
Moreover the pair $(\varphi, \varphi^*)$ satisfies 
\begin{equation}
\displaystyle W_2^2(\mu, \nu) =\int_{\rmB(0,R)} \left |x-\nabla
  \varphi(x) \right |^2 \d \nu(x) =\int_{\R^d} \left |y-\nabla
  \varphi^*(y) \right |^2 \d \mu(y).\label{eq:158}
\end{equation}
\end{theorem}
\begin{proof}
  Let us set $D:=\rmB(0,R)$.
  We know (see e.g.~\cite[Theorem 2.9, Lemma 2.10]{Villani03})
  that there exists a  pair $(\phi, \phi^*)$ of lower semicontinuous proper conjugate functions such that $\phi \in L^1(\overline{D},\nu;(-\infty, +\infty])$, $\phi^* \in L^1(\R^d,\mu;(-\infty, +\infty])$ and it holds 
  \begin{equation}
    \label{eq:151}
     \int_{\overline{D}} \phi\, \d \nu + \int_{\R^d} \phi^*\,
     \d \mu =
     \frac{1}{2} \sqm{\nu} + \frac{1}{2} \sqm{\mu} - \frac{1}{2}W_2^2(\nu, \mu),
  \end{equation}
  where
  \begin{equation}
  \phi^*(y):=\sup_{x\in \overline{D}}\la
  x,y\ra-\phi(x)=
  \max_{x\in \overline{D}}\la
  x,y\ra-\phi(x).
  \label{eq:156}
\end{equation}
Recalling that $\phi$ is bounded from below by an affine mapping
(and thus it is uniformly bounded from below in $\overline{D}$)
we immediately see that $\phi^*$ takes values in $\R$ and it is
$R$-Lipschitz. Up to adding a suitable constant to $\phi$ we can
also suppose that 
$\phi^*(0)=0$.

We want to show that the restriction $\varphi$ of $\phi$ to $\rmB(0,R)$
combined with $\phi^*$ 
satisfies conditions (i), (ii), and (iii).

\textbf{(i)}:
Since $\int_{\overline{D}}\phi\,\d\nu<+\infty$ and
$\nu$ has full support, we deduce that the proper
domain of $\phi$ $\{x\in \overline{D}:\phi(x)<+\infty\}$ is dense in
$\overline{D}$;
since the proper domain of a l.s.c.~and convex function is convex
and contains the interior of its closure, we deduce that 
$\phi(x)<+\infty$ for every $x\in
D$ and $\varphi:=\phi\restr D$ is continuous in $D$.

\eqref{eq:diseq} shows that the supremum defining $\phi^*$ in \eqref{eq:156}
can be restricted
to $D$
\begin{equation}
  \label{eq:152}
  \phi^*(y)=\sup_{x\in D}\la
  x,y\ra-\varphi(x),
\end{equation}
so that \eqref{eq:160} holds.
  \eqref{eq:159} is just an application of Fenchel-Moreau Theorem $\phi=\phi^{**}$.

  \textbf{(ii)}: simply follows by \eqref{eq:160} and the fact that $\phi^*(0)=0$.

  \textbf{(iii)}
  It is sufficient to notice that the first integral in \eqref{eq:151} can be restricted to $D$
  since $\nu(\partial D)=0$.
The equality \eqref{eq:158} follows by \cite[Theorem 2.12]{Villani03}.

Let us show that points (i)-(iii) are also sufficient to get
uniqueness. If $(\varphi_0, \varphi_0^*)$ is another pair as in the
statement satisfying points (i)-(iii), then \cite[Theorem
2.12]{Villani03} yields
that both $\nabla \varphi$ and $\nabla \varphi_0$ are optimal
transport maps from $\nu$ to $\mu$, implying that $\nabla \varphi_0 =
\nabla \varphi$ $ \Leb{d}$-a.e.~in $\rmB(0,R)$ by the
a.e.~uniqueness of the optimal transport map.
Since $\inf_{\rmB(0,R)} \varphi = \inf_{\rmB(0,R)} \varphi_0=0$ by
$(ii)$, we get that $\varphi=\varphi_0$ in $\rmB(0,R)$
and therefore $\varphi^*=\varphi^*_0$ in $\R^d$ by \eqref{eq:160}.
\end{proof}
The next Lemma collects useful estimates on convex functions; we set
$\omega_d:=\Leb d\big(\rmB(0,1)\big)$.
\begin{lemma}
  \label{le:estimate} Let $R,I>0$ and let
  $\varphi:\rmB(0,R)\to \R$, $\psi:\R^d\to \R$, be two 
  (continuous and) convex functions satisfying
\begin{equation}
  \begin{gathered}
   |\psi(y)|\le R|y| \quad \text{for every }y\in \R^d,\\
  \varphi(x) = \sup_{y \in \R^d}  \la x , y \ra -
  \psi(y) \ \text{for every } x \in
  \rmB(0,R),\quad
  \int_{\rmB(0,R)}\varphi(x)\,\d x\le I.
  \end{gathered}
  \label{eq:161}
\end{equation}
Then
$\varphi$ is nonnegative and satisfy the uniform bounds
\begin{equation}
  \label{eq:168}
  \sup_{|x|\le r} \varphi(x)\le \frac I{\omega_d
    (R-r)^d},\quad
  \Lip\big(\varphi,\overline{\rmB(0,r)}\big)\le
  \frac{2^{d+1}I}{\omega_d(R-r)^{d+1}}\quad
  0<r<R,
\end{equation}
and
\begin{equation}
  \label{eq:169}
  \text{$\psi$ is $R$-Lipschitz,}\quad
  \psi(y)=\sup_{x\in \rmB(0,R)}\la y,x\ra-\varphi(x)\quad
  \text{for every }y\in \R^d.
\end{equation}
\end{lemma}
\begin{proof}
  Notice that $\psi(0)=0$ yields $\varphi\ge 0$;
the integral estimate of \eqref{eq:161} and Jensen inequality yield
for every $x\in \rmB(0,R)$ with $\varrho:=R-|x|$
\begin{equation}
  \label{eq:165}
  \varphi(x)\le \frac1{\omega_d \varrho^d}
  \int_{\rmB(x,\varrho)}\varphi(z)\,\d z\le \frac I{\omega_d
    \varrho^d} ,\quad\text{so that}\quad
  \max_{|x|\le r_0}\varphi(x)\le \frac I{\omega_d
    (R-r_0)^d} \,
\end{equation}
for every $0<r_0<R$. Still using the fact that $\varphi$ is nonnegative, \eqref{eq:165}
with $r_0:=\frac 12 R+\frac 12 r$ and
the estimate of 
\cite[Corollary 2.4]{Ekeland-Temam74}
yield the Lipschitz bound of \eqref{eq:168}.

The Legendre transform of $\psi$ defined by
$\psi^*(x):=\sup_{y\in \R^d} \la x,y\ra-\psi(y)$
coincides with $\varphi$ in $\rmB(0,R)$ 
(in particular it is finite in $\rmB(0,R)$) and
takes the value $+\infty$ for every $x\in \R^d$
with $|x|>R$, since
\begin{displaymath}
    \psi^*(x)\ge 
    \sup_{y\in \R^d}\la x,y\ra-R|y|=+\infty
    \quad\text{if $|x|>R$}.
\end{displaymath}
Fenchel-Moreau Theorem 
and \eqref{eq:diseq}
then yield
\begin{displaymath}
    \psi(y)=\sup_{x\in \R^d}\la x,y\ra-\psi^*(x)
    =\sup_{x\in \overline{\rmB(0,R)}}\la x,y\ra-\psi^*(x)
    =\sup_{x\in \rmB(0,R)}\la x,y\ra-\psi^*(x)
    =\sup_{x\in \rmB(0,R)}\la x,y\ra-\varphi(x)
\end{displaymath}
thus showing \eqref{eq:169}; 
in particular we get that 
that $\psi$ is $R$-Lipschitz.
\end{proof}
We conclude this part with the study of the stability properties of pairs of potentials.
\begin{lemma}\label{le:convan} Let $R,I>0$ and let
  $\varphi_n:\rmB(0,R)\to \R$, $\psi_n:\R^d\to \R$, $n\in \N$, be two sequences of
  (continuous and) convex functions satisfying
for every $n \in \N$:
\begin{equation}
  \begin{gathered}
  |\psi_n(y)|\le R|y| \quad\text{for every }y\in \R^d, \\
    \varphi_n(x) = \sup_{y \in \R^d}  \la x , y \ra -
    \psi_n(y) \ \text{for every } x \in
    \rmB(0,R),\quad
    \int_{\rmB(0,R)}\varphi_n(x)\,\d x\le I.
    \end{gathered}
    \label{eq:161bis}
 \end{equation}
Then there exist a subsequence $j \mapsto n(j)$ and two convex and
continuous functions $\varphi: \rmB(0,R) \to \R$ and $\psi: \R^d \to \R$ such that
\begin{itemize}
    \item [(i)] $\varphi_{n(j)} \to \varphi$ locally uniformly on $\rmB(0,R)$;
    \item[(ii)] $\psi_{n(j)}\to \psi$ locally uniformly in $\R^d$;
    \item[(iii)] $\psi$ is $R$-Lipschitz and
      $\nabla \psi_{n(j)} \to \nabla \psi$ $ \Leb{d}  $-a.e.~on $\R^d$.
    \end{itemize}
    Moreover the pair $(\varphi,\psi)$ satisfies \eqref{eq:161},
    \eqref{eq:168}, and \eqref{eq:169}.
\end{lemma}
\begin{proof}
  Thanks to \eqref{eq:161bis} and Lemma \ref{le:estimate},
  the sequence of pairs $(\varphi_n,\psi_n)$ satisfies
  the equicontinuity estimates \eqref{eq:168} and \eqref{eq:169}
  with constants $R,I$ independent of $n$.
  
  By Arzel\`a-Ascoli Theorem, we can find a subsequence $j \mapsto n(j)$ and
convex and continuous functions $\varphi:\rmB(0,R)\to \R$ and
$\psi: \R^d \to \R$ such that $\varphi_{n(j)} \to \varphi$ and
$\psi_{n(j)}\to\psi$ locally uniformly in their respective domains.

In particular, $\psi_{n(j)}$ Mosco converges (see
e.g.~\cite[Definition 3.17, Proposition 3.19]{Attouch84})
to $\psi$ and therefore the sequence of its Legendre transforms $\psi_{n(j)}^*$ Mosco
converges to $\psi^*$ (\cite[Theorem 3.18]{Attouch84}).
Since $\psi_n^*$ coincides with $\varphi_n$ in $\rmB(0,R)$
and $\varphi_{n(j)}$ converge locally uniformly to $\varphi$,
we deduce that $\varphi$ coincides with $\psi^*$ in $\rmB(0,R)$:
\begin{displaymath}
  \varphi(x)=\sup_{y\in \R^d}\la x,y\ra-\psi(y)\quad\text{for every
  }x\in \rmB(0,R).
\end{displaymath}
By Fatou's Lemma $\varphi$ also satisfies the integral bound of
\eqref{eq:161}.
A further application of Lemma \ref{le:estimate} yields \eqref{eq:168}
and \eqref{eq:169}.

Finally, the local uniform convergence of $\psi_{n(j)}$ to $\psi$ gives \cite[Theorem 24.5]{Rockafellar70} the pointwise convergence of $\nabla \psi_{n(j)}(x)$ to $\nabla \psi(x)$ at every point $x \in \R^d$ where all the $\psi_{n(j)}$ and $\psi$ are differentiable. This proves (iii) and concludes the proof of the Lemma.
\end{proof}

\section{The Wasserstein Sobolev space
  \texorpdfstring{$H^{1,2}(\prob_2(\R^d ),W_2,\mm)$}{H}}
\label{sec:main2}
In this section we consider the metric space $\prbt$, 
endowed with the $L^2$-Wasserstein distance $\sfd=W_2$ and a finite
positive Borel measure $\mm$.
 We will denote by
$\W_2=\Wmms{\R^d}\mm$ the metric-measure space $(\prob_2(\R^d),W_2,\mm)$
and we want to study the Wasserstein Sobolev space $H^{1,2}(\W_2)$.

We will show that $H^{1,2}(\W_2)$ is Hilbertian
(and therefore the metric space $(\prob_2(\R^d),W_2)$ is
infinitesimally Hilbertian) and its functions admit a nice
approximation in terms of the distinguished algebra of cylinder
functions.

\subsection{The algebra of \texorpdfstring{$\rmC^1$}{C}-cylinder functions}
\label{subsec:cylindrical}
 We denote by $\rmC^1_b(\R^d)$ the space of bounded and Lipschitz $\rmC^1$
functions $\phi:\R^d\to\R$.  This in particular implies that $\sup_{x \in \R^d} |\phi(x)|+|\nabla \phi(x)| < + \infty$ if $\phi \in \rmC^1_b(\R^d)$.  Every $\phi\in \rmC^1_b(\R^d)$ induces the 
function $\lin\phi$ on $\prob(\R^d)$
\begin{equation}\label{eq:monocyl}
  \lin\phi:\mu \to \int_{\R^d} \phi \,\d \mu
\end{equation}
which clearly belongs to $\Lip_b(\prob_2(\R^d),W_2)$ thanks to
\eqref{eq:123}.
More generally, if $\pphi=(\phi_1,\cdots,\phi_N)\in
\big(\rmC^1_b(\R^d)\big)^N$, we denote by 
$\lin\pphi:=(\lin{\phi_1},\cdots,\lin{\phi_N})$ the corresponding map
from $\prob_2(\R^d)$ to $\R^N$.

Our construction is based on the algebra of $\rmC^1$- cylinder
functions generated by \eqref{eq:monocyl} via composition with $\rmC^1$ functions
and it is quite similar to the one of \cite[Section
2]{DelloSchiavo20} (see also \cite{vRS09}). Working in the flat space $\R^d$ allows for a
further simplification in the structure of the tangent bundle and of
corresponding vector fields.

\begin{definition}[$\rmC^1$-Cylinder functions] We say that a function $F: \prob_2(\R^d) \to \R$ is a $\rmC^1$-\emph{cylinder function} if there exist $N \in \N$, $\psi \in \rmC_b^1(\R^N)$ and $\uphi= (\phi_1, \dots, \phi_N) \in (\rmC_b^1(\R^d))^N$ such that 
\begin{equation}\label{eq:cyl}
F(\mu) = \psi(\lin\pphi(\mu))=\psi \big( \lin{\phi_1}(\mu),\cdots,\lin{\phi_N}(\mu)\big)
\quad \text{for every } \mu \in \prob_2(\R^d).
\end{equation}
We denote the set of such functions by $\cyl1b{\prob_2(\R^d)}$.
\end{definition}

\begin{remark}
Notice that $\cyl1b{\prbt}$ is a unital subalgebra of
$\Lipb(\prob_2(\R^d),W_2)$.
One could also consider the smaller algebra $\ccyl1b{\prbt}$
(resp.~$\ccyl\infty c\prbt$)
generated by functions as in \eqref{eq:monocyl} (resp.~by  functions  as in \eqref{eq:monocyl} where $\phi\in \rmC^\infty_c(\R^d)$),
thus restricting $\psi$ to be a polynomial in \eqref{eq:cyl}.  This means that every element $F \in \ccyl1b{\prbt}$ (resp.~$\ccyl\infty c\prbt$) can be written as
\[ F = \psi \circ \lin \pphi \]
for some $\psi$ polynomial in $\R^N$, $\pphi \in (\rmC_b^1(\R^d))^N$ (resp.~$(\rmC^\infty_c(\R^d)^N$) and $N \in \N$, $N \ge 1$.  
We prefer at this stage the choice of $\cyl1b\prbt$, since it 
simplifies some technical points. However,  Proposition 
\ref{prop:density} shows that
using 
$\ccyl1b{\prbt}$ or $\ccyl\infty c\prbt$ will lead to the same
conclusions.
\end{remark}
\begin{remark}\label{rem:unrem}
  Since for every $\uphi\in \big(\rmC^1_b(\R^d)\big)^N$
  the range of $\lin\uphi$ is always contained in the
  bounded set
  $[-M,
  M]^N$ where $M:= \max_{i=1, \dots, d} \|\phi_i\|_\infty$,
  also functions $F=\psi\circ \lin\uphi$ 
  with $\psi \in \rmC^1(\R^N)$ belong
  to $\cyl1b{\prbt}$. Indeed it is enough to consider a function
  $\tilde{\psi} \in \rmC_b^1(\R^N)$ coinciding with $\psi$ on
   $[-M,
  M]^N$ and equal to $0$ outside $[-M-1, M+1]^N$
  so that $F=\tilde\psi\circ \lin\uphi$.
In particular
every function of the form $\lin\phi$, $\phi\in \rmC^1_b(\R^d)$,
belongs to $\cyl1b\prbt$.
\end{remark}
Let us consider the set
\begin{equation}
  \label{eq:51}
  \domG:=\Big\{(\mu,x)\in \prbt\times \R^d:x\in \supp(\mu)\Big\}.
\end{equation}
The set $\domG$ is a Borel set (in fact it is a $G_\delta$): if $(r_n)_n = \Q \cap (0, + \infty)$ we have that $\domG= \cap_n \domG_n$, where
\[ \domG_n := \left \{ (\mu,x)\in \prbt\times \R^d : \mu(\rmB(x,r_n)) > 0 \right \},\]
and each $\domG_n$ is open in $\prbt \times \R^d$, being the inverse image of $(0,+\infty)$ through the lower semicontinuous map $(\mu,x) \mapsto \mu(\rmB(x,r))$. 

\begin{definition}
   If $F=\psi\circ\lin\uphi \in \cyl1b{\prbt}$ as in
  \eqref{eq:cyl}
for some $N \in \N$, $\psi \in \rmC_b^1(\R^N)$ and $\uphi \in
(\rmC_b^1(\R^d))^N$, then the Wasserstein differential of $F$,  
$\rmD F: \overline{\domG}\to \R^d$, is defined by
\begin{align}\label{eq:Fdiff}
\rmD F(\mu,x) &:= \sum_{n=1}^N \partial_n \psi \left ( \lin \uphi(\mu)
                 \right ) \nabla \phi_n(x), \quad (\mu,x)\in \overline{\domG}.
                \intertext{We will also denote by $\rmD F[\mu]$ the
                function $x\mapsto \rmD F(\mu,x)$ and we will set}
\cnorm{F}{\mu}^2 &:= \int_{\R^d} |\rmD F[\mu](x)|^2 \d \mu(x), \quad \mu \in \prbt.
\end{align}
\end{definition}

\begin{remark}
  \label{rem:propDF}
   It is not difficult to check that
  \begin{equation}
    \label{eq:124}
    \rmD F\text{ is continuous in }\overline{\domG}
  \end{equation}
  with respect to the natural product (narrow and euclidean) topology
  of $\prob(\R^d)\times \R^d$.
  
  In principle
  $\rmD F$  (and thus $\cnorm{F}{\mu}$)  may depend on the choice of $N \in
  \N$, $\psi \in \rmC_b^1(\R^N)$ and $\uphi \in (\rmC_b^1(\R^d))^N$
  used to represent $F$. In Proposition \ref{prop:equality} we show
  that for every $\mu\in \prob_2(\R^d)$ 
  the function $\rmD F[\mu]$ is uniquely characterized in $\supp(\mu)$
  and $\cnorm{F}{\mu}$ is well defined, so that $\rmD F$ is
  uniquely characterized by $F$ in $\domG$. By \eqref{eq:124}, $\rmD F$ is
  also uniquely
  characterized by $F$ on $\overline \domG$.
\end{remark}
We have seen that the  Wasserstein differential $\rmD F$ can
be considered as a map from $\overline \domG$ with values in
$\R^d$. It is natural to introduce the
measure $\bmm=\int \delta_\mu\otimes \mu \,\d\mm(\mu)\in \prob(\prob_2(\R^d)\times \R^d)$ obtained integrating
the measures $\mu$ w.r.t.~$\mm$: for every bounded Borel function
$H: \prob_2(\R^d)\times \R^d\to \R$ we have
\begin{equation}
  \label{eq:49}
  \int H(\mu,x)\,\d\bmm(\mu,x)=
  \int_{\prob_2(\R^d)}\Big(\int_{\R^d} H(\mu,x)\,\d\mu(x)\Big)\,\d\mm(\mu).
\end{equation}
Since $\supp(\bmm)\subset \overline\domG$, 
it is then clear that $\rmD F$ belongs to $L^2(\prob_2(\R^d)\times
\R^d,\bmm;\R^d)$
and
\begin{equation}
  \label{eq:50}
   \|\rmD F\|^2_{L^2(\prob_2(\R^d)\times
\R^d,\bmm;\R^d)} =  \int_{\prob_2(\R^d)}\cnorm{F}{\mu}^2 \,\d\mm(\mu)=
  \int_{\overline \domG} |\rmD F(\mu,x)|^2\,\d\bmm(\mu,x).
\end{equation}

\begin{lemma}\label{lem:limit}
   Let $Y$ be a Polish space
  and let $G:\prob(Y) \times Y \to [0, + \infty)$ be a bounded and
  continuous function. If
  $(\mu_n)_{n\in \N}$ is a sequence in $\prob(Y)$ narrowly converging
  to $\mu$ as $n \to + \infty$, then
\[ \lim_{n\to\infty} \int_Y G(\mu_n, y) \d \mu_n(y) = \int_Y G(\mu, y) \d \mu(y).\]
\end{lemma}
\begin{proof}
   We set $g_n(x):=G(\mu_n,x),\ g(x):=G(\mu,x)$.
  Since $G$ is continuous, $g_n$ converge uniformly  to $g$ on compact
  subsets of $Y$ as $n\to\infty$. Thanks to
  \cite[Lemma 5.2.1]{AGS08} $(g_n)_\sharp \mu_n$ converge narrowly to
  $g_\sharp\mu$ in $\prob(\R)$. On the other hand, the support of
  $(g_n)_\sharp \mu_{ n}$ is uniformly bounded  because $G$ is bounded  so that
  \begin{displaymath}
    \lim_{n\to\infty}\int_Y G(\mu_n, y) \d \mu_n(y)=
    \lim_{n\to\infty}\int_{\R}r\,\d((g_n)_\sharp\mu_n)(r)=
    \int_\R r\,\d(g_\sharp\mu)(r)=
    \int_Y G(\mu,y)\,\d\mu(y).
    \qedhere
  \end{displaymath}
\end{proof}

\begin{lemma}\label{lem:derivative} Let $F =\psi\circ\lin\uphi\in
  \cyl1b{\prbt}$
  as in \eqref{eq:cyl}
  and let $(\mu_t)_{t \in [0,1]}$ be an absolutely continuous curve in $\prob_2(\R^d)$. Then
\begin{equation}\label{eq:oscill}
F(\mu_1)- F(\mu_0) = \int_0^1 \int_{\R^d} \la \rmD F[\mu_t](x), v_t(x) \ra \d \mu_t(x)\, \d t,
\end{equation} 
where $v_t \in L^2(\R^d, \mu_t; \R^d)$ is the Wasserstein velocity field (cf.~Theorem \ref{thm:tangentv}) of $(\mu_t)_{t \in [0,1]}$ at time $t$ and $\rmD F$ is as in \eqref{eq:Fdiff}.

In case the curve $(\mu_t)_{t \in [0,1]}$ admits the parametrization 
\[\mu_t:=\big(\boldsymbol x_t\big)_\sharp \mu,\quad t\in [0,1],
\]
for some Borel probability measure $\mu$ in a Polish space $\Omega$ and some map $\boldsymbol x\in \rmC^1([0,1];L^2(\Omega, \mu; \R^d))$,
then $F\circ \mu\in \rmC^1([0,1])$ and
\begin{equation}
\label{eq:chain}
    \frac{\d}{\d t} F(\mu_t) = \int_{\Omega} \la \rmD F(\mu_t,\boldsymbol x_t(\omega)), \dot {\boldsymbol x}_t(\omega)) \ra\, \d \mu(\omega)
    \quad\text{for every } t\in [0,1].
\end{equation}
\end{lemma}
\begin{proof}
Observe that, since $F$ is Lipschitz continuous and $t \mapsto \mu_t$ is absolutely continuous, the map $t \mapsto F(\mu_t)$ is absolutely continuous and thus it holds
\[ F(\mu_1)-F(\mu_0)= \int_0^1 \frac{\d }{\d t} F(\mu_t) \d t.\]
It is then enough to prove that 
\begin{equation}\label{eq:derivative}
\frac{\d}{\d t} F(\mu_t) = \int_{\R^d} \la \rmD F(\mu_t,x), v_t(x) \ra
\,\d \mu_t(x) \quad \text{for a.e. } t \in (0,1).
\end{equation}
We have, for every $t \in A((\mu_t)_{t \in [0,1]}) \subset (0,1)$ (cf.~Theorem \ref{thm:tangentv}), that
\begin{align*}
\frac{ \d}{ \d t} F(\mu_t) &= \sum_{i=1}^N \partial_i \psi( \lin \uphi(\mu_t)) \frac{\d }{ \d t} \int_{\R^d} \phi_i\, \d \mu_t \\
&=\sum_{i=1}^N \partial_i \psi( \lin\uphi(\mu_t)) \int_{\R^d} \la \nabla \phi_i, v_t(x) \ra \,\d \mu_t(x) \\
&= \int_{\R^d} \la \rmD F(\mu_t,x), v_t(x) \ra \,\d \mu_t(x),
\end{align*}
where we used Theorem \ref{thm:tangentv}.
A completely analogous argument 
yields \eqref{eq:chain}.
\end{proof}

\begin{remark}  Consider the case in which  the curve $(\mu_t)_{t \in [0,1]}$ has the simple form
\[ \mu_t := (\ii_{\R^d} + t u)_\sharp \mu, \quad t \in [0,1]\]
for some map $u \in L^2(\R^d, \mu; \R^d)$,  where $\ii_{\R^d}$ denotes the identity map on $\R^d$. Then  
 \eqref{eq:chain} 
yields
\[ \frac{\d}{\d t} F(\mu_t) = \int_{\R^d} \la \rmD F(\mu_t,x), u(x) \ra\, \d \mu_t(x)
\quad\text{for every }t\in [0,1],\]
and, in particular, we get
\begin{equation}\label{eq:simpler} \lim_{t \downarrow 0} \frac{F(\mu_t)-F(\mu)}{t} = \int_{\R^d} \la \rmD F(\mu,x), u(x)) \ra \,\d \mu(x).
\end{equation}

\end{remark}

\begin{proposition} \label{prop:equality} Let $F =\psi\circ\lin\uphi\in
  \cyl1b{\prbt}$
  as in \eqref{eq:cyl}. Then
\[ \cnorm{F}{\mu} = \lip F (\mu) \quad \text{ for every } \mu \in \prbt.\]
In particular $\cnorm{F}{\mu}$ does not depend on the choice of the
representation of $F$
and $\rmD F$ just depends on $F$ on $\overline \domG$.
\end{proposition}
\begin{proof} Let $\mu \in \prbt$ and let $(\mu'_n, \mu''_n) \in \prbt^2$ with $\mu'_n \ne \mu''_n$ be such that $(\mu'_n, \mu''_n) \to (\mu, \mu)$ in $W_2$ and
\[ \lim_n \frac{ \left |F(\mu'_n) - F(\mu''_n) \right |}{W_2(\mu'_n, \mu''_n)} = \lip F (\mu).\]
Let us define, for every $t \in [0,1]$, the map $\sfx^t: \R^d \times \R^d \to \R^d$ as
\[ \sfx^t(x_0,x_1) := (1-t)x_0 + tx_1, \quad (x_0, x_1) \in \R^d \times \R^d.\]
Using \eqref{eq:chain}
along $\mu_n^t := \sfx^t_\sharp \mmu_n$
 for plans $\mmu_n \in
\Gamma_o(\mu'_n, \mu''_n)$
(it is easy to check that $( \mu_n^t)_{t \in [0,1]}$ is Lipschitz
continuous), we get
\begin{align*}
&\left | F(\mu'_n) - F(\mu''_n) \right |= \left |\int_0^1 \int_{\R^d\times \R^d} \la \rmD F(\mu_n^t,\sfx^t(x_0,x_1)), x_1-x_0\ra  \,\d \mmu_n(x_0,x_1) \,\d t \right | \\
&\le \left ( \int_0^1 \int_{\R^d\times \R^d} \left | \rmD F(\mu_n^t,\sfx^t(x_0,x_1))) \right |^2 \,\d \mmu_n \,\d t \right )^{\frac{1}{2}} \left ( \int_0^1 \int_{\R^d\times \R^d} \left |x_1-x_0 \right |^2 \,\d \mmu_n \,\d t \right)^{\frac{1}{2}} \\    &= W_2(\mu'_n, \mu''_n) \left ( \int_0^1 \int_{\R^d} \left |  \rmD F(\mu_n^t,x ) \right |^2 \,\d   \mu^t_n(x) \,\d t \right )^{\frac{1}{2}}, 
\end{align*}
where we used Theorem \ref{thm:tangentv}. Dividing both sides by $W_2(\mu'_n, \mu''_n)$, we obtain
\[ \frac{\left |F(\mu'_n) - F(\mu''_n) \right |}{W_2(\mu'_n, \mu''_n)} \le \left ( \int_0^1 \int_{\R^d} \left | \rmD F(\mu_n^t,x) \right |^2 \,\d  \mu^t_n \,\d t \right )^{\frac{1}{2}}.\]
Observe that $\mmu_n \to \mmu:=(\ii_{\R^d}, \ii_{\R^d})_\sharp
\mu$  narrowly  in $\prob(\R^d\times\R^d)$  as $n\to +\infty$  so that $\mu_n^t\to \mu$  narrowly  in $\prob(\R^d)$  as $n\to +\infty$  for every $t\in [0,1]$. 
We can pass to the limit as $n \to + \infty$ the above inequality using the dominated convergence Theorem and Lemma \ref{lem:limit} with
\[ G(\mu, x):=  \left | \sum_{n=1}^N \partial_n \psi \left ( \lin \uphi(\mu)
                 \right ) \nabla \phi_n(x) \right |^2, \quad \mu \in \prob(\R^d), \, x\in \R^d,\]
 which provides a continuous and bounded extension (depending on the particular choice of $\psi$ and $\pphi$) of $|\rmD F|^2$ to $\prob(\R^d) \times \R^d$.  
We hence get 
\begin{align*}
    \lip F (\mu) &\le 
     \left ( \int_0^1 \int_{\R^d } \left | \rmD F(\mu,x) \right |^2 \,\d \mu (x) \,\d t \right )^{\frac{1}{2}}=
                 \cnorm{F}{\mu}.
\end{align*}

This proves one inequality. In order to prove the opposite one, it is not restrictive to assume 
$\cnorm{F}{\mu}>0$. Let us now consider the map $T:  \supp(\mu)  \to \R^d$ defined as
\[ T(x) := \rmD F[\mu](x), \quad x \in  \supp(\mu).\]
By definition of $\Tan_\mu(\prob_2(\R^d))$, we have that $T \in \Tan_\mu(\prob_2(\R^d))$ so that, by \cite[ Proposition  8.5.6]{AGS08}, we have
\[ \lim_{\eps \downarrow 0} \frac{W_2(\mu, (\ii_{\R^d}+ \eps T)_\sharp \mu)}{\eps} = \|T\|_{L^2(\R^d, \mu; \R^d)} = \cnorm{F}{\mu}.\]
Moreover, if we apply \eqref{eq:simpler} to the curve $\mu_\eps:=
(\ii_{\R^d}+ \eps T)_\sharp \mu$, $\eps \in [0,1]$, we get
\[ \lim_{\eps \downarrow 0} \frac{F(\mu_\eps)-F(\mu) }{\eps} = \int_{\R^d} \la \rmD F(\mu,x), T(x)) \ra \,\d \mu(x) = \cnorm{F}{\mu}^2, \]
thus
\[ \lip F (\mu) \ge \lim_{\eps \downarrow 0} \frac{F(\mu_\eps)-F(\mu) }{W_2(\mu_\eps, \mu)} = \cnorm{F}{\mu}.\]
This shows the other inequality and concludes the proof.
\end{proof}
\subsection{The density result}
\label{subsec:wsspace}
Recall that for a bounded Lipschitz function
$F:\prob_2(\R^d)\to\R$ the pre-Cheeger energy (cf.~\eqref{eq:prec})  associated to $\mm$ is defined by
\begin{equation}
  \label{eq:52}  \pCE_2(F)=\int_{\prob_2(\R^d)}\big(\lip F(\mu)\big)^2\,\d\mm(\mu).
\end{equation}
Thanks to Proposition \ref{prop:equality}, if $F$ is a cylinder
function in $\cyl1b\prbt$, we have a nice equivalent expression
\begin{equation}
  \label{eq:53}
  \pCE_2(F)=\int_{\prob_2(\R^d)}\cnorm{F}{\mu} ^2\,\d\mm(\mu)=
  \int |\rmD F(\mu,x)|^2\,\d\bmm(\mu,x),
\end{equation}
which shows that the restriction of $\pCE_2$ to $\cyl1b \prbt$ is a
quadratic form (thus satisfying \eqref{eq:31}) induced by the bilinear
form
\begin{equation}
  \label{eq:54}
  \pCE_2(F,G):=\int \rmD F(\mu,x)\cdot \rmD G(\mu,x)\,\d\bmm(\mu,x), \quad F,G \in \cyl1b \prbt
\end{equation} and coincides with the typical bilinear forms on cylinder
functions used in
\cite{vRS09,Sturm11,DelloSchiavo20,DelloSchiavo22}.
It is therefore important to prove that $\cyl1b \prbt$ is dense in energy
and therefore $H^{1,2}(\W_2)$ is a Hilbert space:
this is precisely the object of our main result.
\begin{theorem}
  \label{thm:main}
  The algebra $\cyl1b \prbt$ is
  dense in $2$-energy: 
  for every $F\in D^{1,2}(\W_2)$
  there exists a sequence $F_n\in \cyl1b \prbt$, $n\in \N,$ such that
  \begin{equation}
    \label{eq:55}
    F_n\to F\ \text{ $\mm$-a.e.,}\quad \lip(F_n)\to \relgrad{F} \text{ in }L^2( \prob_2(\R^d),\mm);
  \end{equation}
  if moreover $F\in L^p( \prob_2(\R^d),\mm)$, $p\in [1,+\infty)$, then
  we can find a sequence $F_n\in \cyl1b \prbt$ as in \eqref{eq:55}
  and converging to $F$ in $L^p( \prob_2(\R^d),\mm)$.
\end{theorem}
\begin{corollary}
  \label{cor:WHilbert}
  $H^{1,2}(\W_2)$ is a separable Hilbert space and $\cyl1b \prbt$ is
  strongly dense in $H^{1,2}(\W_2)$.
  If $\big(\pCE_2,\cyl1b\prbt\big)$ is closable (recall Remark
  \ref{rem:closable})
  then
  its smallest closed extension coincides with $(\CE_2,
  H^{1,2}(\W_2))$.
\end{corollary}
According to the terminology introduced in \cite{Gigli15-new} (see
also \cite{AGS14I})
we can say that $(\prbt,W_2,\mm)$ is infinitesimally Hilbertian for
every positive  Borel measure $\mm$. \\
 We devote the remaining part of this subsection to the proof this
 result, using Theorem \ref{theo:startingpoint}.
 
 We adopt the notation $\AA:=\cyl1b \prbt$.

 We start with a preliminary lemma, which provides a simple
 gradient estimate for the distance from
 the Dirac mass centered at 0, i.e.~the quadratic moment of a measure.
 \begin{lemma}
   \label{le:square-lip}
   Let $\vartheta\in \Lip(\R^d)$ be a $L$-Lipschitz function which is
   continuously differentiable in the open set $\Omega_\vartheta:=\big\{x\in \R^d:
   \vartheta(x)\neq 0\big\}$.
   Then the
   map
   \begin{equation}
     \label{eq:162}
     F:\mu\to \big(\lin{\vartheta^2} (\mu) \big)^{1/2} =\Big(\int_{\R^d}\vartheta^2(x)\,\d\mu(x)\Big)^{1/2}
   \end{equation}
   is $L$-Lipschitz and belongs to $D^{1,2}(\W_2,\AA)$, in particular
   its $(2,\AA)$-relaxed gradient is bounded above by $L$ and satisfies
   \begin{equation}
     \label{eq:189}
     |\rmD F|_{\star,\AA}^2(\mu)\le
     \frac{1}{F^2(\mu)}\int_{\R^d}\vartheta^2|\nabla\vartheta|^2\,\d\mu\quad
     \text{for $\mm$-a.e.~$\mu\in \prob_2(\R^d)$ with $F(\mu)>0$}.
   \end{equation}
 \end{lemma}
 \begin{proof}
   Let $T\in \rmC^\infty(\R)$ be an odd, nondecreasing truncation function
   satisfying
   \begin{equation}
     \label{eq:163}
     T(x)=x\quad\text{if }|x|\le 1/2,\quad
     |T(x)| =1\quad\text{if }|x|\ge 2,\quad
     |T'(x)|\le 1,
   \end{equation}
   and let us set $T_n(x):=nT(x/n)$, $\vartheta_n:=T_n\circ\vartheta$, so that
   $\vartheta_n$ is $L$-Lipschitz and continuously differentiable in
   $\Omega_\vartheta$, so that 
   $\vartheta_n^2\in \rmC^1_b(\R^d)$.
   
   We define $\psi_n(r):=(r+1/n)^{1/2}$ and
   $F_n:=\psi_n\circ\lin{\vartheta_n^2}$.
   By construction $F_n\in \AA$ with
   \begin{align}
     \notag
     \rmD F_{n}(\mu,x)
     &=
       \frac 1{F_{n}(\mu)}\vartheta_n(x)\nabla\vartheta_n(x),\\
     \label{eq:190}
     \big(\lip F_n(\mu)\big)^2=\|\rmD F_n[\mu]\|^2
     &= 
     \frac1{F^2_{n}(\mu)}\int_{\R^d}\vartheta_n^2(x)|\nabla
       \vartheta_n(x)|^2\,\d\mu(x)
       \le
     L^2. 
   \end{align}
   Since $(\prob_2(\R^d),W_2)$ is a length space we deduce that $F_n$
   is $L$-Lipschitz. On the other hand $\lim_{n\to\infty} F_n(\mu)=F(\mu)$ pointwise
   everywhere, so that $F$ is $L$-Lipschitz as well, it belongs to
   $D^{1,2}(\W_2,\AA)$ and $|\rmD F|_{\star,\AA}\le L$.
   Passing eventually to the limit as $n\to\infty$ in \eqref{eq:190}
   for $\mu$ in the open set $\{\mu\in \prob_2(\R^d):F(\mu)>0\}$ we
   get \eqref{eq:189}.
 \end{proof}
 Selecting $\vartheta(x):=|x|$  and applying the first part of Lemma \ref{le:square-lip}  we immediately get  the following corollary.  
 \begin{corollary}
   \label{cor:preliminary1}
   The function $\rsqm\cdot$ as in \eqref{eq:183}  belongs to
   $D^{1,2}(\W_2,\AA)$ with
   \begin{equation}
     \label{eq:175}
     |\rmD \mathsf m_2|_{\star,\AA}(\mu)\le 1\quad\text{for $\mm$-a.e.~$\mu\in \prob_2(\R^d)$}.
   \end{equation}
 \end{corollary}

 We now use $\sfm_2$ for localizing gradient estimates in $\prob_2(\R^d)$.
 \begin{lemma}
   \label{le:limsup-approximation}
   Let $F_n$ be a sequence of functions in $D^{1,2}(\W_2,\AA) \cap L^\infty(\prob_2(\R^d), \mm)$ such
   that $F_n$ and 
   $|\rmD F_n|_{\star,\AA}$ are uniformly bounded in every bounded set
   of $\prob_2(\R^d)$ and let $F,G$ be Borel functions in
   $L^2(\prob_2(\R^d),\mm)$, $G$ nonnegative.
   If
   \begin{equation}
     \label{eq:170}
     \lim_{n\to\infty}F_n(\mu)= F(\mu),\quad
     \limsup_{n\to\infty}|\rmD F_n|_{\star,\AA}(\mu)\le G(\mu)
     \quad\text{$\mm$-a.e.~in $\prob_2(\R^d)$},
   \end{equation}
   then $F\in H^{1,2}(\W_2,\AA)$ and $|\rmD F|_{\star,\AA}\le G$.
 \end{lemma}
 \begin{proof}
   Let us consider a smooth nonincreasing function $\theta\in \rmC^\infty[0,+\infty)$ such
   that
   \begin{equation}
     \label{eq:176}
     \theta(r)=1\quad\text{if }0\le r\le 1,\quad
     \theta(r)=0\quad\text{if }r\ge 2,\quad
     |\theta'(r)|\le 2
   \end{equation}
   and set
   \begin{equation}
     \label{eq:177}
     \nchi_n(\mu):=\theta\big(\sfm_2(\mu)/n\big).
   \end{equation}
   By Corollary \ref{cor:preliminary1} we have
   \begin{equation}
     \label{eq:178}
     \nchi_n\in H^{1,2}(\W_2,\AA),\quad
     |\rmD \nchi_n|_{\star,\AA}\le 2/n,\quad
     |\rmD \nchi_n|_{\star,\AA}(\mu)=0  \text{ if }\rsqm\mu \le
     n\text{ or } \rsqm\mu\ge 2n.
   \end{equation}
   Thanks to the Leibniz rule, setting $
   F_{n,m}(\mu):=F_n(\mu)\nchi^2_m(\mu)$ and $G_n:=|\rmD
   F_n|_{\star,\AA}$, we have
   \begin{equation}
     \label{eq:179}
     F_{n,m}\in \rmD^{1,2}(\W_2,\AA),\quad
     |\rmD F_{n,m}|_{\star,\AA}(\mu)\le G_n(\mu)\nchi^2_{m}(\mu)+
     4/m F_{n}(\mu)\nchi_{m}(\mu).
   \end{equation}
   Since for every $m\in \N$ the sequence
   $n\mapsto G_n\nchi_m^2$ is uniformly bounded, we can find an
   increasing subsequence $k\mapsto n(k)$ such that
   $k\mapsto G_{n(k)}\nchi_m^2$ is
   weakly$^*$ convergent in $L^\infty(\prob_2(\R^d),\mm)$ and we
   denote by $\tilde G_m$ is
   weak$^*$ limit. By Fatou's lemma, for every Borel set
   $B\subset \prob_2(\R^d)$ we get
   \begin{align*}
     \int_B \tilde G_m\,\d\mm&=
     \lim_{k\to\infty}\int_B G_{n(k)}(\mu)\nchi^2_m(\mu)\,\d\mm(\mu)
                               \\&\le
     \int_B
     \limsup_{k\to\infty}\Big(G_{n(k)}(\mu)\nchi^2_m(\mu)\Big)\,\d\mm(\mu)\\
     &\le \int_B G^2\nchi_m^2\,\d\mm
   \end{align*}
   so that we deduce
   \begin{equation}
     \label{eq:181}
     \tilde G_m\le G^2\nchi_m^2\quad\text{$\mm$-a.e.~in
       $\prob_2(\R^d)$, for every $m\in \N$.}
   \end{equation}
   On the other hand, passing to the limit in \eqref{eq:179} along the
   subsequence $n(k)$ and recalling that
   $\lim_{k\to\infty}F_{n(k),m}=F\nchi_m^2$ $\mm$-a.e.~we get
   \begin{equation}
     \label{eq:182}
     |\rmD (F\nchi_m^2)|_{\star,\AA}(\mu)\le \tilde G_m(\mu)+\frac 4m F(\mu)\nchi_m(\mu)\le
     G(\mu)\nchi_m^2(\mu)+\frac 4m F(\mu)\nchi_m(\mu)
     \quad\text{for $\mm$-a.e.~$\mu\in \prob_2(\R^d)$.}
   \end{equation}
   We eventually pass to the limit as $m\to\infty$ concluding the proof
   of the Lemma.
 \end{proof}
 We now derive a natural estimate, extending \eqref{eq:Fdiff} to the
 case of quadratically coercive functions whose gradient has a linear growth.
 \begin{lemma}
   \label{le:speriamo-che-basti}
   Let $\phi\in \rmC^1(\R^d)$ be satisfying the growth conditions
      \begin{equation}
     \label{eq:186}
     \phi(x)\ge A|x|^2-B,\quad |\nabla\phi(x)|\le C(|x|+1)\quad\text{for
       every }x\in \R^d
   \end{equation}
   for given positive constants $A,B,C>0$
   and let $\zeta:\R\to\R$ be a $\rmC^1$ nondecreasing
   function whose derivative has compact support.
   Then
     the function
     $F(\mu):=\zeta\circ\lin\phi$
     is Lipschitz in $\prob_2(\R^d)$,
     it
     belongs to
   $H^{1,2}(\W_2,\AA)$, and
   \begin{equation}
     \label{eq:180}
     |\rmD F|_{\star,\AA}(\mu)\le \zeta'(\lin\phi(\mu))\Big(\int_{\R^d}|\nabla\phi(x)|^2\,\d\mu(x)\Big)^{1/2}.
   \end{equation}
 \end{lemma}
 \begin{proof}
   We set $\zeta_a(z) := (z+a)^{1/2}$ and $\vartheta_a:= \zeta_a \circ \phi$,  with $a:=A+B$,  so that 
   \begin{displaymath}
     \vartheta_a  \in \rmC^1(\R^d),\quad
     \vartheta_a  \ge \big(A(|x|^2+1)\big)^{1/2},\quad
     |\nabla \vartheta_a(x)  |=\frac{|\nabla\phi(x)|}{2(\phi(x)+a)^{1/2}}\le
     L,\quad
     L:=A^{-1/2}C
   \end{displaymath}
   for every $x\in \R^d$.

   We can then apply Lemma \ref{le:square-lip}, observing that
   \begin{displaymath}
     \big(\lin{  \vartheta_a^2  }(\mu)\big)^{1/2}= \zeta_a   \big(\lin\phi(\mu)\big);
   \end{displaymath}
   we deduce that $F_a=\zeta_a\circ \lin\phi$ is $L$-Lipschitz, it belongs to
    $D^{1,2}(\W_2,\AA)$   and satisfies (recall \eqref{eq:189})
   \begin{equation}\label{eq:191}
     |\rmD F_a|_{\star,\AA}(\mu)\le
     \frac{1}{2F_a(\mu)}\Big(\int_{\R^d} |\nabla(\vartheta_a^2)|^2\,\d\mu\Big)^{1/2}
     = \frac{1}{2F_a(\mu)}  \Big(\int_{\R^d}|\nabla\phi|^2\,\d\mu\Big)^{1/2}.
   \end{equation}
   We eventually observe that
   $ F =\psi_a\circ  F_a$, where $\psi_a(z)=\zeta(z^2-a)$  is a $\rmC^1$ Lipschitz function since $\zeta'$ has compact support.
   By the chain rule in Theorem \ref{thm:omnibus}(7) we get that \[|\rmD F|_{\star,\AA} = |\psi'_a \circ F_a| |\rmD F_a|_{\star,\AA}=2F_a \zeta'(\lin \phi)|\rmD F_a|_{\star,\AA}.\]
    Then \eqref{eq:191} yields \eqref{eq:180}.
 \end{proof}
We collect in the following definition
some useful tools and notation we will
extensively use.
\begin{definition}
\label{def:short}
We denote by 
$\kappa \in \rmC_c^{\infty}(\R^d)$ 
a smooth function satisfying $\supp{\kappa}=\overline{\rmB(0,1)}$, $\kappa(x) \ge 0$ for every $x \in \R^d$ and $\kappa(x)>0$ for every $x \in \rmB(0,1)$, $\int_{\R^d} \kappa \d  \Leb{d}  =1$ and $\kappa(-x)=\kappa(x)$ for every $x \in \R^d$. 

For every $0<\eps<1$ we define the family of associated mollifiers
\begin{equation}
    \kappa_\eps (x) := \frac{1}{\eps^d} \kappa (x/\eps) \quad x \in \R^d,
\end{equation}
and for every $\sigma \in \prbt$ we define
\begin{align} 
\sigma_\eps &:= \sigma \ast \kappa_\eps,\\
  \hat{\sigma}_\eps &:= \frac{\sigma_\eps \mres \rmB(0,1/\eps) +
                      \eps^{d+3}  \Leb{d} \mres
                      \rmB(0,1/\eps)}{\sigma_\eps(\rmB(0,1/\eps)) +
                      \eps^{d+3}  \Leb{d}(\rmB(0,1/\eps))}.
\end{align}
For every $\nu \in \prbt$ we eventually define the continuous functions $W_\nu, W_\nu^\eps, F_\nu^\eps : \prbt \to \R$ as
\begin{equation}
\label{eq:collect}
    W_\nu(\mu):= W_2(\mu, \nu), \quad W_\nu^{\eps}(\mu) := W_{\hat{\nu}_{\eps}}(\mu_\eps), \quad F^\eps_\nu (\mu) := \frac{1}{2}(W_\nu^{\eps}(\mu))^2, \quad \mu \in \prbt.
 \end{equation}
\end{definition}

Notice that $\sigma_\eps, \hat{\sigma}_\eps \in \prob_2^r(\R^d)$, $\supp{\hat{\sigma}_\eps} = \overline{\rmB(0,1/\eps)}$ and $W_2(\sigma_\eps, \sigma) \to 0$, $W_2(\hat{\sigma}_\eps, \sigma) \to 0$ as $\eps \downarrow 0$. Moreover, if $\sigma, \sigma' \in \prbt$, we have
\begin{equation}\label{eq:ineqconv}
    W_2(\sigma_\eps, \sigma'_\eps) \le W_2(\sigma, \sigma') \quad \text{ for every } 0<\eps<1
\end{equation}
and it is easy to check that, 
if we set
\begin{equation}\label{eq:kappamom}
    C_\eps := \rsqm {\kappa_\eps  \Leb{d}},
\end{equation}
then we have
\begin{equation}\label{eq:contrmom}
\rsqm{\mu_\eps} \le \rsqm{\mu} + C_\eps \quad \text{ for every } 0<\eps<1.
\end{equation}
\begin{proposition}\label{prop:fund} Let $\nu \in \prbt$, $\eps \in (0,1)$ and let $\zeta:\R\to\R$ be a $\rmC^1$ nondecreasing function whose derivative has compact support. With the notation of Definition 
\ref{def:short}
we have
  \begin{equation}
    \label{eq:174}
  \relgradA{(\zeta \circ F_\nu^\eps)}(\mu) \le
  \zeta'(F^\eps_\nu(\mu))\Big(\int_{\R^d} \left | x- \nabla
      (\varphi_\eps^* \ast \kappa_\eps)(x) \right |^2 \d \mu(x) \Big)^{1/2}\quad
    \text{ for } \mm \text{-a.e. }\mu \in \prbt,
  \end{equation}
where $\varphi_\eps^* = \Phi^*(\hat{\nu}_\eps, \mu_\eps)$ as in Theorem \ref{thm:ot}.
\end{proposition}
\begin{proof}
  Let $\mathcal{G}:=\{\mu^h\}_{h \in \N}$ be a dense and countable set
  in $\prbt$ and let us set, for every $h \in \N$, $\varphi_{ \eps,  h} :=
  \Phi(\hat{\nu}_\eps, \mu^h_\eps)$, $\varphi_{ \eps,  h}^* :=
  \Phi^*
  (\hat{\nu}_\eps, \mu^h_\eps)$ (see Theorem \ref{thm:ot}), 
\[ a_{ \eps,  h}:= \int_{\rmB(0,1/\eps)} \left ( \frac{1}{2}|y|^2 - \varphi_{ \eps,  h}(y) \right ) \d \hat{\nu}_\eps(y), \quad u_{ \eps,  h}(x):= \frac{1}{2} |x|^2 - \varphi_{ \eps,  h}^*(x) +a_{ \eps,  h}, \quad x \in \R^d \]
and
\[ G_{ \eps,  k}(\mu):= \max_{1 \le h \le k} \int_{\R^d} u_{ \eps,  h} \d \mu_\eps, \quad
  \mu \in \prbt.\]

We first observe that $ \varphi^*_{ \eps,  h}(x)$ is $1/\eps$-Lipschitz  (cf.~Theorem \ref{thm:ot}), so that
$|\varphi^*_{ \eps,  h}(x)|\le |x|/\eps$ and 
\begin{align}
  \label{eq:184}
  u_{ \eps,  h}(x)
  &\ge
  \frac 12 |x|^2-\frac 1\eps |x|+a_{ \eps,  h}\ge
    \frac 14|x|^2-\frac 1{\eps^2}+a_{ \eps,  h},\\
  u_{ \eps,  h}(x)
  &\le |x|^2+\frac 1{\eps^2}+a_{ \eps,  h}.
\end{align}

\emph{Claim 1}. It holds
\[ \lim_{k \to + \infty} G_{ \eps,  k}(\mu) = 
  F_\nu^{\eps}(\mu) \quad \text{ for every } \mu \in \prbt.\]
\emph{Proof of claim 1}. Since $G_{ \eps,  k+1}(\mu) \ge G_{ \eps,  k}(\mu)$ for every $\mu \in \prbt$, we have that 
\[ \lim_{k \to + \infty} G_{ \eps,  k}(\mu) = \sup_k G_{ \eps,  k}(\mu) = \sup_h \int_{\R^d} u_{ \eps,  h} \d \mu_\eps \quad \text{ for every } \mu \in \prbt.\]
By the definition of $\varphi_{  \eps, h}$ and $\varphi_{ \eps,  h}^*$  (see \eqref{eq:160}) we have that 
\[ \frac{1}{2} |x|^2 -\varphi^*_{ \eps, h}(x) + \frac{1}{2}|y|^2 -\varphi_{ \eps, h}(y) \le \frac{1}{2}|x-y|^2 \quad \text{ for every } x \in \R^d, \, y \in \rmB(0,1/\eps),\]
so that for every $\mu \in \prbt$ and $h \in \N$, we get
\begin{align*}
    \int_{\R^d} u_{ \eps, h} \d \mu_\eps &= \int_{\R^d} \left ( \frac{1}{2}|x|^2-\varphi_{ \eps, h}^*(x) \right ) \d \mu_\eps + \int_{\rmB(0,1/\eps)} \left ( \frac{1}{2}|y|^2 - \varphi_{ \eps, h}(y) \right ) \d \hat{\nu}_\eps(y)\\
                                & \le \frac{1}{2}\int_{\R^d \times \R^d} |x-y|^2 \, \d \ggamma(x,y) \\
                                &= \frac{1}{2} W_2^2(\mu_\eps, \hat{\nu}_\eps)\\
                &= 
                                  F_\nu^{\eps}(\mu),
\end{align*}
where $\ggamma \in \Gamma_o(\mu_{\eps},\hat{\nu}_\eps)$. 
This proves that $\sup_k G_{ \eps,  k}(\mu) \le
F_\nu^{\eps}(\mu)$ for every $\mu \in \prbt$.  If $\mu \in \mathcal{G}$, then we can find $h \in \N$ such that $\mu=\mu^h$ so that, by definition of $\varphi_{ \eps, h}$ and $\varphi_{ \eps, h}^*$, we obtain that 
\[ \int_{\R^d} u_{ \eps,  h} \d \mu_\eps = \int_{\R^d} \left ( \frac{1}{2}|x|^2-\varphi_{ \eps, h}^*(x) \right ) \d \mu_\eps + \int_{\rmB(0,1/\eps)} \left ( \frac{1}{2}|y|^2 - \varphi_{ \eps,  h}(y) \right ) \d \hat{\nu}_\eps(y) = \frac{1}{2}W_2^2(\mu_\eps^h,\hat{\nu}_\eps).\]
Hence, if $\mu \in \mathcal{G}$, then $\sup_k G_{ \eps,  k}(\mu) =
F_\nu^{\eps}(\mu)$.

Let now $\mu, \mu' \in \prbt$ and $h \in \N$ and observe that
\begin{align*}
    \int_{\R^d}u_{ \eps,  h} \d \mu_\eps - \int_{\R^d} u_{ \eps,  h} \d \mu'_\eps &= \frac{1}{2} \sqm{\mu_\eps} - \frac{1}{2} \sqm{\mu'_\eps} - \int_{\R^d} \varphi_{ \eps,  h}^* \d (\mu_\eps-\mu'_\eps) \\
    & \le \frac{1}{2}
    \Big(\rsqm{\mu_\eps} + \rsqm{\mu'_\eps} \Big)W_2(\mu_\eps, \mu'_\eps) + \frac{1}{\eps}W_2(\mu_\eps, \mu'_\eps) \\
     &\le \frac{1}{2}\Big(\rsqm{\mu} + \rsqm{\mu'}+2C_\eps\Big) W_2(\mu, \mu') + \frac{1}{\eps}W_2(\mu, \mu') \\
    &  \le \Big ( \rsqm{\mu} +  \rsqm{\mu'} +C_\eps + \frac{1}{\eps} \Big ) W_2(\mu, \mu'),
\end{align*}
 where we used \eqref{eq:ineqconv}, \eqref{eq:contrmom},  the fact that $\varphi^*_{ \eps,  h}$ is $1/\eps$-Lipschitz continuous and \eqref{eq:123}. We hence deduce that for every $k \in \N$
\begin{equation}\label{eq:12}
    \left |G_{ \eps,  k}(\mu)-G_{ \eps,  k}(\mu') \right | \le \left (\rsqm{\mu} + \rsqm{\mu'} + \frac{1}{\eps}+
    C_\eps  \right ) W_2(\mu, \mu') \quad \text{ for every } \mu, \mu' \in \prbt.
\end{equation}
Choosing $\mu' \in \mathcal{G}$ and passing to the limit as $k \to + \infty$ we get from \eqref{eq:12}
\[ \left | \lim_{k \to + \infty} G_{ \eps,  k}(\mu)-F_\nu^\eps(\mu') \right | \le \left (\rsqm{\mu} + \rsqm{\mu'}  +C_\eps  + \frac{1}{\eps} \right ) W_2(\mu, \mu') \quad \text{ for every } \mu \in \prbt, \, \mu' \in \mathcal{G}.\]
Using the density of $\mathcal{G}$ and the continuity of $\mu' \mapsto F_\nu^{\eps}(\mu')$ we deduce that 
\[ \lim_{k \to + \infty} G_{ \eps,  k}(\mu) =
  F_\nu^\eps(\mu) \quad \text{ for every } \mu \in \prbt\]
proving the first claim.\\ 
\emph{Claim 2}.
If $H_{ \eps,  k}:= \zeta\circ G_{ \eps,  k}$ and $ \tilde{u}_{\eps, h} :=u_{ \eps,  h}\ast \kappa_\eps  \in \rmC^1(\R^d) $ it holds
\[ \relgradA{H_{ \eps,  k}}^2(\mu) \le
  \big(\zeta'(G_{ \eps,  k}(\mu))\big)^2\int_{\R^d}\left | \nabla  \tilde{u}_{\eps, h}  \right |^2  \d \mu(x) =
  \big(\zeta'(G_{ \eps,  k}(\mu))\big)^2\int_{\R^d}\left | x-\nabla (\varphi_{ \eps,  h}^* \ast \kappa_\eps)(x) \right |^2 \d \mu(x), \]
for $\mm$-a.e.~$\mu \in B_{ \eps,  h}^k$, where $B_{ \eps,  h}^k := \{ \mu \in \prbt \mid G_{ \eps,  k}(\mu) = \int_{\R^d} u_{ \eps,  h} \d \mu_\eps \}$, $h \in \{1, \dots, k\}$.\\
\emph{Proof of claim 2}. For every $h \in \N$,
\eqref{eq:184} yields
\begin{equation}
  \label{eq:193}
   \tilde{u}_{\eps, h} (x)\ge \frac 14|x|^2
  -\frac 1{\eps^2}+a_{ \eps,  h},\quad
  |\nabla  \tilde{u}_{\eps, h} (x)|\le 
  |x|+\frac 1\eps 
  ;
\end{equation}
where we used that
\[|x |^2\ast \kappa_\eps\ge |x\ast \kappa_\eps|^2=
|x|^2,\quad
\nabla u_{ \eps,  h}(x)=x-\nabla\varphi_{ \eps,  h}^*(x),\quad
\nabla u_{ \eps,  h}\ast \kappa_\eps=x-\nabla\varphi_{ \eps,  h}^*\ast\kappa_\eps.\]
Since  the map $\ell_{ \eps,  h}: \prbt \to \R$ defined as $\ell_{ \eps,  h}(\mu):=
\int_{\R^d} u_{ \eps,  h} \d \mu_\eps $ satisfies
\[ \ell_{ \eps,  h}(\mu)= \int_{\R^d} (u_{ \eps,  h} \ast \kappa_\eps) \d \mu=\lin{ \tilde{u}_{\eps, h} }(\mu), \quad \mu \in \prbt,\]
Lemma \ref{le:speriamo-che-basti} and the
above estimates yield
\[ \relgrad{(\zeta\circ \ell_{ \eps,  h})}(\mu) \le
  \zeta'(\ell_{ \eps,  h}(\mu))\Big(\int_{\R^d} \left | \nabla
     \tilde{u}_{\eps, h} 
  \right |^2 \d \mu \Big)^{1/2}\quad
  \text{for $\mm$-a.e. } \mu \in \prbt.\]
 Since $\zeta$ is nondecreasing,  $H_{ \eps,  k}$ can be written as
\[ H_{ \eps,  k}(\mu) = \max_{1 \le h \le k} (\zeta\circ \ell_{ \eps,  h})(\mu), \quad \mu \in \prbt,\]
 so that  we can apply Theorem \ref{thm:omnibus} (8) and conclude the proof of
the second claim.\\

\emph{Claim  3  }. Let $(h_n)_n \subset \N$ be a non-decreasing 
sequence and
let $\mu \in \prbt$. If $\lim_n \int_{\R^d} u_{ \eps,  h_n} \d \mu_\eps  = 
F_\nu^\eps(\mu)$, then 
\[ \lim_n \int_{\R^d} \left | x- \nabla ( \varphi_{ \eps,  h_n}^* \ast \kappa_\eps )(x) \right |^2 \d \mu(x) = \int_{\R^d} \left | x- \nabla ( \varphi_\eps^* \ast \kappa_\eps )(x) \right |^2 \d \mu(x),\]
where $\varphi^*_\eps= \Phi^*(\hat{\nu}_\eps, \mu_\eps)$.
\quad \\
\emph{Proof of claim  3 }.
Let us set for every $n \in \N$
\[ \phi_{ \eps,  n} := \varphi_{ \eps,  h_n}, \quad \psi_{ \eps,  n}=\phi^*_{ \eps,  n}:= \varphi^*_{ \eps,  h_n}.\]

We will show that from any (non relabeled) increasing  subsequence  it is
possible to extract a further subsequence $j\mapsto {n(j)}$ such
that 
\begin{displaymath}
  \lim_j \int_{\R^d} \left | x- \nabla ( \phi^*_{ \eps,  {n(j)}} \ast
      \kappa_\eps )(x) \right |^2 \d \mu(x) = 
    \int_{\R^d} \left | x- \nabla ( \varphi_\eps^* \ast \kappa_\eps )(x) \right |^2 \d \mu(x).
\end{displaymath}
 By Theorem \ref{thm:ot}, we have that, for every $n \in \N$,
 $\phi^*_{ \eps,  n}: \R^d \to \R$ is convex and $1/\eps$-Lipschitz continuous
 with $\phi^*_{ \eps,  n}(0)=0$, $\phi_{ \eps,  n}: \rmB(0, 1/\eps) \to \R$ is convex and continuous and 
 \[
   \phi_{ \eps,  n}(x)= \sup_{y \in \R^d }  \la x, y \ra - \phi^*_{ \eps,  n}(y) 
 \quad \text{ for every } x \in \R^d.\]
Moreover, since  $F^\eps_\nu \ge0$ and $|\phi^*_{ \eps,  n}(x)|\le \eps^{-1}|x|$ ,
for $n$ sufficiently large we have 
\begin{equation}
  \label{eq:171}
      a_{ \eps,  h_n}=\int_{\R^d}\Big(u_{ \eps,  h_n}(x)
      +\phi^*_{ \eps,  n}(x)-\frac12
  |x|^2\Big)\,\d\mu_\eps(x)\ge -C(\eps,\mu) 
\end{equation}
 with $C(\eps,\mu):= 1+\frac{1}{2} \sqm{\mu_\eps} +\frac1\eps {\rsqm{\mu_\eps}}$. 
It follows that 
\begin{align*}
  \int_{\rmB(0,1/\eps)}\phi_{ \eps,  n}\,\d\hat\nu_\eps
  &=\frac12\rsqm{\hat\nu^\eps}-a_{ \eps,  h_n}\le C'(\eps,\mu,\nu)
\end{align*}
where $ C'(\eps,\mu,\nu):=C(\eps,\mu)+\frac 12\sqm{\hat{\nu}_\eps}$ .
Since $\hat \nu^\eps\ge  \frac{\eps^{d+3}}{1+\eps^3 \omega_d}\Leb d\mres\rmB(0,1/\eps)  $
we can find a constant $I=I(\eps,\mu,\nu)$ such that
$\int_{\rmB(0,1/\eps)}\phi_{ \eps,  n}\,\d x\le I$ for sufficiently large $n$.

Thus by Lemma \ref{le:convan}, we get the existence of a subsequence
$j \mapsto n(j)$ and two  convex  continuous functions $\phi^*_{ \eps }: \R^d \to \R$ and $\phi_{ \eps}: \rmB(0, 1/\eps)
\to \R$
such that points (i), (ii), (iii)  and conclusions of Lemma \ref{le:convan} hold.
By points (i) and (ii) we can use Fatou Lemma and the dominated convergence Theorem to conclude that 
\[ \liminf_j \int_{\rmB(0,1/\eps)} \phi_{ \eps,  n(j)} \d \hat{\nu}_\eps \ge \int_{\rmB(0,1/\eps)} \phi_{ \eps} \d \hat{\nu}_{\eps}, \quad 
\lim_j \int_{\R^d} \phi^*_{ \eps,  n(j)} \d \mu_\eps = \int_{\R^d} \phi^*_{ \eps} \d \mu_{\eps}.\]
We thus deduce that 
\[ \int_{\R^d} \left ( \frac{1}{2} |x|^2 - \phi^*_{ \eps}(x) \right )\d
  \mu_\eps(x) + \int_{\rmB(0,1/\eps)} \left ( \frac{1}{2} |y|^2 -
    \phi_{ \eps}(y) \right )\d \hat{\nu}_\eps(y) \ge \limsup_j \int_{\R^d}
  u_{ \eps,  h_{n(j)}} \d \mu_\eps =
  F_\nu^\eps(\mu) \]
proving that
\[\int_{\rmB(0,1/\eps)} \phi_{ \eps}\, \d \hat{\nu}_\eps + \int_{\R^d} \phi^*_{ \eps}
  \,
  \d \mu_\eps = \frac{1}{2} \sqm{\hat{\nu}_\eps} + \frac{1}{2} \sqm{\mu_\eps} - \frac{1}{2}W_2^2(\hat{\nu}_\eps, \mu_\eps).\]
By the uniqueness part of Theorem \ref{thm:ot}
we deduce that $\phi_{ \eps}  = \varphi_{\eps} =\Phi(\hat{\nu}_\eps, \mu_\eps)$ and $\phi^*_{ \eps}  = \varphi^*_\eps  =
\Phi^*(\hat{\nu}_\eps, \mu_\eps)$. Finally, the a.e.~convergence of
the gradient of $\phi^*_{ \eps,  n}$ to the gradient of $\phi^*_{ \eps}$ given by point
(iii) in Lemma \ref{le:convan} gives that $\nabla (\phi^*_{ \eps,  n(j)} \ast
\kappa_\eps) \to \nabla (\phi^*_{ \eps} \ast \kappa_\eps)$ pointwise
everywhere. Moreover, since for every $x \in \R^d$ we have $x\ast
\kappa_\eps=x$ and 
\[ \left |x-\nabla (\varphi^*_{ \eps,  n(j)} \ast \kappa_\eps)(x) \right |^2 \le 
 \big(|x| + 1/\eps\big)^2 \in L^1(\R^d, \mu), \]
we can use the dominated convergence Theorem to conclude that 
\[ \lim_j \int_{\R^d} \left |x-\nabla (\phi^*_{ \eps,  n(j)} \ast \kappa_\eps)(x) \right |^2 \d \mu(x) = \int_{\R^d} \left |x-\nabla (\phi^*_{ \eps} \ast \kappa_\eps)(x) \right |^2 \d \mu(x).\]
This concludes the proof of the  third  claim.

\emph{Claim  4 }. It holds
\[ \limsup_k \relgradA{H_{ \eps,  k}}(\mu) \le \zeta'(F^\eps_\nu(\mu))
  \Big(\int_{\R^d}\left | x-\nabla (\varphi_\eps^* \ast \kappa_\eps)(x) \right |^2 \d \mu(x)\Big)^{1/2} \quad \text{ for $\mm$-a.e. } \mu \in \prbt,\]
where $\varphi^*_\eps= \Phi^*(\hat{\nu}_\eps, \mu_\eps)$.\\
\emph{Proof of claim 4}. Let $B_{ \eps} \subset \prbt$ be defined as
\[ B_{ \eps}:= \bigcap_k \bigcup_{h=1}^k A^k_{ \eps,  h},\]
where $A^k_{ \eps,  h}$ is the full $\mm$-measure subset of $B_{ \eps,  h}^k$ where claim 2 holds. Notice that $B_{ \eps}$ has full $\mm$-measure. Let $\mu \in B_{ \eps}$ be fixed and let us pick a  non-decreasing  sequence $k \mapsto h_k$ such that 
\[ G_{ \eps,  k}(\mu) = \int_{\R^d} u_{ \eps,  h_k} \d \mu_\eps.\] 
By claim 1 we know that $G_{ \eps, k}(\mu) \to F_\nu^\eps(\mu)$ so that we can apply claim 4 and conclude that 
\begin{equation}\label{eq:nuova} \zeta'(F^\eps_\nu(\mu))\int_{\R^d} \left | x- \nabla (
    \varphi_\eps^* \ast \kappa_\eps )(x) \right |^2 \d \mu(x) =
  \lim_{k} \zeta'(G_{ \eps,  k}(\mu))\int_{\R^d} \left | x- \nabla ( \varphi^*_{ \eps,  h_k} \ast \kappa_\eps )(x) \right |^2 \d \mu(x).
  \end{equation}
By claim 2, the right hand side is greater than $\limsup_k
\relgradA{H_{ \eps,  k}}^2(\mu)$; this concludes the proof of the  fourth  claim.

\emph{Conclusion.}
 We conclude the proof applying Lemma \ref{le:limsup-approximation} with $(H_{ \eps, k})_k$ in the role of $(F_n)_n$, $F:=\zeta \circ F_\nu^\eps$ and $G$ given by 
\[ G(\mu):= \zeta'(F^\eps_\nu(\mu))\Big(\int_{\R^d} \left | x- \nabla
      (\varphi_\eps^* \ast \kappa_\eps)(x) \right |^2 \d \mu(x) \Big)^{1/2}, \quad \mu \in \prob_2(\R^d).\]
We check that the hypotheses of Lemma \ref{le:limsup-approximation} are satisfied: by Claim 2 we have that $H_{ \eps, k} \in D^{1,2}(\W_2,\AA)$ and it is also in $L^\infty(\prob_2(\R^d), \mm)$ since $\zeta$ is uniformly bounded. 
Notice that, by \eqref{eq:193}, for every $R>0$ it holds that
\begin{equation}
  \label{eq:192}
  \Big(\int_{\R^d} \left | \nabla \tilde{u}_{\eps, h} (x)
  \right |^2 \d \mu(x)\Big)^{1/2} \le R+1/\eps \quad\text{whenever }\rsqm\mu\le R.
\end{equation}
This gives, also using Claim 2, that $|\rmD H_{ \eps, k}|_{\star, \AA}$ is uniformly bounded on bounded subsets of $\prob_2(\R^d)$ (recall that $\zeta'$ is uniformly bounded). It is also clear that $H_{ \eps, k}$ is uniformly bounded on bounded subsets of $\prob_2(\R^d)$ since it is uniformly bounded by the infinity norm of $\zeta$.

The function $F$, being bounded again by the infinity norm of $\zeta$, belongs to $L^2(\prob_2(\R^d), \mm)$. The same holds for $G$: using \eqref{eq:nuova} and passing to the limit the estimate in \eqref{eq:192} we see that $G$ is uniformly bounded, having $\zeta'$ compact support.

By Claim 1 and Claim 4 we have
\[ \lim_{k \to + \infty} H_{ \eps, k}(\mu) = F(\mu), \quad \limsup_{k \to +\infty} |\rmD H_{ \eps, k}|_{\star, \AA}(\mu) \le G(\mu) \quad \text{ for $\mm$-a.e.~$\mu \in \prob_2(\R^d)$}.\]

By Lemma \ref{le:limsup-approximation} we get \eqref{eq:174}.
\end{proof}
 We still keep the notation of 
Definition \ref{def:short}.
\begin{corollary}
  \label{cor:bound}
  Let $\nu \in \prbt$. Then 
  \begin{equation}
\relgradA{W_\nu}(\mu) \le 1\quad \text{for $\mm$-a.e. } \mu \in
  \prbt.\label{eq:196}
\end{equation}
\end{corollary}
\begin{proof}
  First of all we prove that for every $0< \eps <1$, it holds 
\begin{equation}\label{eq:witheps}
  \int_{\R^d} \left | x- \nabla
    (\varphi_\eps^* \ast \kappa_\eps)(x) \right |^2 \d \mu(x)
  \le W^2_2(\mu_\eps, \hat{\nu}_\eps) \quad \text{ for every } \mu \in \prbt,
\end{equation}
 where $\varphi_\eps^* = \Phi^*(\hat{\nu}_\eps, \mu_\eps)$ as in Theorem \ref{thm:ot}. 
Since
\[ \left |x - \nabla (\varphi_\eps^* \ast \kappa_\eps)(x) \right |^2 \le \left |x-\nabla \varphi_\eps^*(x) \right |^2 \ast \kappa_\eps(x) \quad \text{ for every } x \in \R^d,\]
we get that
\begin{align*}
                                  \int_{\R^d} \left | x- \nabla
                                  (\varphi_\eps^* \ast \kappa_\eps)(x)
                                  \right |^2 \d \mu(x)
                                   & \le \int_{\R^d} \left ( \left |x-\nabla \varphi_\eps^*(x) \right |^2 \ast \kappa_\eps(x) \right ) \d \mu(x) \\
 & = \int_{\R^d} \left |x-\nabla \varphi_\eps^*(x) \right |^2 \d \mu_{\eps}(x) \\
                                &=
                                  W^2_2(\mu_\eps, \hat{\nu}_\eps),
\end{align*}
for every $\mu \in \prbt$, where the last equality comes from
Theorem \ref{thm:ot}. This proves \eqref{eq:witheps}.
 It  follows from Proposition \ref{prop:fund}  that, for every 
 nondecreasing function $\zeta\in \rmC^1(\R)$ whose derivative has compact support, it holds  
\begin{equation}\label{eq:withzetano}
  |\rmD \,(\zeta\circ F^\eps_\nu)|_{\star,\AA}(\mu)
  \le
    \zeta'(F^\eps_\nu (\mu)) \sqrt{2F^\eps_\nu(\mu)} \quad \text{ for $\mm$-a.e.~$\mu \in \prbt$}.
\end{equation}
Let us now consider a sequence of continuous and compactly supported functions $\alpha_n: \R \to \R$ such that 
\[ 0 \le \alpha_n(s) \uparrow \frac{\nchi_{(0,+\infty)}(s)}{1+s^2} \le 1,  \quad \text{ for every } s \in \R\]
and let us define $\zeta_n(s): \R \to \R$ as
\[ \zeta_n(s) = \int_0^s \alpha_n(r) \d r, \quad s \in \R.\]
Then, for every $n \in \N$, $\zeta_n: \R \to \R$ is a $\rmC^1$ nondecreasing  function  whose derivative has compact support so that we can plug it into \eqref{eq:withzetano} in place of $\zeta$ and we see that 
\begin{equation}\label{eq:withzetano1}
  |\rmD \,(\zeta_n\circ F^\eps_\nu)|_{\star,\AA}(\mu)
  \le
    \zeta_n'(F^\eps_\nu (\mu)) \sqrt{2F^\eps_\nu(\mu)} \quad \text{ for $\mm$-a.e.~$\mu \in \prbt$}.
\end{equation}
Observe that $\zeta_n(s) \to \arctan(s) \nchi_{(0,+\infty)}$ and $\zeta'_{ n}(s) \to \frac{\nchi_{(0,+\infty)}(s)}{1+s^2}$ for every $s \in \R$ and the r.h.s.~of \eqref{eq:withzetano1} is uniformly bounded. Using Theorem \ref{thm:omnibus}(1)-(3) we can thus pass to the limit as $n \to + \infty$ and we obtain
\[\relgradA{(\vartheta \circ W_\nu^{\eps})}(\mu) 
\le \vartheta'(W_\nu^{\eps}( \mu )) \quad \text{ for $\mm$-a.e.~$\mu \in \prbt$},\]
where $\vartheta:\R \to \R$ is defined as $\vartheta(s):=\arctan(s^2/2) \nchi_{(0,+\infty)}(s)$, $s \in \R$. We can thus apply Lemma \ref{le:truncations} and conclude that 
\begin{equation}
  \label{eq:195}
  |\rmD W^\eps_\nu|_{\star,\AA}\le 1 \quad \text{  $\mm$-a.e.~and for every $0<\eps<1$ }.
\end{equation}
Choosing $\eps=1/k$, we have $\lim_{k \to + \infty} W_\nu^{1/k}(\mu) = W_\nu (\mu)$
for every $\mu \in \prbt$;
using Theorem \ref{thm:omnibus} (1)-(3), we obtain \eqref{eq:196}.
\end{proof}

The \textbf{proof of Theorem \ref{thm:main}}
then easily follows by Corollary \ref{cor:bound} and Theorem
\ref{theo:startingpoint}.

\medskip\noindent
We conclude this section with a simple but useful density property,
which shows the possibility to use smaller algebra of cylinder
functions to operate in $H^{1,2}(\W_2)$.
\begin{proposition}
  \label{prop:density}
  Let $\mathscr F$ be a subset of $\rmC^1_b(\R^d)$
  satisfying
  the following property: for every $f\in \rmC^1_b(\R^d)$ there
  exists
  a sequence $f_n\in \mathscr F$, $n\in \N$, such that
  \begin{equation}
    \label{eq:83}
     \sup_{n} \|f_n\|_{\infty}+\|\nabla f_n\|_\infty  <\infty,\quad
    \lim_{n\to\infty}\int_{\R^d}|f_n-f|+|\nabla
    (f_n-f)|\,\d\mu=0
    \quad\text{for $\mm$-a.e.~$\mu\in \prbt$}.
  \end{equation}
  Then the algebra $\AA\subset\cyl1b \prbt$ generated by the set
  of cylinder functions $\big\{\lin f: f\in
  \mathscr F\big\}$ is dense in $H^{1,2}(\W_2)$ and satisfies the strong
  approximation property of Theorem \ref{thm:main}.

   In particular the algebra $\ccyl \infty b{\prbt}$
  generated by $\big\{\lin f:f\in \rmC^\infty_c(\R^d)\big\}$ is
  strongly dense in $H^{1,2}(\W_2)$ and satisfies the approximation
  property of Theorem \ref{thm:main}.
\end{proposition}
\begin{proof}
  Thanks to Theorem \ref{thm:main} and a simple diagonal argument, it
  is sufficient to prove that for every cylinder function $F\in
  \cyl1b \prbt$ there exists a sequence $F_n\in \AA$ such that
  \begin{equation}
    \label{eq:85}
    F_n \to F \text{ in } L^2(\prbt, \mm)\quad\text{and}\quad  \pCE_2(F_n-F)\to0\quad\text{as }n\to\infty.
  \end{equation}
  In the case $F=\lin f$
  with $f\in \rmC^1_b(\R^d)$,
  \eqref{eq:83} and Lebesgue Dominated Convergence Theorem
  show that we can
  find a sequence $f_n\in \FF$ such that, setting  $F_n:=\lin{f_n}$, we have 
  \begin{align*}
    \int_\prbt|F_n-F|^2\,\d\mm
    &=
    \int_\prbt
      \Big|\int_{\R^d}(f_n(x)-f(x))\,\d\mu(x)\Big|^2\,\d\mm(\mu)\to0
      \quad\text{as }n\to\infty,\\
    \pCE_2(F_n-F)
    &=
    \int_\prbt \int |\nabla
      f_n(x)-\nabla f(x)|^2\,\d\mu(x)\,\d\mm(x)\to0\quad\text{as }n\to\infty.               
  \end{align*}
  Let us now consider a general $F=\psi\circ\lin{\boldsymbol f}$
  as in \eqref{eq:cyl}, where $\boldsymbol f=(f_1,\cdots, f_N)$ is a
  vector of functions in $\rmC^1_b(\R^d)$ and
  $\psi\in \rmC^1_b(\R^N)$. If we consider $\tilde \ff:=(1,f_1, \dots, f_N)$ and $\tilde{\psi} \in \rmC^1_b(\R^{N+1})$ defined as
  \[ \tilde{\psi}(x_0, x_1, \dots, x_n) := \psi(0)x_0-\psi(0)+\psi(x_1, x_2, \dots, x_N), \quad (x_0,x_1, \dots, x_N) \in \R^{N+1},\]
  we have that $\tilde{\psi}(0)=0$ and $\tilde{\psi}\circ\lin{\tilde\ff} = F$. For this reason we can  always suppose that $f_1\equiv 1$ and $\psi(0)=0$.
  It is also not restrictive to assume that $\psi$ is a polynomial with $\psi(0)=0$:
  in fact, setting $R:=\sup_{\R^d, \, 1 \le k \le N}\Big(|f_{k}|+|\nabla f_{k}|\Big)$,
  we can find a sequence of polynomials $(P_h)_h$ in $\R^N$ such
  that
  \begin{equation}
    \label{eq:86}
    P_h(0)=0,\quad
    \sup_{|z|\le R}|P_h(z)-\psi(z)|+|\nabla
    P_h(z)-\nabla \psi(z)|\to0\quad\text{as }h\to\infty.
  \end{equation}
  It follows that
  $F_h:=P_h\circ\lin\ff$ satisfies
  \begin{equation}
    \label{eq:87}
    \lim_{h\to\infty}\sup_\prbt \Big(|F_h(\mu)-F(\mu)|+\|\rmD
    F_h[\mu]-\rmD F[\mu]\|_\mu\Big)=0.
  \end{equation}
   Let us consider  sequences $(f_{k,n})_{n\in \N}$,
  $k=1,\cdots, N$, approximating $f_k$ as in \eqref{eq:83}.
  In particular, there exists $R>0$ such that
  $\sup_{\R^d}\Big(|f_{k,n}|+| \nabla  f_{k,n}|+|f_k|+| \nabla  f_k|\Big)\le
  R$
  for every $n\in \N,\ k\in \{1,\cdots,N\}$.
  If $\psi$ is a polynomial in $\R^N$ with $\psi(0)=0$ then
  the function $F_n:=\psi\circ \lin{\ff_n}$ belongs
  to $\AA$ (cf.~Remark \ref{rem:unrem}), where $\ff_n = (f_{1,n}, f_{2,n}, \dots, f_{N,n})$.  Denoting by $L$ the maximum of the Lipschitz constants of
  $\psi$ and $\partial_k\psi$
  in the cube $[-R,R]^N$ with respect to the $\infty$-norm, it is easy to see that
  \begin{align*}
    |F_n(\mu)-F(\mu)|
    &=
      \Big|\psi(\lin{\ff_n}(\mu)) -  \psi(\lin{\ff_n}(\mu)))\Big|\le
      L \sup_{k}|\lin{f_{k,n}}(\mu)-\lin{f_k}(\mu)|\to0%
      ,\\
    \|\rmD
    F_n[\mu]-\rmD F[\mu]\|_\mu
    &=\Big\|\sum_k \Big(\partial_k\psi(
      \lin{\ff_n}(\mu))\nabla f_{k,n}-
      \partial_k\psi(
      \lin\ff(\mu))\nabla f_{k}\Big)\Big\|_\mu
    \\
    &\le
      \sum_k \Big\|\partial_k\psi(
      \lin{\ff_n}(\mu))\nabla f_{k,n}-
      \partial_k\psi(
      \lin{\ff_n}(\mu))\nabla f_{k}\Big\|_\mu\\&\quad+
      \sum_k \Big\|\Big(\partial_k\psi(
    \lin{\ff_n}(\mu))-
      \partial_k\psi(
    \lin\ff(\mu))\Big)\nabla f_{k}\Big\|_\mu
    \\&\le L \sum_k \left ( \Big\|\nabla f_{k,n}-
    \nabla f_{k}\Big\|_\mu+R \Big|\langle f_{k,n}-f_k,\mu\rangle\Big| \right ).
  \end{align*}
  Both terms are uniformly bounded w.r.t.~$\mu$ and $n$, and
  converge to $0$ as $n\to\infty$. We deduce that \eqref{eq:85} holds.
\end{proof}
\begin{remark}[Polynomials]
  If there exists a  radius  $R>0$ such that 
  $\supp(\mu)\subset \overline{\rmB(0,R)}$ for   $\mm$-a.e.~$\mu$  
  then we can also choose subsets $\FF$ of $\rmC^1(\R^d)$ in
  Proposition \ref{prop:density}.
  An interesting example is provided by the  collection  $\FF$ of all
  the polynomials.
  In this case the algebra $\AA$ is the set of functionals
  \begin{displaymath}
    \mu\mapsto
    \int_{(\R^d)^k}P(x_1,\cdots,x_k)\,\d\mu^{\otimes
      k}(x_1,\cdots,x_k),\quad
    P\text{ polynomial in }(\R^d)^k,\quad
    k\in \N.
  \end{displaymath}
\end{remark}

\section{Calculus rules}
\label{sec:calculus}
We now show that the Cheeger energy can be expressed in terms of
an appropriate notion of (relaxed) Wasserstein gradient, also depending on
$\mm$, which enjoys useful calculus rules.
\begin{theorem}[$\mm$-Wasserstein differential]
  \label{prop:rel-diff}
  For every $F\in D^{1,2}(\W_2)$ 
    there exists a unique vector field $\rmD_\mm F\in L^2(\prbt\times
    \R^d,\bmm;\R^d)$ (the $\mm$-Wasserstein differential of $F$)
    such that for every sequence
    $F_n\in \cyl1b \prbt$, $n\in \N$, satisfying 
  \eqref{eq:55}
  we have
  \begin{equation}
    \label{eq:56}
    \rmD F_n\to \rmD_\mm F\quad\text{strongly in } L^2(\prbt\times
  \R^d,\bmm;\R^d).
\end{equation}
Moreover:
\begin{enumerate}[\rm (a)]
\item
  The map $F\mapsto \rmD_\mm F$ from $D^{1,2}(\W_2)$ to
  $L^2(\prbt\times \R^d,\bmm;\R^d) $ is linear and for every
  $F,G\in D^{1,2}(\W_2)$ we have
  \begin{equation}
    \label{eq:57}
    \begin{aligned}
      \CE_2(F,G)=\int \rmD_\mm F(\mu,x)\cdot \rmD_\mm
      G(\mu,x)\,\d\bmm(\mu,x),\quad
      \CE_2(F)=\int |\rmD_\mm
      F(\mu,x)|^2\,\d\bmm(\mu,x),
    \end{aligned}
  \end{equation}
  where $\CE_2(\cdot, \cdot)$ denotes the quadratic form associated to $\CE_2(\cdot)$ as in Remark \ref{rem:prec}.
\item
  The map $F\mapsto (F,\rmD_\mm F)$ is a linear isometric (thus
  continuous) immersion of
  $H^{1,2}(\W_2)$ into $L^2(\prbt,\mm)\times L^2(\prbt\times \R^d,\bmm;\R^d)
  $. 
\item
  The graph of $\rmD_\mm$ in $L^2(\prbt,\mm)\times
  L^2(\prbt\times\R^d,\bmm;\R^d)$ is (weakly) closed:
  for every sequence $F_n\in H^{1,2}(\W_2)$
  \begin{equation}
    \label{eq:80}
    \left.
      \begin{aligned}
        F_n\weakto F&\text{ in }L^2(\prbt,\mm)\\
        \rmD_\mm F_n\weakto
        \gG&\text{ in }L^2(\prbt\times \R^d,\bmm;\R^d)
      \end{aligned}
      \right\}
    \quad\Rightarrow\quad
    F\in H^{1,2}(\W_2),\ \gG=\rmD_\mm F.
  \end{equation}
\end{enumerate}
\end{theorem}
\begin{proof}
  The proof uses well known arguments of the theory of quadratic
  forms. If $F_n$, $n\in \N$, is a sequence in $\cyl1b \prbt$, then  
  for every $m,n\in \N$ we have
  \begin{equation}
    \label{eq:58}
    \frac14\pCE_2(F_m-F_n)=
    \frac12\Big(
    \pCE_2(F_m)+\pCE(F_n)\Big)
    - \pCE_2\Big(\frac12(F_m+F_n)\Big).
  \end{equation}
  If \eqref{eq:55} holds, we can pass to the limit as $m,n\to\infty$,
  observing that
  $\lim_{m,n\to\infty}\frac12(F_m+F_n)=F$, 
  and therefore by \eqref{eq:relpre}  
  $\liminf_{m,n\to\infty}\pCE_2\Big(\frac12(F_m+F_n)\Big)\ge
  \CE_2(F)$;
  we thus obtain
  \begin{equation}
    \label{eq:59}
    \limsup_{m,n\to\infty} \frac14\pCE_2(F_m-F_n)=
    \limsup_{m,n\to\infty}
    \frac14\int |\rmD F_m(\mu,x)-\rmD F_n(\mu,x)|^2\,\d\bmm(\mu,x)\le 0
  \end{equation}
  which shows that $ m \mapsto \rmD F_m$ is a Cauchy sequence in
  $L^2(\prbt\times\R^d,\bmm;\R^d)$ and therefore converges to some element $\vV$.

  If $\tilde F_n$ is another sequence satisfying \eqref{eq:55}, we can
  use the identity
  \begin{equation}
    \label{eq:60}
    \frac14\pCE_2(F_n-\tilde F_n)=
    \frac12\Big(
    \pCE_2(F_n)+\pCE(\tilde F_n)\Big)
    - \pCE_2\Big(\frac12(F_n+\tilde F_n)\Big)
  \end{equation}
  and the same argument to conclude that
  $\lim_{n\to\infty}\pCE_2(F_n-\tilde F_n)=0$, so that the limit $\vV$
  is independent of the approximating sequence and we are authorized
  to call it $\rmD_\mm F$.

  Concerning claim (a),
  the linearity of $\rmD_\mm$ follows immediately from the linearity
  of $\rmD$  as a map  from $\cyl1b \prbt$ to $L^2(\prbt\times\R^d,\bmm;\R^d)$.
  
  If $F,G \in D^{1,2}(\W_2)$ and $(F_n)_n,(G_n)_n \subset \cyl1b \prbt$ are sequences satisfying \eqref{eq:55} for $F$ and $G$ respectively, we can see that $\pCE_2(F_n, G_n) \to \CE_2(F,G)$; indeed
  \begin{align*}
      \pCE_2(F_n, G_n) &= \frac{1}{2} \pCE_2(F_n+G_n) - \frac{1}{2}\pCE_2(F_n) - \frac{1}{2} \pCE_2(G_n),\\
    &= -\frac{1}{2} \pCE_2(F_n-G_n) + \frac{1}{2}\pCE_2(F_n) + \frac{1}{2} \pCE_2(G_n).
  \end{align*}
  Passing the first equality to the $\liminf_n$, the second one to the $\limsup_n$ and using \eqref{eq:relpre}, we get that $\pCE_2(F_n, G_n) \to \CE_2(F,G)$. Passing then  to the limit in  \eqref{eq:54}  
  we immediately see that
  \begin{equation}
    \label{eq:82}
    \CE_2(F,G)=\int \rmD_\mm F(\mu,x)\cdot \rmD_\mm
      G(\mu,x)\,\d\bmm(\mu,x)
    \end{equation}
    which, together with \eqref{eq:ce2sp},  shows that $F\mapsto (F,\rmD_\mm F)$ is an isometry
    from
    $H^{1,2}(\W_2)$ into $L^2(\prbt,\mm)\times
    L^2(\prbt\times\R^d,\bmm;\R^d)$ (claim (b)).

     Claim (c) then follows by claim (b) and the fact that
    $H^{1,2}(\W_2)$ is a Hilbert space.
\end{proof}
Let us now collect a few properties of $\rmD_\mm F$, which follow by
the corresponding
metric versions of Theorem \ref{thm:omnibus} and
the approximation property of Theorem \ref{prop:rel-diff}.
\begin{proposition}[Calculus properties of $\rmD_\mm F$]
  \label{prop:calculus}
  The $\mm$-Wasserstein differential satisfies the following properties:
    \begin{enumerate}[\rm (a)]
    \item (Minimal relaxed gradient and pointwise Lipschitz constant)
      For every
       $F\in D^{1,2}(\W_2)$ 
      we have
      \begin{equation}
        \label{eq:63}
        \|\rmD_\mm F[\mu]\|_\mu^2=
        \int |\rmD_\mm F (\mu,x)|^2\,\d\mu(x)=
        |\rmD F|_\star^2(\mu)\quad
        \text{for $\mm$-a.e.~$\mu\in \prob_2(\R^d)$}.
      \end{equation}
      In particular we have the
       \emph{pointwise Rademacher property}:
      for every $F\in \Lip_b(\prbt)$
    \begin{equation}
      \label{eq:61}
      \|\rmD_\mm F[\mu]\|_\mu^2=
      \int |\rmD_\mm F (\mu,x)|^2\,\d\mu(x)\le
        \big|\rmD F\big|^2 (\mu) 
      \quad\text{for $\mm$-a.e.~$\mu\in \prob_2(\R^d)$},
      \end{equation}
      and if $F\in \cyl1b \prbt$
      \begin{equation}
        \label{eq:62}
        \int |\rmD_\mm F (\mu,x)|^2\,\d\mu(x)
        \le
        \int |\rmD F (\mu,x)|^2\,\d\mu(x)\quad\text{for $\mm$-a.e.~$\mu\in \prob_2(\R^d)$}.
      \end{equation}
    \item (Leibniz rule)
      If $F,G\in L^\infty(\prbt,\mm)\cap D^{1,2}(\W_2)$, then
    $H:=FG\in   D^{1,2}(\W_2)$ and
    \begin{equation}
      \label{eq:5ee}
      \rmD_\mm H(\mu,x)= F(\mu)\rmD_\mm G(\mu,x)+G(\mu)\rmD_\mm
      F(\mu,x)
      \quad\text{for $\bmm$-a.e.~$(\mu,x)\in \prbt \times \R^d$}.
    \end{equation}
  \item (Locality)
    If $F\in  D^{1,2}(\W_2)$ then
    for any $\LL^1$-negligible Borel subset $N\subset \R$ we have
    \begin{equation}
      \label{eq:7ee}
      \rmD_\mm F[\mu]=0\quad\text{in }L^2(\R^d,\mu;\R^d)\quad\text{$\mm$-a.e.~on $F^{-1}(N)$}.
    \end{equation}
  \item (Truncations) If $F_j\in  D^{1,2}(\W_2)$, $j=1,\cdots,J$,
      then
      also the functions \[ F_+:=\max(F_1,\cdots,F_J) \text{ and }
      F_-:=\min(F_1,\cdots,F_J)\] belong to $ D^{1,2}(\W_2)$
      and
      \begin{align}
        \label{eq:9ee}
        \rmD_\mm F_+=\rmD_\mm F_j&\quad\text{$\bmm$-a.e.~on }
                                                       \{(\mu,x)\in
                                                       \prbt\times\R^d:F_+(\mu)=F_j(\mu)\},\\
        \rmD_\mm F_-=\rmD_\mm F_j&\quad\text{$\bmm$-a.e.~on }
                                   \{(\mu,x)\in \prbt\times\R^d:F_-(\mu)=F_j(\mu)\}.
      \end{align}
    \item (Chain rule)
    If $F\in D^{1,2}(\W_2)$ and $\phi\in \Lip(\R)$
    then $\phi\circ F\in  D^{1,2}(\W_2)$ and
    \begin{equation}
      \label{eq:8ee}
      \rmD_\mm (\phi\circ F)= \phi'(F)\,\rmD_\mm F  \quad\text{$\bmm$-a.e.~on }\prbt \times \R^d.
    \end{equation}
  \end{enumerate}
\end{proposition}
\begin{remark}
  Notice that the product in \eqref{eq:8ee} is well defined since there
  exists a $\Leb 1$-negligible Borel set $N\subset \R$ such that
  $\phi$ is differentiable in $\R\setminus N$  and
  $\rmD_\mm F$ vanishes $\mm$-a.e.~in $F^{-1}(N)$ thanks to
  the locality property \eqref{eq:7ee}.
\end{remark}
\begin{proof}
  Claim (a) is an immediate consequence of the fact that
  \eqref{eq:55}
  yields $\lip F_n\to |\rmD F|_\star$ strongly in $L^2(\prbt,\mm)$; up to
  extracting a suitable (not relabeled) subsequence we get
  $\int |\rmD F_n [\mu]|^2\,\d\mu\to |\rmD F|^2_\star(\mu)$
  for $\mm$-a.e. $\mu$. On the other hand  since by \eqref{eq:56} $\rmD F_n \to \rmD_\mm F$ in $L^2(\prbt\times
  \R^d,\bmm;\R^d)$, then $|\rmD F_n|^2 \to |\rmD_\mm F|^2$ in $L^1(\prbt\times
  \R^d,\bmm)$; indeed
  \begin{align*}
    \int\Big||\rmD F_n(\mu,x)|^2-|\rmD_\mm
    F(\mu,x)|^2 \Big| \,\d\bmm(\mu,x) &\le \int \left ( (|\rmD F_n| + |\rmD_\mm F|) |\rmD F_n -\rmD_\mm F| \right ) \, \d \bmm \\
    & \le \left ( \int (|\rmD F_n| + |\rmD_\mm F|)^2\, \d \bmm \right )^{1/2} \\
    &\quad  \quad \cdot \| \rmD F_n - \rmD_\mm F\|_{L^2(\prbt\times
  \R^d,\bmm;\R^d)}
  \end{align*}
  so that  
  \begin{equation}\label{eq:64}
     \int\Big||\rmD F_n(\mu,x)|^2-|\rmD_\mm
    F(\mu,x)|^2 \Big|  \,\d\bmm(\mu,x) \to 0 \quad \text{ as } n \to + \infty.
  \end{equation}
   Hence  Fubini's Theorem yields, up to extracting a suitable
  subsequence,
  \begin{equation}
    \label{eq:65}
    \int|\rmD F_n(\mu,x)|^2\,\d\mu\to \int |\rmD_\mm
    F(\mu,x)|^2\,\d\mu\quad\text{for $\mm$-a.e.$\mu\in \prbt$}.
  \end{equation}
  \eqref{eq:61} and \eqref{eq:62} then follows by the general
  properties of the minimal relaxed gradients
   (recall Remark \ref{rem:N}).

  Claim (c) follows by \eqref{eq:7} and \eqref{eq:63}.

  Claim (d) is just a consequence of the locality property
  \eqref{eq:7ee}.
  
  Claim (e) is true if $\phi\in \rmC^1_b(\R)$ just by
  passing to the limit in the corresponding formula
  for a cylinder function.
  In fact if $F_n\in \cyl1b\prbt$ is a sequence as in \eqref{eq:55} and
  \eqref{eq:56} we have
  \begin{equation}
   \rmD(\phi \circ F_n) = (\phi'\circ F_n )\rmD F_n\quad\text{in
   }\domG. \label{eq:125}
 \end{equation}
 Since $\phi'$ is bounded and continuous we get
 \begin{equation}
   \label{eq:126}
   \rmD(\phi \circ F_n)\to \gG=(\phi'\circ F)\rmD_\mm F\quad\text{strongly in
   }L^2(\prbt\times\R^d,\bmm;\R^d)\quad
   \text{as $n\to\infty$}.
 \end{equation}
 Integrating w.r.t.~$\bmm$ and recalling \eqref{eq:63} and Theorem \ref{thm:omnibus}(7) we get
 \begin{displaymath}
   \int |\gG|^2\,\d\bmm=
   \int |\phi'(F(\mu))|^2|\rmD_\mm F(\mu,x)|^2\,\d\bmm(\mu,x)=
   \int |\phi'(F(\mu))|\,|\rmD F|_\star^2\,\d\mm=
   \CE_2(\phi\circ F)
 \end{displaymath}
 so that
 \begin{displaymath}
   \lim_{n\to\infty}\pCE_2(\phi\circ F_n)=\CE_2(\phi\circ F).
 \end{displaymath}
 We conclude by Theorem \ref{prop:rel-diff} that $\gG=(\phi'\circ
 F)\rmD_m F$ coincides with $\rmD_\mm(\phi\circ F)$. 
  
  Let us now consider the case of a general Lipschitz function $\phi$; by truncation and Claim
  (d)  it is not restrictive to assume that $\phi$ is also bounded.
  We can find a sequence $\phi_n\in \rmC^1_b(\R)$ such
  that $\sup_\R |\phi_n|+|\phi_n'|\le L<\infty$,
  $\phi_n\to \phi$ uniformly, and $\phi_n'(x)\to\phi'(x)$
  for every $x\in \R\setminus N$ for a Borel set $N$ with $\Leb
  1(N)=0$.
  We have
  \begin{equation}
    \label{eq:66}
    \rmD_\mm (\phi_n\circ F)=\phi_n'(F)\rmD_\mm F \quad\text{$\mm$-a.e.~in }\prbt.
  \end{equation}
  Setting $\tilde N:=\{(\mu,x)\in \overline \domG:F(\mu)\in N\}$,
  Fubini's Theorem and the locality property \eqref{eq:7ee} yield
  $\rmD_\mm F(\mu,x)=0$ for $\bmm$-a.e.~$(\mu,x)\in \tilde N$.
  On the other hand $\phi_n'(F(\mu))\to \phi'(F(\mu))$ for every
  $(\mu,x)\in \domG\setminus \tilde N$; since $\phi_n'$ is uniformly
  bounded, we deduce that
  \begin{equation}
    \label{eq:67}
    \phi_n'(F)\rmD_\mm F\to \phi'( F)\rmD_\mm F\quad\text{strongly in }L^2(\prbt\times\R^d,\bmm;\R^d).
  \end{equation}
  We conclude by Theorem \ref{prop:rel-diff}(b)  that $\rmD_\mm (\phi\circ F)=\phi'(F)\rmD_\mm F$.
  
  Claim (b) follows by claim (e); indeed, since $F,G \in L^\infty(\prbt,\mm)$, we can find a constant $M>0$ such that 
  \[ |F|(\mu) \le M, \quad  |G|(\mu) \le M, \quad  |F+G|(\mu) \le M \quad\text{for $\mm$-a.e.~$\mu\in \prbt$}.\]
  Let $\phi \in \Lip(\R)$ be such that $\phi(x)=x^2$ for every $x \in [-M-1,M+1]$; then we have
  \begin{align*}
      \rmD_\mm FG &= \frac{1}{2} \rmD_\mm ((F+G)^2) - \frac{1}{2}\rmD_\mm (F^2) -\frac{1}{2} \rmD_{ \mm} (G^2) \\
      &= \frac{1}{2} \rmD_\mm (\phi \circ(F+G)) - \frac{1}{2}\rmD_\mm (\phi \circ F) -\frac{1}{2} \rmD_{ \mm} (\phi \circ G) \\
      &= \frac{1}{2}\phi'(F+G)\rmD_\mm(F+G) - \frac{1}{2}\phi'(F)\rmD_\mm F -\frac{1}{2} \phi'(G)\rmD_{ \mm} G \\
      &= (F+G)\rmD_\mm(F+G) - F\rmD_\mm F -G\rmD_{ \mm} G \\
      &= F\rmD_\mm G+G\rmD_\mm
      F
  \end{align*}
  for $\bmm$-a.e.~$(\mu,x) \in \prbt \times \R^d$.
 
\end{proof}
\begin{corollary}
  \label{cor:Dirichlet}
    $\CE_2$ is a local Dirichlet form in $L^2(\prbt,\mm)$
  \cite[3.1.1]{Bouleau-Hirsch91}
  enjoying $\Gamma$-calculus with Carr\'e du champs $\Gamma$ given by
  \begin{equation}
    \label{eq:81}
    \Gamma(F,G)[\mu]:=\int \rmD_\mm F(\mu,x)\cdot \rmD_\mm
    G(\mu,x)\,\d\mu(x)
      \quad\text{for $\mm$-a.e.~$\mu\in \prbt$}.
  \end{equation}
  In particular, for every
  $F,G\in H^{1,2}(\W_2)$ we have
  \begin{equation}
    \label{eq:57ee}
    \begin{aligned}
      \CE_2(F,G)&=\int_\prbt \Gamma(F,G)[\mu]\,\d\mm(\mu)=\int \rmD_\mm F(\mu,x)\cdot \rmD_\mm
      G(\mu,x)\,\d\bmm(\mu,x),\\
      \CE_2(F)&=\int_\prbt \Gamma(F,F)\,\d\mm(\mu)=\int |\rmD_\mm
      F(\mu,x)|^2\,\d\bmm(\mu,x).
    \end{aligned}
  \end{equation}
\end{corollary}
\begin{proof}
  The fact that $\CE_2$ is a Dirichlet form follows by the truncation
  property
  \eqref{eq:8ee} with $\phi(r):=r\land 1$.
  Since $\CE_2(1)=0$, the same property with $\phi(r)=|r|$ also shows
  that $\CE_2$ is local
  (see \cite[Corollary 5.1.4]{Bouleau-Hirsch91}).

  Using the Leibniz rule \eqref{eq:5ee} one can also easily show that
  the $\Gamma$-operator \eqref{eq:81} is the Carr\'e du champ
  associated to $\frac12 \CE_2$ \cite[Definition 4.1.2]{Bouleau-Hirsch91}.
\end{proof}

\subsection{Tangent bundle, residual differentials and relaxation}\label{subsec:resdiff}
In general we cannot guarantee that
$\CE_2(F)$ coincides with $\pCE_2(F)$ if $F\in \cyl1b \prbt$, or,
equivalently, that $\rmD_\mm F=\rmD F$: this 
property corresponds to the closability of $\pCE_2$.
We can however investigate the relations between
$\rmD F$ and $\rmD_\mm F$:
 two useful tools are represented by the
closure of the graph of $\rmD$
 and by the
collection of all the weak limits of Wasserstein differentials along vanishing  sequences.

\begin{definition}[Multivalued gradient]
  We denote by $\brmG\subset L^2(\prbt,\mm)\times
  L^2(\prbt\times\R^d,\bmm;\R^d)$
  the closure of the linear space $\big\{(F,\rmD F):F\in
  \cyl1b{\prbt}\big\}$.
  The multivalued gradient $\mgrad:H^{1,2}(\W_2)\rightrightarrows
   L^2(\prbt\times\R^d,\bmm;\R^d)$ is the operator whose graph is $\brmG$.
 \end{definition}
 It is clear that $\brmG$ is a closed vector  subspace  of
 $ L^2(\prbt,\mm)\times
  L^2(\prbt\times\R^d,\bmm;\R^d)$, which can also be obtained as the
  weak closure of $\big\{(F,\rmD F):F\in
  \cyl1b{\prbt}\big\}$. Thus
  $\vV\in \mgrad F$ if and only if there exists a sequence $F_n\in
  \cyl1b{\prbt}$ such that
  \begin{equation}
    \label{eq:128}
    F_n\to F\text{ in }L^2(\prbt,\mm),\quad
    \rmD F_n\weakto \vV\text{ in }L^2(\prbt\times \R^d,\bmm;\R^d).
  \end{equation}
\begin{definition}[Residual gradients]
  \label{def:kernel}
  The set of residual gradients $\Residual\subset L^2(\prbt\times \R^d,\bmm;\R^d)$
  is defined as
  \begin{equation}
    \label{eq:68}
    \begin{aligned}
      \Residual:=\Big\{\vV\in{} &L^2(\prbt\times \R^d,\bmm;\R^d):\text{
        there exists }(F_n)_{n\in \N}\subset
      \cyl1b \prbt: \\
      &F_n\to 0\text{ in }L^2(\prbt,\mm),\ \rmD F_n\weakto \vV
      \text{ in }L^2(\prbt\times \R^d,\bmm;\R^d)\Big\}.
    \end{aligned}
  \end{equation}
\end{definition}
 The notion of residual gradient is known in the literature, see e.g.~\cite[Section 1.2]{tesi}. 
Notice that $\pCE_2$ is closable if and only if $\Residual$ is
trivial  and that $\Residual$ contains all the vector fields that are limits of gradients of vanishing sequence of functions (see also Lemma \ref{le:unlemma}(1)). 
A third important space is the $L^2$ tangent bundle of $\prbt$.
 In the following, given a Borel map $\gG \in \mathcal L^2(\prbt\times \R^d, \bmm; \R^d)$, we denote, for every $\mu \in \prbt$, by $\gG[\mu]$ the map $x \mapsto \gG(\mu,x)$. 
\begin{definition}
  We denote by $\Tan(\prbt,\bmm)$ the subspace of
  $L^2(\prbt\times\R^d,\bmm;\R^d)$ of vector fields $\vV$ satisfying
  \begin{equation}
    \label{eq:130}
    \vV[\mu]\in \Tan_\mu\prbt\quad\text{for $\mm$-a.e.~$\mu\in \prbt$}.
  \end{equation}
\end{definition}
\begin{lemma}
  \label{le:allTan}
  $\Tan(\prbt,\bmm)$ is a closed subspace of
  $L^2(\prbt\times\R^d,\bmm;\R^d)$ which is a $L^\infty(\prbt,\mm)$
  module:
  \begin{equation}
    \label{eq:131}
    \text{for every }\vV\in \Tan(\prbt,\bmm),\ H\in
    L^\infty(\prbt,\mm):\quad
    H\vV\in \Tan(\prbt,\bmm).
  \end{equation}
  For every $ F  \in H^{1,2}(\W_2)$ (resp.~ $F\in \cyl1b{\prbt}$) $\rmD_\mm
  F\in \Tan(\prbt,\bmm)$ (resp.~$\rmD F\in \Tan(\prbt,\bmm)$).
  Finally, if $\mathscr C\subset \rmC^\infty_c(\R^d)$ is a countable
  set dense in $\rmC^\infty_c(\R^d)$ with respect to the 
  Lipschitz norm $\|\zeta\|_{\Lip}:=\sup_{\R^d}|\zeta|+|\nabla\zeta|$
  and $\mathscr L$ is a countable set dense in $L^2(\prbt,\mm)$ then
  the set
  \begin{equation}
    \mathscr T=\operatorname{span}\Big\{H \nabla\zeta:H\in \mathscr L,\ \zeta\in \mathscr
    C\Big\}\quad
    \text{is dense
  in $\Tan(\prbt,\bmm)$}.\label{eq:38}
\end{equation}
\end{lemma}
\begin{proof}
  Let $(\vV_n)_{n\in \N}$ be a sequence in $\Tan(\prbt,\bmm)$ strongly
  converging to $\vV$ in $L^2$; it is not restrictive to assume that $\vV_n$
  are Borel maps satisfying $\vV_n[\mu]\in \Tan_\mu\prbt$ for every $\mu\in
  \prbt\setminus \cN$ for a $\mm$-negligible set of $\prbt$.
  Up to extracting a suitable subsequence, we can also assume that
  $\sum_{n=1}^\infty\|\vV_n-\vV\|_{L^2}^2<\infty$.
  Applying Fubini's Theorem it follows that
  \begin{displaymath}
    \int_{\prbt}\Big( \sum_{n=1}^\infty \int_{\R^d}|\vV_n[\mu](x)-\vV[\mu](x)|^2\,\d\mu(x)\Big)\,\d\mm<+\infty
  \end{displaymath}
  so that there exists a $\mm$-negligible set $\cN'\supset\cN$ such that
  \begin{displaymath}
     \sum_{n=1}^\infty
     \int_{\R^d}|\vV_n[\mu](x)-\vV[\mu](x)|^2\,\d\mu(x)<\infty
     \quad\text{for every }\mu\in \prbt\setminus\cN';
   \end{displaymath}
   and this implies that $\vV_n[\mu]\to \vV[\mu]$ strongly in
   $L^2(\prbt,\mu ; \R^d )$, so that
   $\vV[\mu]\in \Tan_\mu\prbt$ for every $\mu\in \prbt\setminus\cN'$.

   \eqref{eq:131} is obvious. Since for every $F=\lin\phi$, $\phi\in
   \rmC^1_b$ $\rmD F[\mu]=\nabla\phi\in \Tan_\mu\prbt$ for every
   $\mu\in \prbt$, it is immediate to check that $\rmD F\in  \Tan(\prbt,\bmm) $
   for every cylinder function.
   The closure property of $\Tan(\prbt,\bmm)$ then yields the
   analogous conclusion for the Wasserstein differential of $\rmD_\mm
   F$ of a Sobolev function $F\in H^{1,2}(\W_2)$.

   Let us eventually consider \eqref{eq:38}: it is sufficient to prove
   that any $\vV\in \mathscr T^\perp$ belongs to
   $\big(\Tan(\prbt,\bmm)\big)^\perp$, where $\perp$ denotes the orthogonal complement in the Hilbert space $L^2(\prbt\times\R^d,\bmm;\R^d)$. 
   If $\vV\in \mathscr T^\perp$ is a Borel vector field, then
   \begin{displaymath}
     \int_{\prbt}\Big(\int\langle \nabla \zeta,\vV(\mu,x)\rangle\,\d\mu(x)\Big)\,H(\mu)\,\d\mm(\mu)=0
   \end{displaymath}
   for every $\zeta\in \mathscr C,\ H\in \mathscr L$. Since $\mathscr
   L$ is dense in $L^2(\prbt,\mm)$ we have for every $\zeta\in
   \mathscr C$
   \begin{displaymath}
     \int\langle \nabla \zeta,\vV(\mu,x)\rangle\,\d\mu(x)=0
     \quad\text{for $\mm$-a.e.~$\mu\in \prbt$}
   \end{displaymath}
   Since $\mathscr C$ is countable, we can find a $\mm$-negligible set
   $\cN\subset \prbt$ such that
   \begin{displaymath}
     \int\langle \nabla \zeta,\vV(\mu,x)\rangle\,\d\mu(x)=0
     \quad\text{for every $\zeta\in \mathscr C$ and every $\mu\in \prbt\setminus\cN$}
   \end{displaymath}
   which shows that $\vV[\mu]\in \Big(\Tan_\mu\prbt\Big)^\perp$ for
   every $\mu\in \prbt\setminus\cN$, so that for every $\wW\in
   \Tan(\prbt,\bmm)$
   \begin{align*}
     \int \langle\vV(\mu,x),\wW(\mu,x)\rangle\,\d\bmm
     &=
       \int_{\prbt}\Big(\int_{\R^d}
       \langle\vV[\mu](x),\wW[\mu](x)\rangle\,\d\mu(x)\Big)\,\d\mm(\mu)=0.\qedhere
   \end{align*}
\end{proof}

Let us collect a few simple properties of $\Residual$. 
\begin{lemma}\label{le:unlemma}
  Let $\Residual$ be as in \eqref{eq:68}.
  \begin{enumerate}[\rm (1)]
      \item  $\Residual$ is a closed subspace of $ L^2(\prbt\times
    \R^d,\bmm;\R^d)$ and coincides with the set
    \begin{equation}
    \mgrad 0=\big\{\vV\in
    L^2(\prbt\times
    \R^d,\bmm;\R^d):(0,\vV)\in \rmG\big\}.\label{eq:127}
  \end{equation}
  \item For every $\vV\in\Residual$ there exists a sequence
    $F_n\in
    \cyl1b \prbt,\ n\in \N,\ \text{such that}$
    \begin{equation}
      \label{eq:69}
      F_n\to 0\text{ in }L^2(\prbt,\mm),\ \rmD F_n\to \vV
      \text{ strongly in }L^2(\prbt\times \R^d,\bmm;\R^d).
    \end{equation}
    Every element $\vV\in \Residual$ is therefore characterized by the
    property
    \begin{equation}
      \label{eq:71}
      \forall\,\eps>0\ \ \exists\,F\in \cyl1b \prbt:\quad
      \|F\|_{L^2(\prbt,\mm)}\le\eps,\quad
      \|\rmD F-\vV\|_{L^2(\prbt\times \R^d,\bmm;\R^d)}\le \eps.
    \end{equation}
  \item $\Residual$ satisfies the locality property
    \begin{equation}
      \label{eq:70}
      \text{for every }\vV\in \Residual,\ H\in L^\infty(\prbt,\mm):\quad
      H\vV\in \Residual.
    \end{equation}
\end{enumerate}
\end{lemma}
\begin{proof} We have already observed that $\rmG$ is a closed vector space,
  coinciding with the weak closure of $\big\{(F,\rmD F):F\in
  \cyl1b{\prbt}\big\}$; in view of \eqref{eq:128}, \eqref{eq:68} precisely characterizes the
  elements $\vV$ for which $(0,\vV)\in \rmG$. Therefore the first two
  claims are obvious. 

  Let us eventually prove the last claim.
  We first consider the case when $H\in \cyl1b \prbt$.
  If $\vV\in \Residual$ we can find a sequence $F_n\in \cyl1b \prbt$ such
  that
  \eqref{eq:69} holds.
  Setting $G_n:= HF_n$, since $H$ is bounded we clearly have $G_n\to
  0$ strongly in $L^2(\prbt,\mm)$;
  moreover, by the Leibniz rule we get
  \begin{equation}
    \label{eq:73}
    \rmD G_n=H\rmD F_n+F_n\rmD H\to H\vV
  \end{equation}
  since $\rmD H\in L^\infty(\prbt\times  \R^d,\bmm;\R^d)$ and $F_n\to 0$
  strongly in $L^2(\prbt,\mm)$.
  We deduce that $H\vV\in \Residual$ as well.

  If now $H$ is a function in $L^\infty(\prbt,\mm)$ we can find by \eqref{eq:isdense} a uniformly bounded sequence $H_n\in \cyl1b \prbt$ converging to $H$
  $\mm$-a.e.~in $\prbt$, so that $H_n\vV\to H\vV$ in
  $L^2(\prbt\times\R^d,\bmm;\R^d)$.
  Being $\Residual$ a closed subspace and $H_n\vV\in \Residual$ by the
  previous step, we deduce that $H\vV\in \Residual$.
\end{proof}
We now define
\begin{equation}
  \label{eq:132}
  \Tangent:=\Tan(\prbt,\bmm)\cap \Residual^\perp=\Big\{\vV\in
  \Tan(\prbt,\bmm):\langle\vV,\wW\rangle_{L^2}=0\ \text{for every
  }\wW\in \Residual\Big\},
\end{equation}
 where $\perp$ denotes the orthogonal complement in the Hilbert space $L^2(\prbt\times\R^d,\bmm;\R^d)$. 
We can now obtain our main structure result.
\begin{theorem}
  \label{thm:decomposition}
  For every $F\in H^{1,2}(\W_2)$ we have $\rmD_\mm F\in \Tangent$ and for every
  $\vV\in\Residual$ we have
  the pointwise orthogonality property
  
  \begin{equation}
    \label{eq:75}
    \int_{\R^d} \rmD_\mm
    F(\mu,x)\cdot\vV(\mu,x)\,\d\mu(x)=0\quad\text{for
      $\mm$-a.e.~$\mu\in \prbt$}.
  \end{equation}
  If $\vV\in \mgrad F$ 
  then 
  $\vV-\rmD_\mm F\in \Residual$.
  In particular for every $F\in \cyl1b \prbt$ $\rmD F-\rmD_\mm F\in
  \Residual$ and for every $G\in H^{1,2}(\W_2)$
  \begin{equation}
    \label{eq:79}
    \int_{\R^d}\rmD_\mm F(\mu,x)\cdot \rmD_\mm G(\mu,x)\,\d\mu(x)=
    \int_{\R^d}\rmD F(\mu,x)\cdot \rmD_\mm G(\mu,x)\,\d\mu(x)
    \quad
    \text{for $\mm$-a.e.~$\mu\in \prbt$}.
  \end{equation}
   Finally,  for every $F \in H^{1,2}(\W_2)$,  $\rmD_\mm F$ is the element of minimal $L^2$-norm in $\mgrad F$.
\end{theorem}
\begin{proof}
  Let us first observe that if $F_n\in \cyl1b \prbt$ satisfies \eqref{eq:69}
  and $\tilde F_n\in \cyl1b \prbt$ satisfies \eqref{eq:56}, we have
  $F_n+\tilde F_n\to F$ strongly in $L^2(\prbt,\mm)$, with
  $\rmD(F_n+\tilde F_n)\to \rmD_\mm F+\vV$, so that  the lower semicontinuity of the Cheeger energy with respect to $L^2$ convergence yields together with \eqref{eq:57} that  
  \begin{equation}
    \label{eq:76}
    \CE_2(F)=\int |\rmD_\mm F|^2\,\d\bmm\le
    \int |\rmD_\mm F+\vV|^2\,\d\bmm.
  \end{equation}
  Since $\vV$ is arbitrary in $\Residual$ we deduce that
  \begin{displaymath}
    \int \rmD_\mm F\cdot \vV\,\d\bmm=0\quad\text{for every }\vV\in \Residual.
  \end{displaymath}
  Replacing $\vV$ with $H\vV$, $H\in L^\infty(\prbt,\mm)$ we get
  \begin{equation}
    \label{eq:77}
    \int_\prbt \Big(\int_{\R^d}\rmD_\mm F\cdot \vV\,\d\mu(x)\Big)
    H(\mu)\,\d\mm(\mu)=0\quad\text{for every }\vV\in \Residual,\ H\in L^\infty(\prbt,\mm),
  \end{equation}
  which yields \eqref{eq:75}.

  If now $F_n\in \cyl1b \prbt$ converges strongly to $F$ with $\rmD F\weakto
   \gG$, selecting $\tilde F_n$ as above, we have
  $F_n-\tilde F_n\to 0$ strongly in $L^2(\prbt,\mm)$ and $\rmD(F_n-\tilde
  F_n)\weakto \gG-\rmD_\mm F$ weakly in $L^2(\prbt\times
  \R^d,\bmm;\R^d)$, so that $\gG-\rmD_\mm F\in \Residual$.
   By \eqref{eq:76} we conclude that $\rmD_\mm F$ is the element
  of minimal norm in $\mgrad F=\rmD_\mm F+\Residual$.
\end{proof}
We can give a ``pointwise'' interpretation of the orthogonality
properties of the previous Theorem.
Let us select an orthonormal set $\ocR:=\{\vV_n:n\in \N\}\subset \mathcal L^2(\prbt\times \R^d, \bmm; \R^d)$ dense in
$\Residual$ (we are thus assuming that $\vV_n$ are Borel vector fields
everywhere
defined).
Since
\begin{displaymath}
  \int_\prbt\Big(\int_{\R^d}|\vV_n(\mu,x)|^2\,\d\mu(x)\Big)\,\d\mm(\mu)=1
\end{displaymath}
we deduce that there exists a $\mm$-negligible set $\cN\subset \prbt$ such that
\begin{equation}
  \label{eq:78}
  \int_{\R^d}|\vV_n(\mu,x)|^2\,\d\mu(x)<\infty\quad\text{for every
  }n\in \N,\ \mu\in \prbt\setminus\cN.
\end{equation}
We thus define $\Residual[\mu]:=\overline{\mathrm{span}\{\vV_n[\mu]:n\in
  \N\}}\subset L^2(\R^d,\mu; \R^d)$ for
every $\mu\in \prbt\setminus\cN$ and
$\Tangent[\mu]:=\big(\Residual[\mu]\big)^\perp\cap \Tan_\mu\prbt$, where here $\perp$ denotes the orthogonal complement in the Hilbert space $L^2(\R^d,\mu;\R^d)$. 
\begin{theorem}
  \label{thm:pointwise}
  Let $F\in H^{1,2}(\W_2)$ and $\vV\in L^2(\prbt\times\R^d,\bmm;\R^d)$.
  \begin{enumerate}[\rm(1)]
  \item 
    $\vV$ belongs to
  $\Residual$ if and only if,  for $\mm$-a.e.~$\mu$,  $\vV[\mu]\in
  \Residual[\mu]$.
\item
  $\vV$ belongs to
   $\Tangent$ if and only if,  for $\mm$-a.e.~$\mu$,  $\vV[\mu]\in
  \Tangent[\mu]$.
  \item
  $\rmD_\mm F[\mu]\in \Tangent[\mu]$ 
  for $\mm$-a.e.~$\mu$. 
\item If $F\in \cyl1b \prbt$ then, 
    for $\mm$-a.e.~$\mu\in \prbt$, 
  $\rmD_\mm F[\mu]$ is the $L^2(\R^d,\mu)$-orthogonal projection of 
  $\rmD F[\mu]$ on $\Tangent[\mu]$.
  \end{enumerate}
\end{theorem}
\begin{proof}
  If $\vV\in \Residual$ we can write $\vV=\lim_{N\to\infty}\vV^N$ in
  $L^2(\prbt\times\R^d,\bmm;\R^d)$
  where $\vV^N=\sum_{n=1}^N u_n\vV_n$ is the orthogonal projection of
  $\vV$ on the space generated by $\{\vV_1,\cdots,\vV_N\}$, with $u_n:= \la V, V_n \ra$. 
  Clearly $\vV^N[\mu]\in \Residual[\mu]$ for every $N\in \N$ and $\mu\in
  \prbt\setminus \cN$.
  Moreover we can find a subsequence, not relabeled, and a
  $\mm$-negligible set $\cN'\supset\cN$ such that 
  $\vV^N[\mu]\to \vV[\mu]$ in $L^2(\R^d,\mu; \R^d)$ for every $\mu\in
  \prbt\setminus \cN'$, so that $\vV[\mu]\in \Residual[\mu]$ for every $\mu\in
  \prbt\setminus \cN'$.

  Let now $\vV\in L^2(\prbt\times\R^d,\bmm;\R^d)$  be a  vector field such that
  $\vV[\mu]\in \Residual[\mu]$ for $\mm$-a.e.~$\mu\in \prbt$. Since $\Residual$
  is a closed subspace, in order to show that $\vV\in \Residual$ it is
  sufficient to prove that the scalar product with 
  every element $\wW\in \Residual^\perp$ vanishes.

  If $\wW\in \Residual^\perp$ then for every $H\in L^\infty(\prbt,\mm)$
  and every $n\in \N$ we get
  \begin{displaymath}
    \int_\prbt\Big(\int_{\R^d} \wW\cdot \vV_n\,\d\mu(x)\Big)H(\mu)\,\d\mm(\mu)=0,
  \end{displaymath}
  since $H\vV_n \in \Residual $ by \eqref{eq:70}. 
  Being $H$ arbitrary, we find that
  there exists a $\mm$-negligible set $\cN''\subset \prbt$ such that
  \begin{displaymath}
    \int_{\R^d} \wW[\mu] \cdot \vV_n[\mu]\,\d\mu=0\quad
    \text{for every }n\in \N,\ \mu\in \prbt\setminus\cN'',
  \end{displaymath}
  so that $\wW[\mu]\in \big(\Residual[\mu]\big)^\perp$ for 
  $\mm$-a.e.$\mu\in \prbt$.  
   We then deduce that
  \begin{displaymath}
    \int_{\R^d} \wW[\mu]\cdot \vV[\mu]\,\d\mu=0\quad
    \text{for $\mm$-a.e.~$\mu\in \prbt$},
  \end{displaymath}
  and therefore
  \begin{displaymath}
    \langle\wW,\vV\rangle_{L^2}=\int_\prbt\Big(\int_{\R^d} \wW\cdot \vV\,\d\mu(x)\Big)\,\d\mm(\mu)=0.
  \end{displaymath}
  The previous argument also shows that a vector field $\vV$ belongs
  to $\Residual^\perp$ if and only if $\vV[\mu]\in
  \big(\Residual[\mu])^\perp$
  for $\mm$-a.e.~$\mu\in \prbt$. This fact, together with the very
  definition of $\Tan(\prbt,\bmm)$ \eqref{eq:130}, yields claim (2).

  Claim (3) just follows by Theorem
  \ref{thm:decomposition}, since 
  \eqref{eq:75} shows that, for every $F\in H^{1,2}(\W_2)$,
  $\rmD_\mm F[\mu]\in \Tangent[\mu]$
  for $\mm$-a.e.~$\mu\in \prbt$.
  
  If $F\in \cyl1b \prbt$, combining claim 1 and Theorem \ref{thm:decomposition}, we see that  $\rmD F[\mu]-\rmD_\mm F[\mu]\in
  \Residual[\mu] \subset \big(\Tangent[\mu]\big)^\perp$ $\mm$-a.e., so that
  $\rmD_\mm F[\mu]$ is the $L^2(\R^d,\mu; \R^d)$-orthogonal projection of $\rmD
  F[\mu]$ on  $\Tangent[\mu]$, as stated in Claim (4).
\end{proof}
We can now interpret the above results in terms of the nonsmooth
tangent and cotangent structures introduced and developed by Gigli in
\cite{Gigli18}. Since we are in the Hilbertian case, we can identify
the cotangent module $L^2(T^*\prbt)$ and dual tangent module $L^2(T\prbt)$ with the Hilbert space $\Tangent$ defined by
\eqref{eq:132}. Let us report a useful characterization of the cotangent module $L^2(T^*X)$ \cite[Theorem 4.1.1]{GP20} for a general metric measure space $(X, \sfd, \mm)$.

\begin{theorem}\label{thm:module} Let $(X, \sfd, \mm)$ be a metric measure space. Then there exists a unique pair $((\mathcal{M}, \| \cdot \|_{\mathcal{M}}, \cdot_\mathcal{M}, |\cdot|_{\mathcal{M}}), \diff)$ such that $(\mathcal{M}, \| \cdot \|_{\mathcal{M}}, \cdot_\mathcal{M}, |\cdot|_{\mathcal{M}})$ is a $L^2(X,\mm)$-normed $L^\infty(X,\mm)$ module (cf.~\cite[Definition 3.1.1]{GP20}) and $\diff: D^{1,2}(X, \sfd, \mm) \to \mathcal{M}$ is a linear operator such that
\begin{itemize}
    \item [(i)] $|\diff(f)|_{\mathcal{M}} = \relgrad{f}$ $\mm$-a.e.~in $X$ for every $f \in D^{1,2}(X, \sfd, \mm)$.
    \item[(ii)] $\mathcal{M}$ is generated by $\left \{ \diff(f) : f \in D^{1,2}(X, \sfd, \mm) \right \}$.
\end{itemize}
Uniqueness is intended in the following sense: if $((\tilde{\mathcal{M}}, \| \cdot \|_{\tilde{\mathcal{M}}}, \cdot_{\tilde{\mathcal{M}}}, |\cdot|_{\tilde{\mathcal{M}}}), \tilde{\diff})$ is another pair with the above properties, then there exists a unique module isomorphism $\mathcal{J}: \mathcal{M} \to \tilde{\mathcal{M}}$ such that $\tilde{\diff}= \mathcal{J} \circ \diff$.
\end{theorem}
 We thus have the following result.
\begin{theorem} There exists a unique module isomorphism $\mathcal{I}: \Tangent \to L^2(T^*\prob_2(\R^d)) \cong L^2(T\prbt) $ such that $\mathcal{I} \circ \rmD_\mm$ coincides with the abstract differential operator taking values in $L^2(T^*\prob_2(\R^d))$ as in \cite[ Definition  2.2.2]{Gigli18}.
\end{theorem}
\begin{proof}
It is enough to show that $\Tangent$ (with an appropriate module structure) and the map $\rmD_\mm$ satisfy the properties listed in Theorem \ref{thm:module}. \\
If as $\|\cdot \|_\Tangent$ we take the $L^2(\prob_2(\R^d)\times \R^d, \bmm; \R^d)$ norm, it is clear that $(\Tangent, \|\cdot\|_\Tangent)$ is a Banach space, being closed by Lemma \ref{le:allTan}. The pointwise product $\cdot_\Tangent : L^\infty(\prob_2(\R^d), \mm) \times \Tangent \to \Tangent $ is well defined by \eqref{eq:131} and \eqref{eq:70}, bilinear and associative in $L^\infty(\prob_2(\R^d), \mm)$ by definition. Defining the pointwise norm $|\cdot|_\Tangent$ as the map sending $\vV \in \Tangent $ to $\|\vV[\mu]\|_\mu$, we immediately have that $\|\vV\|_\Tangent = \||\vV|_\Tangent\|_{L^2(\prob_2(\R^d), \mm)}$ and $|H \cdot_\Tangent \vV|_\Tangent = |H||\vV|_\Tangent$ $\mm$-a.e.~in $\prob_2(\R^d)$ for every $\vV \in \Tangent$ and every $H \in L^\infty(\prob_2(\R^d), \mm)$. This shows that $(\Tangent, \| \cdot \|_{\Tangent}, \cdot_\Tangent, |\cdot|_{\Tangent})$ is a $L^2(\prob_2(\R^d),\mm)$-normed $L^\infty(\prob_2(\R^d),\mm)$ module. Taking as $\diff$ the map $\rmD_\mm : D^{1,2}(\prob_2(\R^d), W_2, \mm) \to \Tangent$, we see that it is well defined and linear by Theorem \ref{prop:rel-diff} and Theorem \ref{thm:decomposition}. Property (i) of Theorem \ref{thm:module} follows by \eqref{eq:63}. Finally property (ii) of Theorem \ref{thm:module}, meaning that (\cite[Definition 3.1.13]{GP20}) $\Tangent$ coincides with the $\|\cdot\|_\Tangent$-closure of 
\[  \Tangent_0:=  \operatorname{span} \left \{ H \rmD_\mm F : H \in L^\infty(\prob_2(\R^d), \mm), \, F \in D^{1,2}(\prob_2(\R^d), W_2, \mm) \right \},\]
follows by \eqref{eq:38} and the definition of $\Tangent$.  Indeed, let $\mathscr{L} \subset L^\infty(\prob_2(\R^d), \mm)$ be a dense subset of $L^2(\prob_2(\R^d), \mm)$ and $\mathscr{C}$ be a dense subset of $\rmC_c^\infty(\R^d)$ with respect to the Lipschtiz norm as in Lemma \ref{le:allTan}. If $\vV \in \Tangent$, in particular $\vV \in \Tan(\prbt,\bmm)$ so that we can find by \eqref{eq:38} numbers $(N_n)_n \subset \N$, $(\{\alpha_n^i\}_{i=1}^{N_n})_n \subset \R$ and functions $(\{H_n^i\}_{i=1}^{N_n})_n \subset \mathscr{L}$, $(\{\zeta_n^i\}_{i=1}^{N_n})_n \subset \mathscr{C}$ such that the sequence
\[ \vV_n(\mu,x):= \sum_{i=1}^{N_n}\alpha_n^i H_n^i(\mu) \nabla \zeta_n^i(x), \quad (\mu, x) \in\prob_2(\R^d)\times \R^d\,  n \in \N\]
converges to $\vV$ in $L^2(\prob_2(\R^d)\times \R^d, \bmm;\R^d)$. Consider now the functions $F_n^i:= \lin {\zeta_n^i}$ and the sequence 
\[ \vV'_n(\mu,x):= \sum_{i=1}^{N_n}\alpha_n^i H_n^i(\mu) \rmD_\mm F_n^i, \quad (\mu, x) \in\prob_2(\R^d)\times \R^d\,  n \in \N.\]
It is clear that $(\vV'_n)_n \subset \Tangent_0$; by Theorem \ref{thm:pointwise}, $\vV'_n$ is the orthogonal projection of $\vV_n$ on $\Tangent$ for every $n \in \N$, so that $\vV'_n$ converges to $\vV$ in $L^2(\prob_2(\R^d)\times \R^d, \bmm;\R^d)$.

Theorem \ref{thm:module} gives thus the existence of a unique module isomorphism $\mathcal{I} : \Tangent \to L^2(T^*\prob_2(\R^d))$.\\
Finally, notice that $L^2(T^*\prob_2(\R^d)) \cong L^2(T\prbt)$ since $(\prob_2(\R^d), W_2, \mm)$ is infinitesimally Hilbertian by Corollary \ref{cor:WHilbert} (see also \cite[Theorem 4.3]{GP20}). 
\end{proof}

\subsection{Examples}
\label{subsec:examples}
\subsubsection*{Isometric embedding of Euclidean Sobolev spaces}
Let $\Omega$ be a Lipschitz bounded open set in $\R^d$.
For every $\omega\in \Omega$ let us consider the Dirac mass $\delta_\omega$
concentrated at $\omega$. The map $\iota:\omega\mapsto \delta_\omega$ is an isometry
between $\R^d$ and $\iota(\R^d)\subset \prob_2(\R^d)$.
Setting $\mm:=\iota_\sharp \Leb d\res\Omega$ we easily see that
$H^{1,2}(\prob_2(\R^d),W_2,\mm)$ is isomorphic to
$H^{1,2}(\Omega)$.

In this case only Dirac masses are involved and cylinder functions are of the form $F(\delta_\omega)=\psi( \pphi(\omega))$, so that
the Wasserstein gradient reduces to the usual gradient of
$\psi\circ \pphi$.

Another isometric embedding is also possible:
we fix a reference measure $\lambda \in \prob_2(\R^d)$ symmetric
w.r.t.~the origin
and we consider the map $\iota:\Omega\to \prob_2(\R^d)$ given by
\begin{equation}
  \label{eq:110}
  \iota(\omega):=\lambda(\cdot-\omega)=(\mathsf
  t_\omega)_\sharp\lambda,\quad
  \mathsf t_\omega(x):=x+\omega,\quad\omega\in \Omega.
\end{equation}
 To every function $F:\prob_2(\R^d)\to \R$ corresponds a map $\hat F: \Omega \to \R$ defined as 
\begin{equation}
  \label{eq:101}
  \hat F(\omega) :  =F((\mathsf
  t_\omega)_\sharp\lambda).
\end{equation}
In the case of a cylinder function as in \eqref{eq:cyl} we get
\begin{equation}
  \label{eq:106}
  \hat
  F(\omega)=\psi\Big(\int\phi_1(x+\omega)\,\d\lambda(x),\cdots,\int\phi_N(x+\omega)\,\d\lambda(x)\Big)=
  \psi\Big(\phi_1*\lambda(\omega),\cdots,\phi_N*\lambda(\omega)\Big).
\end{equation}
In this case (identifying $\iota(\omega)$ with $\omega$) we have
\begin{equation}
  \label{eq:109}
  \rmD
  F(\omega,x)=\sum_{j=1}^N\partial\psi_j(\phi_1*\lambda(\omega),\cdots,\phi_N*\lambda(\omega))
  \nabla\phi_j(x)
\end{equation}
and
\begin{equation}
  \label{eq:111}
  \|\rmD
  F[\omega]\|_{\omega}^2=
  \int_{\R^d}\Big|\sum_{j=1}^N\partial\psi_j(\phi_1*\lambda(\omega),\cdots,\phi_N*\lambda(\omega))
  \nabla\phi_j(x+\omega)\Big|^2\,\d\lambda(x).
\end{equation}
On the other hand, $\iota$ is an isometry of $\R^d$ into
$\prob_2(\R^d)$, so that
the space $H^{1,2}(\prob_2(\R^d),W_2,\mm)$ is still isomorphic to
$H^{1,2}(\Omega)$.
It follows that the $\mm$-Wasserstein gradient of $F$ is
\begin{equation}
  \label{eq:112}
  \rmD_\mm
  F(\omega,x)=\sum_{j=1}^N\partial\psi_j(\phi_1*\lambda(\omega),\cdots,\phi_N*\lambda(\omega))
  \nabla\phi_j*\lambda(\omega)
\end{equation}
independent of $x$ and the minimal relaxed gradient is
\begin{equation}
  \label{eq:113}
  |\rmD_\mm F|^2_\star(\omega)=
  \Big|\sum_{j=1}^N\partial\psi_j(\phi_1*\lambda(\omega),\cdots,\phi_N*\lambda(\omega))
  \nabla\phi_j*\lambda(\omega)\Big|^2
\end{equation}
\subsubsection*{Gaussian distributions}
Let now $\kappa=N(\omega,\Sigma):=(\det (2\pi \Sigma))^{-1/2}\mathrm
e^{-\frac12 \langle  \omega ,\Sigma^{-1} \omega \rangle}\Leb d$ be a Gaussian
measure with mean $\omega$ and covariance matrix $\Sigma\in
\mathrm{Sym}^+(d)$, the space of symmetric and positive definite
$d\times d$-matrices; we
consider the set
\begin{equation}
  \label{eq:114}
  \mathcal N^d:=\Big\{\mathcal N(\omega,\Sigma):
  \omega\in \R^d,\ \Sigma \in \mathrm{Sym}^+(d)\Big\},
\end{equation}
endowed with the Wasserstein distance and a finite positive Borel  measure
$\mm$ concentrated on $\mathcal N^d$.
Since
\begin{equation}
  \label{eq:115}
  W^2_2(N(\omega_1,\Sigma_1), N(\omega_2,\Sigma_2))=
  |\omega_1-\omega_2|^2+\mathrm{tr}\Sigma_1+\mathrm{tr}\Sigma_2-
  2\mathrm{tr}\Big(\Sigma_1^{1/2}\Sigma_2\Sigma_1^{1/2}\Big)^{1/2},
\end{equation}
 $H^{1,2}(\prob_2(\R^d),W_2,\mm)$ is isometric to 
$H^{1,2}( U  ,\sfd,\hat\mm)$
where $ U=\R^d\times \overline{\mathrm{Sym}^+(d)}\subset \R^d\times
\R^{d\times d}$
  endowed with the distance $\sfd$ induced by the formula
  \eqref{eq:115} and $\hat \mm$ is the measure induced by $\mm$.

\subsubsection*{The closable case} Following \cite{DelloSchiavo20} (here in the simpler setting of the Euclidean space, but see Section \ref{subsec:WSR} below), 
we assume that $\mm$ has no atoms and the following integration by
parts formula: for every $G \in \ccyl\infty c\prbt$
and $w \in \rmC^{\infty}_c(\R^d;\R^d)$ there exists $\rmD^*_w G \in
L^2(\prob_2(\R^d),\mm)$ such that for every $F \in \ccyl\infty c\prbt$ it holds
\[ \int_{\prob_2(\R^d)} \left (\int_{\R^d} \rmD F(\mu,x) \cdot w(x) \, \d \mu(x) \right ) G(\mu)\, \d \mm(\mu) = \int_{\prob_2(\R^d)}  \rmD^*_w G(\mu)  F(\mu) \, \d \mm(\mu). \]
This equality implies that $\Residual = \{ 0\}$ i.e.~that $\pCE_2$ is
closable. We notice that the measure $\mm$ induced by the immersion in
the space of delta measures considered at the beginning of this section
satisfies the integration by parts formula above (see also Example 5.4
in \cite{DelloSchiavo20}). In \cite{DelloSchiavo20}, in case the base
space is a compact Riemannian manifold, are reported important
examples of measures $\mm$ satisfying the (Riemannian analogue of the)
integration by parts formula: the normalized mixed Poisson measure
(Example 5.11 in \cite{DelloSchiavo20} and \cite{AKR98,RS99}), the
entropic measure over $\mathbb{S}^1$ (Example 5.15 in
\cite{DelloSchiavo20} and \cite{vRS09}, see also the multidimensional
case \cite{Sturm11}) and the Malliavin–Shavgulidze image measure (Example 5.18 in \cite{DelloSchiavo20} and \cite{MM91}).

\section{Extensions to Riemannian manifolds and Hilbert spaces}
\label{sec:extensions}
The aim of this Section is to extend the density result stated in
Theorem \ref{thm:main} from the finite dimensional and flat space $\R^d$  to Riemannian manifolds and (possibly infinite dimensional) Hilbert spaces.
Our first step deals with manifold embedded in some Euclidean space
$\R^d$ and in fact we will consider more general closed subsets of
$\R^d$.

\subsection{Intrinsic Wasserstein spaces on closed subsets of \texorpdfstring{$\R^d$}{R}}
\label{subsec:intrinsic}
 In this subsection we denote by $\varrho$ the Euclidean distance on $\R^d$. $\prob_2(\R^d)$ still denotes the subset of Borel probability measures on $\R^d$ with finite second $\varrho$-moment and $W_2$ is the Wasserstein distance on $\prob_2(\R^d)$ induced by $\varrho$.

We assume that $C \subset \R^d$ is a closed set and that $\sigma$ is a
distance on $C$ such that $(C,\sigma)$ is a complete and separable
metric space
and 
\begin{equation}\label{eq:disdist}
    \varrho(x_1,x_2) \le \sigma(x_1,x_2) \le \varrho_{C, \ell}(x_1,x_2) \quad \text{ for every } x_1,x_2 \in C,
\end{equation}
where $\varrho_{C, \ell} $ is defined as in \eqref{eq:40} with respect
to the distance $\sfd:=\varrho$.
 Since the topology induced by $\sigma$
is stronger than the Euclidean topology and they are both Polish
topologies,
the Borel sets of $(C,\sigma)$ coincide with the Borel sets of $C$ as
a subset of the Euclidean space $\R^d$. This means that every { Borel} 
probability measure on $\R^d$
with support contained in $C$ can be identified with a { Borel} 
probability measure in $(C, \sigma)$. Conversely any probability
measure on $(C,\sigma)$ extends to a probability measure on $\R^d$. We
can thus denote unambiguously by $\prob(C)$ the set of { Borel}
probability measures on $C$ and by $\prob_{2,\sigma}(C)$ the elements
of $\prob(C)$ with finite second $\sigma$-moment.

$\prob_{2,\sigma}(C)$
can be identified with the subset of $\prob_2(\R^d)$
\[ \left \{ \mu \in \prob_2(\R^d) : \supp(\mu) \subset C, \quad \int_C
    \sigma^2(x_0,x) \,\d \mu(x) < + \infty \text{ for some } x_0 \in C
  \right \}.\]
We will denote by $\iota:C\to\R^d$ the inclusion map; $\iiota:
\prob_{2,\sigma}(C)\to \prob_2(\R^d)$
is the corresponding continuous injection given by
$\iiota(\mu):=\iota_\sharp \mu$, which may be identified with the
inclusion map of $\prob_{2,\sigma}(C)$ into $\prob_{2,\sigma}(\R^d)$.

Since $(\prob_2(C),W_{2,\sigma})$ is a complete and separable
metric space and the topology induced by $W_{2,\sigma}$ is stronger
than the topology induced by $W_2$, we deduce that
$\prob_{2,\sigma}(C)$ is a Lusin (and therefore Borel) subset of
$\prob_2(\R^d)$.

If $\mm$ is a finite and positive Borel  measure on $\prob_{2,\sigma}(C)$,
$\iiota_\sharp \mm$ is
the Borel measure in $\prob_2(\R^d)$ which is
concentrated on $\prob_{2,\sigma}(C)$ and satisfies $ \iiota_\sharp \mm(Z)  =\mm(Z\cap
\prob_{2,\sigma}(C))$
for every Borel set $Z\subset \prob_2(\R^d).$

In a similar way, if $F:\prob_2(\R^d)\to \R$ is a Borel (or
$ \iiota_\sharp \mm$- measurable) map, we will
set $\iiota^* F:=F\circ\iiota:\prob_{2,\sigma}(C)\to \R$.
\begin{theorem}
  \label{thm:identification}
  We have $H^{1,2}(\prob_{2,\sigma}(C), W_{2, \sigma}, \mm)
   \cong H^{1,2}(\prob_2(\R^d), W_2, \iiota_\sharp\mm)
  $
  with equal minimal relaxed gradient, meaning that 
  \begin{equation}
\relgrad{(\iiota^* F)} = \iiota^* \left ( \relgrad{F} \right ) \text{
  for every } F \in H^{1,2}(\prob_2(\R^d), W_2,
\iiota_\sharp\mm).\label{eq:140}
\end{equation}
In particular
  $H^{1,2}(\prob_{2,\sigma}(C), W_{2, \sigma}, \mm)$ is a Hilbert space and the algebra of
  cylinder functions $\iiota^*\big(\cyl1b \prbt\big)$ is dense in
  $H^{1,2}(\prob_{2,\sigma}(C), W_{2, \sigma}, \mm)$ in the sense of \eqref{eq:55}.
\end{theorem}
\begin{proof}
  We want to apply Theorem \ref{thm:identification2}
  where $X:=\prob_2(\R^d)$,
  $\sfd:=W_{2}$, \
  $Y:=\prob_{2,\sigma}(C)$, and $\delta:=W_{2,\sigma}$.
  The first assumption of Condition (A), $\iiota_\sharp
  \mm(\prob_2(\R^d)\setminus \prob_{2,\sigma}(C))=0$, is clearly
  satisfied by construction.

  In order to prove \eqref{eq:139} we consider a $W_2$-Lipschitz curve
  $\mu:[0,\ell]\to \prob_2(\R^d)$ parametrized by the $W_2$-arc-length
  such that $ \mu_s  \in \prob_{2,\sigma}(C)$ for $\Leb 1$-a.e.~$s\in
  [0,\ell]$.
  Since the map $\mu$ is continuous in $\prob_2(\R^d)$, $C$ is a
  closed set, and $\mu_s(\R^d\setminus C)=0$ for
   $\Leb 1$-a.e.~$s\in
  [0,\ell]$, we conclude that $\mu_s(\R^d\setminus C)=0$ for every
  $s\in [0,\ell]$.
  
  By \cite[Theorem 8.2.1, Theorem 8.3.1]{AGS08})
  there exists a measure  $\eeta\in \prob(\rmC([0,\ell];\R^d))$  concentrated on absolutely continuous curves such that $(\mathsf e_t)_\sharp (\eeta)= \mu_t  $ for every $t\in [0,\ell]$ and
  \begin{equation}
    \label{eq:88}
    \int |\gamma'(t)|^2\,\d\eeta(\gamma) = \int
    |\dot{\gamma}|^2_{\varrho}(t)\,\d\eeta(\gamma) = 1 \quad\text{for
      a.e.~}t\in [0,\ell].
  \end{equation}
  Let us also consider the function $\zeta(x):=\mathrm{dist}(x,C)\land
  1$, $x\in \R^d$,
  where $\mathrm{dist}(x,C):=\min_{z\in C}\varrho(x,z)$. $\zeta$ is a bounded
  Lipschitz function which vanishes precisely on $C$.
  Fubini's Theorem yields
    \begin{align*}
      \int\Big(\int_0^\ell \zeta(\gamma(t))\,\d
      t\Big)\,\d\eeta(\gamma)
    &=
     \int_0^\ell \int \zeta(\sfe_t(\gamma))\,
      \d\eeta(\gamma)\,\d t
      =\int_0^\ell \int_{ \R^d} \zeta\,\d\mu_t\,\d t=0
  \end{align*}
  since $\int \zeta(x)\,\d\mu_t=0$ for $\Leb 1$-a.e.~$t\in
  (0,\ell)$.

  It follows that $\int_0^\ell \zeta(\gamma(t))\,\d t=0$ for $\eeta$-a.e.~$\gamma$, so
  that the set of $t\in [0,\ell]$ for which
  $\gamma(t)\in C$ is dense in $[0,\ell]$. Being $C$ closed,
  we conclude that $\gamma$ takes values in $C$ for
  $\eeta$-a.e.~$\gamma$.
  
  We can now estimate the $W_{2,\sigma}$ distance between  the two measures
  $\mu_{t_0}$ and $\mu_{t_1}$, where $0 \le t_0 < t_1 \le \ell$  :
  \begin{align*}
    W^2_{2,\sigma}(\mu_{t_0},\mu_{t_1})
    &
      \le
      \int \sigma^2(\gamma(t_0),\gamma(t_1))\,\d\eeta(\gamma)\\
       &\le  
      \int \Big(\int_{t_0}^{t_1} |\dot\gamma|_\sigma(s)\,\d s\Big)^2\,\d\eeta(\gamma)\\
      & = \int \Big(\int_{t_0}^{t_1} |\dot\gamma|_\varrho(s)\,\d s\Big)^2\,\d\eeta(\gamma)\\
     &\le  (t_1-t_0) \int \int_{t_0}^{t_1} |\dot\gamma|_\varrho^2\,\d s\,\d\eeta(\gamma)\\
      &=(t_1-t_0)\int_{t_0}^{t_1} \int |\dot\gamma|_{\varrho}^2\,\d\eeta(\gamma)\,\d s\\
       &= (t_1-t_0)^2,
  \end{align*}
  where we have used  that $(\mathsf e_{t_0},\mathsf e_{t_1})_\sharp \eeta \in \Gamma(\mu_{t_0}, \mu_{t_1})$,  \eqref{eq:disdist} and Remark \ref{rem:equivcond} to say that $|\dot{\gamma}|_\varrho (s) = |\dot{\gamma}|_\sigma (s)$.

  Choosing $t_0\in [0,\ell]$ such that $\mu_{t_0}\in
  \prob_{2,\sigma}(C)$ we deduce that $\mu_{t_1}\in
  \prob_{2,\sigma}(C)$ as well for every $t_1\in [0,\ell]$. This
  concludes the proof of property (A).

  Condition (B) corresponds to
 
    \begin{equation}\label{eq:aimintr}
      W_2(\mu_0, \mu_1) \le W_{2,\sigma}(\mu_0, \mu_1) \le (W_2)_{Y,
        \ell}(\mu_0, \mu_1) \quad \text{ for every } \mu_0, \mu_1 \in
      Y=
      \prob_{2,\sigma}(C),
  \end{equation}
  where $(W_2)_{Y, \ell}(\mu_0, \mu_1)$ is defined as in \eqref{eq:40}
  with $W_2$ in place of $\sfd$. The first inequality immediately
  follows by \eqref{eq:disdist}; to prove the second one,
  we use
  \eqref{eq:40bis} and 
  the above estimate with $t_0=0$ and $t_1=\ell$
  for  a  $W_2$-Lipschitz curve $\mu:[0, \ell] \to Y$ such that
  $|\dot{\mu}|_{W_2}=1$ a.e.~in $[0, \ell]$ with $ \mu |_{t=0}=\mu_0$ and
  $ \mu|_{t=\ell}=\mu_1$.
  Taking the infimum w.r.t.~$\ell$ 
  we obtain \eqref{eq:aimintr}.
\end{proof} 
\subsection{Wasserstein Sobolev space on complete Riemannian
  manifolds}
\label{subsec:WSR}
In this subsection, we will briefly discuss the case of the Sobolev
space $H^{1,2}(\prob_2(\M),
 W_{2,\sfd_\M},
\mm)$ where
$(\M,\sfd_\M)$ is a smooth and complete Riemannian manifold endowed with
the canonical Riemannian distance $\sfd_\M$   (inducing the Wasserstein distance $W_{2,\sfd_\M}$) and   $\mm$ is a  finite and positive Borel measure  on $\prob_2(\M)$.
We will denote by $\AA$ the unital  algebra generated by $\big\{\lin{f}: f\in \rmC^1_c(\M)\big\}$.
\begin{theorem}
  \label{thm:main2}
  $H^{1,2}(\prob_2(\M),W_{2, \sfd_\M},\mm)$ is a Hilbert space and the algebra $\AA$ is
  (strongly) dense:
  for every $F\in H^{1,2}(\prob_2(\M),W_{2, \sfd_\M},\mm)$
  there exists a sequence $F_n\in \AA$, $n\in \N$ such that 
  \begin{equation}
    \label{eq:55ee}
    F_n\to F,\quad \lip(F_n)\to \relgrad{F} \quad \text{strongly in }L^2(\prob_2(\M),\mm). 
  \end{equation}
\end{theorem}
\begin{proof}
  By Nash isometric embedding Theorem \cite{Nash54}
  we can find a dimension $d$,
  and an isometric embedding
  $\jmath:\M\to
  \jmath(\M)\subset \R^d$.
  On $M:=\jmath(\M)$ we can define
  the (Riemannian) metric $\sfd_M$ inherited by $\sfd_\M$:
  $\sfd_M(\jmath(x),\jmath(y))=\sfd_\M(x,y)$ so that $\jmath$ is an
  isometry and
  $(M, \sfd_M)$ is a complete and separable metric space.
  We denote by $\jjmath:=\jmath_\sharp$ the corresponding isometry between
  $(\prob_2(\M),W_{2,\sfd_\M})$ and $(\prob_2(M),W_{2,\sfd_M})$ and 
  we also set $\tilde\mm:=\jjmath_\sharp \mm $  which is a positive and finite Borel measure on $\prob_2(M)$.

  It is clear that the
  map $\jjmath^*:F\mapsto F\circ\jjmath$ induces a linear isometric isomorphism
  between 
  $H^{1,2}(\prob_2(M),W_{2,\sfd_M},\tilde\mm)$ and
  $H^{1,2}(\prob_2(\M),W_{2, \sfd_\M},\mm)$.
  
  Since $\M$ is complete and $\jmath$ is an embedding, $M$ is a closed
  subset of $\R^d$ and $\sfd_M$ induces on $M$ the relative topology
  of $\R^d$. Since $\jmath$ is isometric, we also have
  \begin{equation}
    \label{eq:89}
    \varrho(y_1,y_2) \le \sfd_M(y_1,y_2)=\varrho_{M, \ell}(y_1,y_2) \quad \text{ for every } y_1,y_2 \in M, 
  \end{equation}
 where $\varrho_{M,\ell}$ is as in \eqref{eq:40} and $\varrho$ denotes the Euclidean distance on $\R^d$.

 As in Section \ref{subsec:intrinsic},
 we can introduce the inclusion map $\iota:M\to \R^d$ and the
 corresponding
 $\iiota=\iota_\sharp:\prob_{2,\sfd_M}(M)\to \prob_2(\R^d)$.
 By Theorem \ref{thm:identification}
 we have that
 the map $\iiota^*:F\mapsto F\circ\iiota$ provides a linear isometric
 isomorphism
 between $H^{1,2}(\prob_2(\R^d),W_2,\iiota_\sharp \tilde\mm)$ and
 $H^{1,2}(\prob_{2,\sfd_M}(M),W_{2,\sfd_M},\tilde\mm)$ satisfying \eqref{eq:140};
 we conclude that the map
 $\kkappa^*:=\jjmath^*\circ \iiota^*=(\iiota\circ \jjmath)^*$
 is a linear isometric
 isomorphism
 between $H^{1,2}(\prob_2(\R^d),W_2,\kkappa_\sharp \mm)$
 (notice that $\kkappa_\sharp=\iiota_\sharp\circ \jjmath_\sharp$) and
 $H^{1,2}(\prob_{2,\sfd_\M}(M),W_{2,\sfd_\M},\mm)$
 satisfying
 \begin{equation}
   \relgrad{(\kkappa^* F)} = \kkappa^* \left ( \relgrad{F} \right ) \text{
     for every } F \in H^{1,2}(\prob_2(\R^d), W_2,
   \kkappa_\sharp\mm).\label{eq:140bis}
\end{equation}
This property in particular yields the  Hilbertianity of $H^{1,2}(\prob_2(\M),W_{2, \sfd_\M},\mm)$.

In order to prove that $\AA$ is dense in
$H^{1,2}(\prob_2(\M),W_{2, \sfd_\M},\mm)$ we consider
the algebra $\tilde\AA$ generated by $\big\{\lin{\tilde{f}}:
\tilde{f}\in \rmC^\infty_c(\R^d)\big\}$; Proposition
\ref{prop:density} shows that $\tilde\AA$ is strongly dense in
$H^{1,2}(\prob_2(\R^d),W_2,\tilde\mm)$, so that $\AA':=\kkappa^*(\tilde\AA)$
is strongly dense in $H^{1,2}(\prob_{2,\sfd_\M}(\M),W_{2,\sfd_\M},\mm)$.

$\AA'$ is generated by functions of the form $\kkappa^*\lin{\tilde{f}}$,
$\tilde{f}\in \rmC^\infty_c(\R^d)$. Since
\begin{displaymath}
  \kkappa^*\lin{\tilde{f}}(\mu)=\lin{\tilde{f}}(\kkappa(\mu))=
  \int_{\R^d}\tilde f(\kappa(x))\,\d\mu(x)\quad
  \text{for every $\mu\in \prob_{2,\sfd_\M}(\M)$},
\end{displaymath}
where $  \kappa=\iota\circ\jmath$,
we see that $\AA'$ is generated by functions of the form $\lin{\tilde f\circ\kappa}$, so
that $\AA'\subset \AA$ and a fortiori $\AA$ is strongly dense in
$H^{1,2}(\prob_{2,\sfd_\M}(\M),W_{2,\sfd_\M},\mm)$ as well.

To prove \eqref{eq:55ee} (involving the asymptotic Lipschitz constants
of functions in $\AA$ with respect to the Riemannian metric) we
observe that for every $\tilde F\in \tilde\AA$ 
\cite[Lemma 3.1.14]{Savare22}
\begin{equation}
  \kkappa^*\tilde F\in \AA'\subset \AA,\quad
\kkappa^*(\lip_{W_2} \tilde{F})  \ge \lip_{W_{2,
    \sfd_\M}}\kkappa^* \tilde F.\label{eq:141}
\end{equation}
Let now $F=\kkappa^* \tilde F\in H^{1,2}(\prob_2(\M),W_{2,
  \sfd_\M},\mm)$
with $\tilde F\in H^{1,2}(\prob_2(\R^d),W_2,\tilde\mm)$;
there exists a sequence $\tilde F_n\in \tilde{\AA}$ such that
  \begin{displaymath}
  \tilde{F}_n \to \tilde{F}, \quad \lip_{W_2}\tilde{F}_n \to
  \relgrad{\tilde{F}}
  \quad \text{ in } L^2(\prob_2(\R^d), \tilde{\mm}).
\end{displaymath}
Applying the linear isometric isomorphism $\kkappa^*$,
we deduce that the sequence $\kkappa^*
F_n\in \AA'$ satisfies
\begin{equation}
  \label{eq:48}
  \kkappa^*\tilde{F}_n \to F, \quad
  \kkappa^*\big(\lip_{W_2}\tilde{F}_n\big) \to
  \kkappa^* \big(  \relgrad{\tilde{F}}\big)=
  \relgrad{F}
  \quad \text{ in } L^2(\prob_{2,\sfd_\M}(\M), {\mm}).
\end{equation}
Up to extracting a suitable (not relabelled) subsequence and using \eqref{eq:141},
we can suppose that $\lip_{W_{2, \sfd_\M}}\kkappa^* \tilde F_n$
converges weakly in $L^2(\prob_2(\M), W_{2, \sfd_\M})$ to some $G \in
L^2(\prob_2(\M), W_{2, \sfd_\M})$ relaxed gradient of $F$.
\eqref{eq:141} and \eqref{eq:48} also yield
\begin{displaymath}
  \int G^2\,\d\mm\le
  \limsup_{n\to\infty} \int (\lip_{W_{2, \sfd_\M}}\kkappa^* F_n)^2 \, \d \mm
  \le \limsup_{n\to\infty} \int \Big(\kkappa^*(\lip_{W_2} \tilde{F}_n)\Big)^2 \,
  \d \mm=
  \int \relgrad{F}^2 \, \d \mm,
\end{displaymath}
showing that $G=\relgrad F$ and
$\lip_{W_{2, \sfd_\M}}\kkappa^* F_n\to \relgrad F$ strongly in $L^2(\prob_{2,\sfd_\M}(\M),\mm)$.
\end{proof}
\begin{remark}
  \label{rem:closability}
  Arguing as in Section \ref{subsec:cylindrical}
  it is immediate to see that the restriction of $\pCE_2$ to
  the algebra $\ccyl\infty c{\prob_2(\M)}$ is a quadratic form
  and coincides with the pre-Dirichlet forms considered in
  \cite{vRS09,Sturm11,DelloSchiavo20,DelloSchiavo22}.
  If $\big(\pCE_2, \ccyl\infty c{\prob_2(\M)})$ is closable then
  $\big(\CE_2,H^{1,2}(\prob_2(\M),W_{2,\sfd_\M},\mm)\big)$ is a
  Dirichlet form which coincides
  with the smallest closed extension of
  $\big(\pCE_2, \ccyl\infty c{\prob_2(\M)})$, it satisfies the
  so-called Rademacher property (see Proposition \ref{prop:calculus}
  and
  \cite{DelloSchiavo20}) and it is quasi-regular (see Remark
  \ref{rem:closable}).

  In particular it is possible to improve the 
  result \cite[Theorem 2.10]{DelloSchiavo20}. Referring to the
  notation and the formula enumeration of that paper, 
  one can immediately obtain that
  Lipschitz functions belong to $\mathcal F_0$ and the estimate 
  (2.16) holds just assuming that
  $\big(\pCE_2,\ccyl\infty c{\prob_2(\M)}\big)$ is closable.
\end{remark}
\subsection{Wasserstein Sobolev space on Hilbert spaces}
In this last section we will consider the case of a separable Hilbert
space
$(\H,|\cdot|)$; as usual, the space $\prob_2(\H)$ will be endowed with the
Wasserstein distance $\W_2$ induced by the Hilbertian norm of $\H$ and
we will assume that $\mm$ is a  finite and positive Borel  measure on
$\prob_2(\H)$.

We select a complete orthonormal system $E:=(e_n)_{n\in \N}$ and
the collection of maps 
$\pi_d:\H\to \R^d$, $d\in \N$, given by
\begin{equation}
  \label{eq:91}
  \pi^d(x):=\big(\langle x,e_1\rangle,\cdots,\langle x,e_d\rangle\big).
\end{equation}
The adjoint map $\pi^{d*}:\R^d\to\H$ is given by
\begin{equation}
  \label{eq:103}
  \pi^{d*}(y_1,\cdots,y_d):=\sum_{j=1}^d y_j\,e_j.
\end{equation}
The map $\hat\pi^d:=\pi^{d*}\circ\pi^d$ is the orthogonal projection
of $\H$ onto $\Span\{e_1,\cdots,e_d\}$.
We say that a function $\phi:\H\to\R$ belongs to $\rmC^1_b(\H,E)$ if it can be
written as
\begin{equation}
  \label{eq:95}
  \phi:=\varphi\circ \pi^d\quad\text{for some }d\in \N,\ \varphi\in \rmC^1_b(\R^d).
\end{equation}
 If $\phi \in \rmC^1_b(\H,E)$ then it belongs to $\rmC^1_b(\H)$  and its gradient $\nabla \phi$ can be written as
\begin{equation}
  \label{eq:96}
  \nabla\phi=\pi^{d*}\circ \nabla\varphi\circ\pi^d,\quad
  \nabla\phi(x)= 
  \sum_{j=1}^d \partial_j\varphi(\pi^d(x)) e_j.
\end{equation}
We then consider the algebra $\ccyl1b {\prob_2(\H)}$ generated by  $\left \{ \lin\phi : \phi \in \rmC^1_b(\H,E) \right \}$.  For every $F\in \ccyl1b {\prob_2(\H)}$ we can find $N\in \N$, a polynomial
$\psi:\R^N\to \R$ and functions $\phi_n  \in
\rmC^1_b(\H,E)$,  $n=1,\cdots,N$,   such that
\begin{equation}
  \label{eq:93}
  F(\mu) =(\psi \circ \lin \pphi)(\mu).
\end{equation}
As in \eqref{eq:Fdiff} we can set
\begin{equation}
  \label{eq:94}
   \rmD F(\mu,x):=\sum_{n=1}^N \partial_n\psi( \lin
  \pphi(\mu)
  )\nabla \phi_n(x).
\end{equation}
It is also easy to check that a function $F$ belongs to
$\ccyl1b{\prob_2(\H)}$
if and only if there exists $d\in \N$ and $\tilde F\in
\ccyl1b{\prob_2(\R^d)}$ such that
\begin{equation}
  \label{eq:102}
  F(\mu)=\tilde F(\pi^d_\sharp(\mu))\quad\text{for every }\mu\in \prob_2(\H),
\end{equation}
so that
\begin{equation}
  \label{eq:104}
  \rmD F(\mu,x)=\pi^{d*}\Big(\rmD\tilde
  F(\pi^d_\sharp\mu,\pi^d(x))\Big),\quad
   \|\rmD F[\mu]\|_\mu=\|\rmD \tilde F(\pi^d_\sharp\mu)\|_{\pi^d_\sharp\mu}.
\end{equation}
 By  Proposition \ref{prop:equality} and using \eqref{eq:104}
it is not difficult to check that
\begin{equation}
  \label{eq:97}
  \|\rmD F[\mu]\|_\mu=\lip F(\mu)\quad\text{for every }\mu\in \prob_2(\H).
\end{equation}
Adapting in an obvious way the definitions in \eqref{eq:53} and \eqref{eq:54} to
the Hilbertian framework, we have the following result.
\begin{theorem}
  \label{thm:main3}
  $H^{1,2}(\prob_2(\H),W_2,\mm)$ is a Hilbert space and the algebra $\ccyl1b {\prob_2(\H)}$ is
  (strongly) dense:
  for every $F\in H^{1,2}(\prob_2(\H),W_2,\mm)$
  there exists a sequence $F_n\in \ccyl1b {\prob_2(\H)}$, $n\in \N$ such that
  \begin{equation}
    \label{eq:55-ter}
     F_n\to F,\quad \lip(F_n)\to \relgrad{F} \quad \text{strongly in }L^2(\prob_2(\H),\mm).
  \end{equation}
\end{theorem}
\begin{proof}
  Let us set $\AA:=\ccyl1b {\prob_2(\H)}$; we use Theorem \ref{theo:startingpoint}
  and we want to prove that for every $\nu\in \prob_2(\H)$ the
  function
  \begin{equation}
    \label{eq:98}
    F(\mu):=W_2(\nu,\mu)\quad\text{satisfies}\quad
    |\rmD F|_{\star,\AA}\le 1\quad\text{$\mm$-a.e.}.
  \end{equation}
  We split the proof in two steps.

  Step 1: it is sufficient to prove
  that, for every $h \in \N$, the function $F_h:\prob_2(\H)\to\R$
  \begin{equation}
    \label{eq:99}
    F_h(\mu):=W_2(\hat\pi^h_\sharp\nu,\hat \pi^h_\sharp\mu)\quad
    \text{satisfies}\quad
    |\rmD F_h|_{\star,\AA}\le 1\quad\text{$\mm$-a.e.}
  \end{equation}
  In fact, using the continuity property of the Wasserstein distance, it is
  clear that for every $\mu\in \prob_2(\H)$
  \begin{equation}
    \label{eq:100} 
    \lim_{ h  \to\infty}F_h(\mu)=F(\mu), 
  \end{equation}
   so that it is enough to apply Theorem \ref{thm:omnibus}(1)-(3) to obtain \eqref{eq:98}. 
  
  Step 2: 
   Let $h \in \N$ be fixed and let us denote by $W_{2,h}$ the Wasserstein distance on $\prob_2(\R^h)$; it is easy to check that
  \[ W_{2,h}(\pi^h_\sharp \mu_0, \pi^h_\sharp \mu_1) = W_2(\hat \pi^h_\sharp \mu_0, \hat \pi^h_\sharp \mu_1) \quad \text{ for every } \mu_0, \mu_1 \in \prob_2(\H).\]
  Thus, if we define the function $\tilde F_h:\prob_2(\R^h)\to\R$
  as
  \[\tilde F_h(\mu):=W_{2,h}(\pi^h_\sharp\nu,\mu)\]
  we get that 
 \[ F_h(\mu)=\tilde F_h(\pi^h_\sharp \mu).\]  

  We also introduce the measure  $\mm_h$ on $\prob_2(\R^h)$ given by  the push-forward of $\mm$ through the ($1$-Lipschitz) map
  $P^h: \prob_2(\H)\to \prob_2(\R^h)$ defined as $P^h(\mu):=\pi^h_\sharp \mu$.
  By  Theorem \ref{thm:main} applied to $H^{1,2}(\prob_2(\R^h),W_{2,h},\mm_h)$, 
  we can find a sequence of cylinder functions $\tilde F_{h,n}\in
  \ccyl1b{\prob_2(\R^h)}$, $n\in \N$, 
  such that
  \begin{gather}
    \label{eq:105}
    \tilde F_{h,n}\to \tilde F_h\text{ in
      $\mm_h$-measure},\\
    \label{eq:107}
    \lip_{\prob_2(\R^h)} \tilde F_{h,n}\to g_h\quad\text{in
    }L^2(\prob_2(\R^h),\mm_h)\quad\text{with }g_h\le 1
    \text{ $\mm_h$-a.e.}    
  \end{gather}
  We thus consider the functions $F_{h,n}\in \ccyl1b{\prob_2(\H)}$
  defined as in \eqref{eq:102} by
  \begin{equation}
    \label{eq:108}
    F_{h,n}(\mu):=\tilde F_{h,n}(\pi^h_\sharp\mu)=\tilde
    F_{h,n}(P^h(\mu))\quad
    \text{for every }\mu\in \prob_2(\H).
  \end{equation}
  Since for every $\eps>0$
    \begin{align*}
    \mm\Big(\big\{\mu:|F_{h,n}(\mu)-F_h(\mu)|>\eps\big\}\Big)&=
    \mm\Big(\big\{\mu:|\tilde F_{h,n}(P^h(\mu))-\tilde
    F_h(P^h(\mu))|>\eps\big\}\Big)\\&=
    \mm_h\Big(\big\{\mu:|\tilde F_{h,n}(\mu)-\tilde F_h(\mu)|>\eps\big\}\Big),    
  \end{align*}
  \eqref{eq:105} yields that $F_{h,n}\to F_h$ in $\mm$-measure as
  $n\to\infty$.

  On the other hand, \eqref{eq:104} yields
  \begin{align*}
    \lip F_{h,n}(\mu)&=
    \lip_{\prob_2(\R^h)} \tilde F_{h,n}(P^h(\mu))
  \end{align*}
  so that
  \begin{displaymath}
    \lip F_{h,n}\to g_h\circ P^h\quad\text{in }L^2(\prob_2(\H),\mm)
  \end{displaymath}
  and $g_h\circ P^h\le 1$ $\mm$-a.e.~in $\prob_2(\H)$.
   By Theorem \ref{thm:omnibus}(1)-(3), we obtain \eqref{eq:99}, concluding the proof.
\end{proof}
\begin{remark} We observe that the results in Sections \ref{subsec:wsspace} and \ref{subsec:resdiff} can be extended to $\prob_2(\M)$ and $\prob_2(\H)$ in an analogous way.
\end{remark}

\bibliographystyle{plain}
\bibliography{biblio}
\end{document}